\documentclass[11pt,reqno]{amsart}

\usepackage[T1]{fontenc}
\usepackage{lmodern}
\usepackage{amsmath,amssymb,mathtools}
\usepackage[expansion=false]{microtype}
\usepackage{xcolor}
\usepackage{tikz}
\usepackage[margin=30mm]{geometry}
\usepackage[colorlinks=true,linkcolor=blue!55!black,
  citecolor=blue!55!black,urlcolor=blue!55!black]{hyperref}
\hypersetup{
  pdftitle={Failure of the H1 x H1 -> L(1/2,infinity) estimate for the lacunary bilinear spherical maximal function},
  pdfauthor={Sanghyuk Lee, Ji Li, Chaojie Wen, Liangchuan Wu},
  pdfkeywords={bilinear spherical averages, lacunary maximal functions, H1(Rd), endpoint estimates}
}

\usepackage{esint}

\allowdisplaybreaks

\newtheorem{theorem}{Theorem}[section]
\newtheorem{proposition}[theorem]{Proposition}
\newtheorem{lemma}[theorem]{Lemma}
\newtheorem{corollary}[theorem]{Corollary}
\newtheorem{definition}[theorem]{Definition}
\theoremstyle{remark}
\newtheorem{remark}[theorem]{Remark}
\numberwithin{equation}{section}

\newcommand{\R}{\mathbb R}

\newcommand{\A}{\mathcal A}
\newcommand{\B}{\mathcal B}
\newcommand{\Av}{\operatorname{Av}}
\newcommand{\one}{\mathbf 1}
\newcommand{\supp}{\operatorname{supp}}

\title[Failure of the $H^1\times H^1$ endpoint]
{Failure of the $H^1\times H^1\longrightarrow L^{1/2,\infty}$ estimate for the
lacunary bilinear spherical maximal function and applications}

\author{Sanghyuk Lee, Ji Li, Chaojie Wen, Liangchuan Wu}

\makeatletter
\@ifundefined{subjclassname@2020}{
  \@namedef{subjclassname@2020}{\textup{2020} Mathematics Subject Classification}
}{}
\makeatother
\subjclass[2020]{Primary 42B25; Secondary 42B30}
\keywords{Bilinear spherical averages, lacunary maximal functions,
the Hardy space $H^1(\mathbb R^d)$, endpoint estimates}

\begin{document}

\begin{abstract}
In contrast to Christ's theorem that the lacunary spherical maximal
operator is bounded from $H^1(\mathbb R^d)$ to $L^{1,\infty}(\mathbb R^d)$ ($d\geq2$),
we show that the lacunary bilinear spherical maximal operator
$$
 \mathcal B_{\mathrm{lac}}(f,g)(x)
 =\sup_{\ell\in\mathbb Z}
 \Bigl|\int_{\mathbb S^{2d-1}}f(x-2^\ell y)\,g(x-2^\ell z)\,d\sigma(y,z)\Bigr|
$$
is not bounded from $H^1(\mathbb R^d)\times H^1(\mathbb R^d)$ to
$L^{1/2,\infty}(\mathbb R^d)$ for any $d\geq1$. In particular, it is
not bounded from $L^1(\mathbb R^d)\times L^1(\mathbb R^d)$ to $L^{1/2,\infty}(\mathbb R^d)$.
More precisely, for every sufficiently large integer $N$,
we construct two finite linear combinations of smooth, compactly
supported, mean-zero $H^1(\mathbb R^d)$-atoms for which the ratio of
the weak $L^{1/2}$ quasinorm to the product of the two Hardy norms is
at least $c_d\log(N+1)$. This lower bound holds already for the
supremum over a suitable set of $N$ dyadic radii.
For $d=1$ this answers the case $p=q=1$ of a question in Borges's
survey.
For every $k\ge3$ and $d\ge1$, a modification of the construction proves
that the lacunary $k$-linear spherical maximal operator does
not map $(H^1(\R^d))^k$ into $L^{1/k,\infty}(\R^d)$. We also indicate extensions to certain hypersurfaces and to
horizontal spherical averages on the Heisenberg group.
\end{abstract}

\maketitle

\section{Introduction and statement of the result}

Spherical averages arise in the study of the wave and Darboux
equations.
Stein \cite{Stein1976} initiated the study of the full spherical
maximal function and proved its sharp $L^p$ bounds in dimensions
$d\geq3$.
Bourgain \cite{Bourgain1986} established the corresponding bounds
in dimension two.

Along with the full spherical maximal function, the lacunary
spherical maximal function, in which the supremum is taken only over
the dyadic radii $2^\ell$, $\ell\in\mathbb Z$, has traditionally been
an important object of study. Calder\'on \cite{Calderon1979} and
Coifman and Weiss \cite{CoifmanWeiss1978} proved that it is bounded on
$L^p(\mathbb R^d)$ for all $1<p\leq\infty$ when $d\geq2$, a range
larger than that of the full maximal function.
At the endpoint $p=1$, it remains open for $d\geq2$ whether the
linear lacunary spherical maximal operator is of weak type $(1,1)$,
see \cite{Christ1988,SeegerWright2011}.
Christ \cite{Christ1988} proved that this operator maps
$H^1(\mathbb R^d)$ into $L^{1,\infty}(\mathbb R^d)$.
Seeger, Tao and Wright \cite{STW2003} obtained an
$L\log\log L$ endpoint estimate for lacunary spherical means.
Seeger and Wright \cite{SeegerWright2011} proved
$H^1\to L^{1,\infty}$ bounds for more general lacunary maximal
operators under geometric and smoothing assumptions. Related work includes Fourier-decay criteria for maximal operators
\cite{RD}, improved spherical endpoint estimates \cite{CB},
$\delta$-discretised and discrete spherical averages \cite{CLLS,LLLS}, and counterexamples showing that polynomial Fourier decay
alone does not imply the linear Hardy endpoint
\cite{CenHongWu,XiaojingZhou}.
Our result shows that the bilinear Hardy endpoint fails
even for the standard spherical surface measure, for which
the corresponding linear Hardy estimate holds for $d\geq2$.

Let $d\geq1$, and let $d\sigma$ be normalized surface measure on
$\mathbb S^{2d-1}\subset\mathbb R^d\times\mathbb R^d$.  For
$R>0$, define the bilinear spherical average
\begin{equation*}
 \A_R(f,g)(x)
 :=\int_{\mathbb S^{2d-1}}f(x-Ry)g(x-Rz)\,d\sigma(y,z),\qquad x\in \mathbb R^d
\end{equation*}
and define the lacunary bilinear
spherical maximal operator by
\begin{equation*}
 \B_{\mathrm{lac}}(f,g)(x)
 :=\sup_{\ell\in\mathbb Z}|\A_{2^\ell}(f,g)(x)|.
\end{equation*}
The full bilinear spherical maximal operator was first studied by
Geba, Greenleaf, Iosevich, Palsson and Sawyer \cite{GGIPS}. Further
bounds were obtained in
\cite{BarrionuevoEtAl,GrafakosHeHonzik,HeoHongYang}.
Jeong and Lee \cite{JeongLee} obtained Lebesgue-space bounds in the
optimal open range when $d\geq2$, using a slicing formula that
expresses bilinear spherical averages in terms of linear spherical
averages on lower-dimensional spheres.
Dosidis \cite{Dosidis} treated the full multilinear operator, and
Choudhary, Shen and Shrivastava \cite{ChoudharyShenShrivastava}
studied fractal sets of dilations.

For the lacunary operator, Borges and Foster \cite{BorgesFoster}
proved $\B_{\mathrm{lac}}: L^p\times L^q\to L^r$ in the open H\"older
range $1<p,q\leq\infty$, $1/r=1/p+1/q$. For $d=1$, see Christ and
Zhou \cite{ChristZhou}. Since $\B_{\mathrm{lac}}$ commutes with dyadic
dilations, the relation $1/r=1/p+1/q$ is necessary, and $p=q=1$ is
the endpoint of this range. There the strong-type estimate fails
(\cite{Vergara} for $d\geq2$, \cite{ChristZhou} for $d=1$), so the
natural endpoint estimate is the weak-type bound
$L^1\times L^1\to L^{1/2,\infty}$, the bilinear analogue of the weak
type $(1,1)$ problem for the linear operator. For further background,
see the survey by Borges \cite{BorgesSurvey}.

The purpose of this paper is to determine whether this weak-type
endpoint estimate holds. More precisely, we address the following
question.

\textbf{Question.} Does the following estimate hold
\begin{equation}\label{eq:Q}
 \|\B_{\mathrm{lac}}(f,g)\|_{L^{1/2,\infty}(\mathbb R^d)}
 \lesssim
 \|f\|_{L^1(\mathbb R^d)}\|g\|_{L^1(\mathbb R^d)}?
\end{equation}
Here the implicit constant is independent of $f$ and $g$, and
$L^{p,\infty}(\mathbb R^d)$, $p>0$, denotes the weak $L^p$ space with
the quasinorm
\begin{equation}\label{eq:weak-quasinorm}
 \|F\|_{L^{p,\infty}(\mathbb R^d)}
 :=\sup_{\lambda>0}
 \lambda\,|\{x\in\mathbb R^d:|F(x)|>\lambda\}|^{1/p}.
\end{equation}
In view of Christ's $H^1\to L^{1,\infty}$ result \cite{Christ1988}, it is natural
to first consider the weaker estimate in which $L^1$ is replaced by
the Hardy space:
\begin{equation}\label{eq:proposed-endpoint}
 \|\B_{\mathrm{lac}}(f,g)\|_{L^{1/2,\infty}(\mathbb R^d)}
 \lesssim
 \|f\|_{H^1(\mathbb R^d)}\|g\|_{H^1(\mathbb R^d)},
\end{equation}
where $H^1(\mathbb R^d)$ is the real Hardy space \cite{FeffermanStein}.

For a set $E\subset\{2^\ell:\ell\in\mathbb Z\}$, let
$\B_E(f,g)=\sup_{R\in E}|\A_R(f,g)|$, so that $\B_E\leq\B_{\mathrm{lac}}$.
We prove that \eqref{eq:proposed-endpoint} fails. In fact, we obtain a
lower bound in terms of the number of radii.
\begin{theorem}\label{thm:main}
Let $d\geq1$. There is a constant $c_d>0$ with the following
property. For every sufficiently large integer $N$, there are a set
$E_N$ of $N$ dyadic radii and real functions
$f_N,g_N\in C_c^\infty(\mathbb R^d)$, each a finite linear
combination of mean-zero $H^1(\mathbb R^d)$-atoms, such that
\begin{equation}\label{eq:main-ratio}
 \frac{\|\B_{E_N}(f_N,g_N)\|_{L^{1/2,\infty}(\mathbb R^d)}}
 {\|f_N\|_{H^1(\mathbb R^d)}\|g_N\|_{H^1(\mathbb R^d)}}
 \geq c_d\log(N+1).
\end{equation}
Consequently, \eqref{eq:proposed-endpoint} is false.
\end{theorem}

The radii in $E_N$ form a lacunary sequence whose ratio depends on
$N$, see Subsection~\ref{sec:cube-tree}. We do not know whether
the factor $\log(N+1)$ is sharp.
Since $\|f\|_{L^1(\mathbb R^d)}
\lesssim\|f\|_{H^1(\mathbb R^d)}$,
\eqref{eq:Q} would imply \eqref{eq:proposed-endpoint}.
Thus \eqref{eq:Q} also fails.

For $d=1$, the sphere $\mathbb S^1\subset\mathbb R^2$ is the unit
circle, and
$$
 \A_R(f,g)(x)=\frac1{2\pi}\int_0^{2\pi}
 f(x-R\cos t)g(x-R\sin t)\,dt.
$$
By contrast, the linear spherical average on $\mathbb R$ uses the two
points of $\mathbb S^0$.
Theorem~\ref{thm:main} shows that the weak endpoint fails even when
both functions belong to the real Hardy space.
Since $L^{1,1}(\R)=L^1(\R)$, the failure of \eqref{eq:Q} for $d=1$
gives a negative answer at $p=q=1$ to the estimate asked about in
Question 3.2 (2) of \cite{BorgesSurvey}.

The method also applies to the multilinear spherical maximal
operators. For $k\geq3$, let $d\sigma$ be normalized surface measure on
$\mathbb S^{kd-1}\subset(\mathbb R^d)^k$, and define
\begin{equation*}
 \A_R^{(k)}(f_1,\ldots,f_k)(x)
 :=\int_{\mathbb S^{kd-1}}\prod_{i=1}^k f_i(x-Ry_i)
 \,d\sigma(y_1,\ldots,y_k),\qquad x\in\mathbb R^d,
\end{equation*}
and
\begin{equation*}
 \B_{\mathrm{lac}}^{(k)}(f_1,\ldots,f_k)(x)
 :=\sup_{\ell\in\mathbb Z}|\A_{2^\ell}^{(k)}(f_1,\ldots,f_k)(x)|.
\end{equation*}
For a set $E$ of dyadic radii, let
$\B_E^{(k)}=\sup_{R\in E}|\A_R^{(k)}|$.

\begin{theorem}\label{rem:multilinear-endpoint}
Let $k\geq3$ and $d\geq1$.
There is a constant $c_{d,k}>0$ with the following property.
For every sufficiently large integer $N$, there are a set $E_N$ of
$N$ dyadic radii and real functions
$f_{1,N},\ldots,f_{k,N}\in C_c^\infty(\mathbb R^d)$, each a finite
linear combination of mean-zero $H^1(\mathbb R^d)$-atoms, such that
$$
 \frac{\|\B_{E_N}^{(k)}(f_{1,N},\ldots,f_{k,N})
                  \|_{L^{1/k,\infty}(\mathbb R^d)}}
      {\prod_{i=1}^k\|f_{i,N}\|_{H^1(\mathbb R^d)}}
 \geq c_{d,k}\log(N+1).
$$
Consequently, $\B_{\mathrm{lac}}^{(k)}$ does not map
$(H^1(\mathbb R^d))^k$ into $L^{1/k,\infty}(\mathbb R^d)$.
\end{theorem}

\textbf{Strategy of the proof.}
We briefly describe the idea of the proof. A more detailed outline
is given in Subsection~\ref{sec:construction-idea}.
The construction rests on the observation that the bilinear average
of two mean-zero atoms need not cancel. Let $F$ and $G$ be
$L^1$-normalized atoms of scale $s$ with a fixed function  $\phi$,
centered at $u$ and $v$, where $|u-v|$ is small compared with $R$.
Then 
$$ \A_R(F,G)\geq c(sR^{2d-1})^{-1}$$ on part of a shell of thickness
comparable to $s$ about the sphere $\{x:|x-u|^2+|x-v|^2=R^2\}$, a set
of measure comparable to $sR^{d-1}$. After the coarea formula, this
reduces to the positivity of
$\int_{\mathbb R}\widehat\phi(ty_1)\widehat\phi(ty_2)\,dt$, which
holds because we choose $\phi$ with $\widehat\phi\geq0$
(Lemma~\ref{lem:J-positive} and Proposition~\ref{cor:positive-region}).

We use this observation at the $N$ radii $R_1,\ldots,R_N$ with a
single pair of functions, each a sum of many atoms with total
coefficient $\theta$, so that its Hardy norm is $O(\theta)$.
Labels on a family of cubes select the atom pairs used at each radius.
For a pair of cubes whose first common ancestor is at level $j$,
the normalized count of labels selecting the pair is comparable to
$j^{-1}$ (Lemma~\ref{lem:local-limit}). Summing over $j$ produces the
factor $\log(N+1)$. We control the overlap of the resulting regions
by estimating the integral of their counting function and its square.
Finite sign averages give a choice of signs for which cancellation
between different atom pairs does not prevent the required lower bound.

The construction also gives logarithmic lower bounds for finite
geometric progressions of dyadic radii and counterexamples for dyadic
variation. We give these and other consequences in
Section~\ref{sec:applications}. There we also indicate extensions to
certain hypersurfaces and to horizontal spherical averages on the
Heisenberg group.

\textbf{Organization.}
Sections~\ref{sec:setup-and-outline}--\ref{sec:proof-completion}
prove Theorem~\ref{thm:main} for $d\geq2$.
Section~\ref{sec:setup-and-outline} introduces the cube family, the
dyadic radii and the labels, and outlines the proof.
Section~\ref{sec:construction} proves the positivity lemma for the
atom (Lemma~\ref{lem:J-positive}), constructs the atoms, and
bounds the Hardy norms. Section~\ref{sec:atom-pairs} selects the atom
pairs and proves the local lower bounds
(Proposition~\ref{cor:positive-region}).
Section~\ref{sec:region-counting} proves the estimate
$c_j\approx j^{-1}$ for the normalized label counts
(Lemma~\ref{lem:local-limit}), and Section~\ref{sec:region-moments}
estimates the integral of the counting function and its square
(Proposition~\ref{prop:moments}). Section~\ref{sec:proof-completion}
chooses the signs and completes the proof.
Section~\ref{sec:dimension-one} describes the changes for $d=1$, and
Section~\ref{sec:multilinear} proves Theorem~\ref{rem:multilinear-endpoint}.
Section~\ref{sec:applications} gives further consequences of the construction.

\section{Notation and outline of the proof}
\label{sec:setup-and-outline}

Throughout Sections~\ref{sec:setup-and-outline}--\ref{sec:proof-completion},
we assume $d\geq2$.
Fix a large integer $N$ and let
\begin{equation}\label{eq:theta-choice}
 \theta^2=\frac{\rho}{\log(N+1)},
\end{equation}
where $\rho>0$ is a small constant fixed in
Subsection~\ref{sec:labels-coefficients}.
Let $E_N=\{R_1,\ldots,R_N\}$ be the set of $N$ dyadic radii defined
in \eqref{eq:dyadic-radii} below. We construct two finite sums of
atoms such that
\begin{gather*}
 \|f\|_{H^1(\mathbb R^d)},\|g\|_{H^1(\mathbb R^d)}\lesssim\theta
 \quad{\rm and}\quad
 |\{x\in\mathbb R^d:\B_{E_N}(f,g)(x)>\lambda_0\}|\geq c_U/36.
\end{gather*}
Here $\lambda_0,c_U>0$ and the small constant $\rho>0$ are independent
of $N$. Since
$\|F\|_{L^{1/2,\infty}}\geq\lambda_0|\{|F|>\lambda_0\}|^2$ by
\eqref{eq:weak-quasinorm}, these estimates give \eqref{eq:main-ratio}.

\subsection{The cube family, dyadic radii, labels, and weights}
\label{sec:cube-tree}
Choose a dyadic integer $\mathcal L$ satisfying
$$
 \mathcal L\geq\max\left\{2^7\sqrt d,\;N^{1/d}\right\}.
$$
\begin{definition}
Starting from $[0,1)^d$, divide each cube into ${\mathcal L^d}$ half-open
subcubes of equal side length.  Repeat this subdivision $N$ times.
For $0\leq j\leq N$, let $\mathcal Q_j$ be the collection of all
cubes of side
\begin{equation}\label{eq:dyadic-radii}
 R_j=\mathcal L^{j-N}
\end{equation}
produced in this process. These are the level-$j$ cubes.  In particular, $\mathcal Q_0$
contains ${\mathcal L^{dN}}$ cubes of side $R_0$, and
$\mathcal Q_N=\{[0,1)^d\}$.  Every $R_j$ is a dyadic radius.
\end{definition}

For $Q_1\in\mathcal Q_0$, let $Q_1^{(r)}$ be the unique cube in
$\mathcal Q_r$ containing $Q_1$.  Thus
\begin{equation}\label{eq:cube-ancestors}
 Q_1^{(0)}=Q_1,\qquad Q_1^{(N)}=[0,1)^d.
\end{equation}

Let
$$
 \mathcal D=\bigcup_{h=0}^{N-1}\mathcal Q_h
$$
be the set of all these cubes except $[0,1)^d$, so that
$\#\mathcal D=\sum_{h=1}^{N}{\mathcal L^{dh}}$.
A choice of labels is a map $\omega:\mathcal D\to\{0,1\}$.
Each cube $P\in\mathcal D$ receives one label $\omega(P)$.
Under the uniform average below, the labels of distinct cubes are
independent, even when one cube contains the other.
For $H:\{0,1\}^{\mathcal D}\to\mathbb R$, define
$$
 \Av_\omega H=2^{-\#\mathcal D}\sum_{\omega\in\{0,1\}^{\mathcal D}}H(\omega).
$$
For functions depending on disjoint sets of cube labels, finite
summation gives
\begin{equation}\label{eq:label-factorization}
 \Av_\omega (H_1H_2)=(\Av_\omega H_1)(\Av_\omega H_2).
\end{equation}
The same identity holds after the other labels are fixed.

Recall the ancestors $Q_1^{(r)}$ from
\eqref{eq:cube-ancestors}.  Put
$$
 L_r=\lfloor\sqrt r\rfloor,\qquad
 W=\sum_{r=1}^N L_r.
$$
For a choice of labels $\omega$ and $Q_1\in\mathcal Q_0$, define
$$
 Z_{Q_1}(\omega)=\sum_{r=1}^N L_r\,\omega(Q_1^{(r-1)}),
$$
the weighted sum of the labels of the ancestors of $Q_1$.
We write $Z_{Q_1}$ for $Z_{Q_1}(\omega)$ when $\omega$ is fixed.
In particular, $0\leq Z_{Q_1}\leq W$.

Let $Q_1,Q_2\in\mathcal Q_0$ be distinct. Since
$Q_1^{(0)}\ne Q_2^{(0)}$ and $Q_1^{(N)}=Q_2^{(N)}$, there is a unique
$j\in\{1,\ldots,N\}$ such that
$$
 Q_1^{(j)}=Q_2^{(j)},\qquad
 Q_1^{(j-1)}\ne Q_2^{(j-1)}.
$$
We call $Q=Q_1^{(j)}$ the first common ancestor of $Q_1,Q_2$, and
$j$ its level.
Their containing cubes agree at levels $j,\ldots,N$.
The common labels therefore cancel, giving
\begin{equation}\label{eq:label-difference}
 Z_{Q_1}-Z_{Q_2}
 =\sum_{r=1}^j L_r
   \bigl(\omega(Q_1^{(r-1)})-\omega(Q_2^{(r-1)})\bigr).
\end{equation}
The $2j$ cubes in this difference are distinct.

We use the radius $R_j$ for such cube pairs.
The construction uses the $N$ dyadic radii $R_1,\ldots,R_N$, which
form the set $E_N$. It is a subsequence of the dyadic radii with ratio
$\mathcal L$, which may depend on $N$. All implicit constants in the
bilinear construction depend only on $d$ and the fixed atom,
not on $N$, $\mathcal L$, or the labels and base points.

For each $Q_1\in\mathcal Q_0$, let $\frac18Q_1$ denote the cube
with the same center as $Q_1$ and side length $R_0/8$.
For each $P\in\mathcal Q_0$, choose base points
$a_P,b_P\in {\frac18P}$. Write
$a=(a_{Q_1})_{Q_1\in\mathcal Q_0}$ and
$b=(b_{Q_1})_{Q_1\in\mathcal Q_0}$, and let
$$
 {\Omega}=\prod_{Q_1\in\mathcal Q_0}\bigg({\frac18Q_1}\times {\frac18Q_1}\bigg),
 \qquad
 da\,db=\prod_{Q_1\in\mathcal Q_0}da_{Q_1}\,db_{Q_1}.
$$
With the labels fixed, for $H:\Omega\to\mathbb R$ define
$$
 \Av_{a,b} H
 =\fint_{\Omega}H(a,b)\,da\,db
 =\frac1{|{\Omega}|}\int_{{\Omega}}H(a,b)\,da\,db.
$$
Here $|\Omega|=\prod_{Q_1\in\mathcal Q_0}|{\frac18Q_1}|^2
=(8^{-d}R_0^d)^{2\#\mathcal Q_0}$.

\subsection{Outline of the proof of Theorem~\ref{thm:main}}
\label{sec:construction-idea}

The proof has four steps, shown in Figure~\ref{fig:proof-outline}.

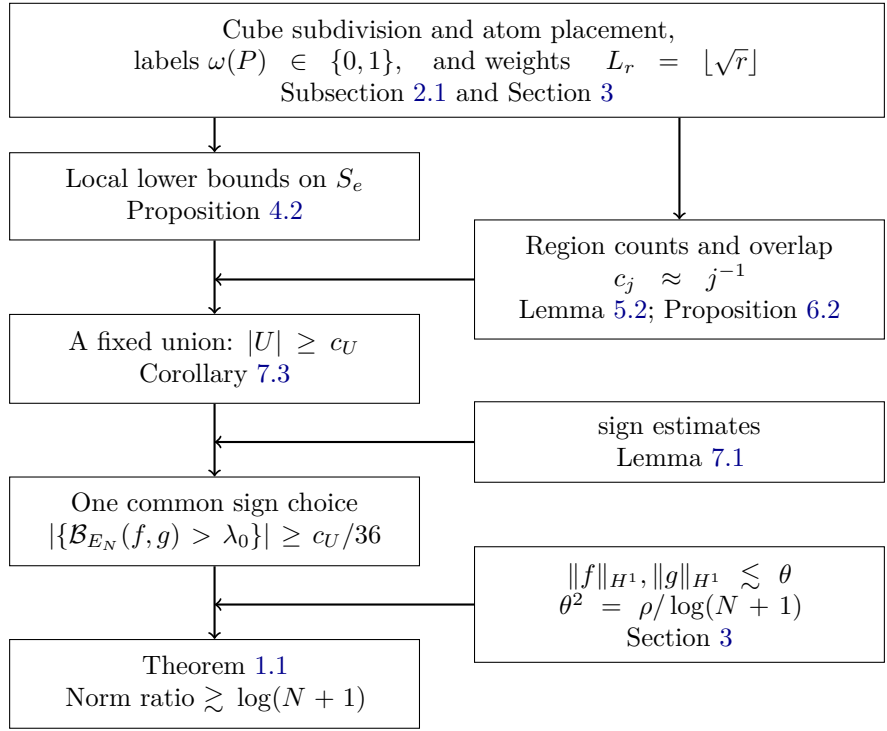
\begin{figure}[t]
\centering
\begin{tikzpicture}[x=1cm,y=1cm,font=\small,
 every node/.style={align=center},
 main/.style={draw,rectangle,text width=5.0cm,inner sep=6pt},
 side/.style={draw,rectangle,text width=5.0cm,inner sep=6pt}]
 \node[main,text width=11.15cm] (construction) at (3.075,1.8)
 {Cube subdivision and atom placement,\\
 labels $\omega(P)\in\{0,1\},\quad {\rm and\ weights\ }\quad L_r=\lfloor\sqrt r\rfloor$\\
 Subsection~\ref{sec:cube-tree} and Section~\ref{sec:construction}};
 \node[main] (local) at (0,0)
 {Local lower bounds on $S_e$\\
 Proposition~\ref{cor:positive-region}};
 \node[main] (union) at (0,-2.15)
 {A fixed union: $|U|\geq c_U$\\
 Corollary~\ref{cor:union}};
 \node[main] (signs) at (0,-4.3)
 {One common sign choice\\
 $|\{\B_{E_N}(f,g)>\lambda_0\}|\geq c_U/36$};
 \node[main] (result) at (0,-6.45)
 {Theorem~\ref{thm:main}\\
 Norm ratio $\gtrsim\log(N+1)$};

 \node[side] (counts) at (6.15,-1.1)
 {Region counts and overlap\\
 $c_j\approx j^{-1}$\\
 Lemma~\ref{lem:local-limit}; Proposition~\ref{prop:moments}};
 \node[side] (moments) at (6.15,-3.25)
 {sign estimates\\
 Lemma~\ref{lem:sign-estimates}};
 \node[side] (norms) at (6.15,-5.4)
 {$\|f\|_{H^1},\|g\|_{H^1}\lesssim\theta$\\
 $\theta^2=\rho/\log(N+1)$\\
 Section~\ref{sec:construction}};

 \draw[->,thick] (construction.south -| local.north) -- (local.north);
 \draw[->,thick] (construction.south -| counts.north) -- (counts.north);
 \draw[->,thick] (local) -- (union);
 \draw[->,thick] (union) -- (signs);
 \draw[->,thick] (signs) -- (result);
 \draw[->,thick] (counts.west) -- (0,-1.1);
 \draw[->,thick] (moments.west) -- (0,-3.25);
 \draw[->,thick] (norms.west) -- (0,-5.4);
\end{tikzpicture}
\caption{Proof of Theorem~\ref{thm:main}. The side estimates are
used at the indicated steps.}
\label{fig:proof-outline}
\end{figure}

\smallskip
\noindent\emph{1. Atoms and Hardy bounds.}
In each $Q_1\in\mathcal Q_0$, choose $m_{Q_1}$ functions
$F_{Q_1,r}$ and $n_{Q_1}$ functions $G_{Q_1,t}$, each a
constant multiple of an atom supported in a cube of side
comparable to $s$, the small scale fixed in \eqref{eq:atom-size}.
Their coefficients satisfy
$$
 m_{Q_1}\mu_{Q_1}=n_{Q_1}\nu_{Q_1}=\theta R_0^d.
$$
Let $f,g$ be their signed sums, as in \eqref{eq:signed-functions}.
Each coefficient sum is $\theta$, so
\eqref{eq:H1-bound} gives the required Hardy bounds for every
choice of labels, base points, and signs.

\smallskip
\noindent\emph{2. Selected pairs and local lower bounds.}
For a level-$j$ cube $Q$, choose the two separated auxiliary cubes
$Q^L,Q^R$ from Section~\ref{sec:placement}.
We select those ordered pairs $Q_1,Q_2\in\mathcal Q_0$ for which
$Q_1\subset Q^L$, $Q_2\subset Q^R$, and $Z_{Q_1}-Z_{Q_2}=j$.
Let $\mathcal E_j$ denote the selected atom pairs
$e=((Q_1,r),(Q_2,t))$ at radius $R_j$.
The choices in Section~\ref{sec:construction} give
$\mu_{Q_1}\nu_{Q_2}=sR_j^{2d-1}$ for every selected pair.
Subsection~\ref{sec:local-lower-bound} attaches to $e$ a region
$S_e$ with $|S_e|\approx sR_j^{d-1}$, and
Proposition~\ref{cor:positive-region} gives
$$
 \A_{R_j}(F_{Q_1,r},G_{Q_2,t})(x)\geq c_0>0
 \qquad(x\in S_e).
$$
Lemma~\ref{lem:separation} makes these regions disjoint
when $(Q_1,Q_2)$ is fixed and $(r,t)$ and $(r',t')$  differ.

\smallskip
\noindent\emph{3. A union of measure bounded below.}
Let $c_j$ be the proportion of label choices that select a
fixed admissible cube pair at level $j$; see \eqref{eq:cj-def}.
By Lemma~\ref{lem:local-limit}, the choice
$L_r=\lfloor\sqrt r\rfloor$ gives $c_j\approx j^{-1}$.
We obtain
$
 \theta^2\sum_{j=1}^N c_j
 \approx\theta^2\log(N+1)=\rho.
$

For the regions above, write
$$
 \mathcal N_j(x)=\sum_{e\in\mathcal E_j}\one_{S_e}(x),
 \qquad \mathcal N(x)=\sum_{j=1}^N\mathcal N_j(x),
 \qquad U=\{x:\mathcal N(x)>0\}.
$$
Proposition~\ref{prop:moments}
states
$$
 \Av_\omega\Av_{a,b}\int\mathcal N\approx\theta^2\sum_{j=1}^Nc_j
 \approx{\rho},\qquad
 \Av_\omega\Av_{a,b}\int\mathcal N^2\lesssim{\rho}+{\rho}^2.
$$
Cauchy--Schwarz on the joint space of labels, base points, and $x$
therefore gives
$$
 \Av_\omega\Av_{a,b}|U|\geq
 \frac{(\Av_\omega\Av_{a,b}\int\mathcal N)^2}{\Av_\omega\Av_{a,b}\int\mathcal N^2}
 \gtrsim\frac{{\rho}^2}{{\rho}+{\rho}^2}>0.
$$
Thus one fixed choice of labels and base points has $|U|\geq c_U>0$.

\smallskip
\noindent\emph{4. One common choice of signs.}
Other atom pairs may cancel the positive contribution of the
selected pair. Fix an ordering of the regions. For each $x\in U$,
choose the first region containing $x$ and let $j(x)$ be its level.
Make this choice before choosing signs.

Let $\Av_{\varepsilon}\Av_{\delta}$ be normalized summation over all
independent signs on the atoms of $f,g$.
By Lemma~\ref{lem:sign-selection}, with $\lambda_0=c_0/\sqrt2$,
$$
 \Av_{\varepsilon}\Av_{\delta}
 \one_{\{|\A_{R_{j(x)}}(f,g)(x)|>\lambda_0\}}
 \geq\frac1{36}\qquad(x\in U).
$$
Integrating in $x$ gives one common sign choice for which
$$
 |\{x\in \mathbb R^d:\B_{E_N}(f,g)(x)>\lambda_0\}|
 \geq |U|/36\geq c_U/36.
$$
The Hardy bounds hold for this choice as well. Hence
\eqref{eq:main-ratio} holds, as explained at the beginning of this
section.
The logarithm comes from $\sum_jc_j$, while each point uses only
one of the $N$ dyadic radii. The following sections establish these four steps.

\section{Atoms in the level-0 cubes}
\label{sec:construction}

The purpose of this section is to construct the functions $f$ and $g$
of Theorem~\ref{thm:main} and to bound their Hardy norms by $C\theta$.
Each function is a finite sum of atoms, all translates and dilates of one
fixed function $\phi$ placed in the level-$0$ cubes. We first choose
$\phi$ with $\widehat\phi\geq0$ and prove the positivity lemma on which
the lower bounds of Section~\ref{sec:atom-pairs} rest. We then fix
the small constants, the number of atoms in each cube (prescribed by
the labels), and their coefficients.

Fix the real radial function
$$
 \chi(x)=
 \begin{cases}
 \displaystyle\exp\left(-\frac1{1-|x|^2}\right),
     & |x|<1,\\[4pt]
 0, & |x|\geq1.
 \end{cases}
$$
Then $\chi\in C_c^\infty(\mathbb R^d)$,
$\supp\chi=\overline{B(0,1)}$, and $\int\chi>0$.
Let
$$
 \widetilde\chi(x)=\chi(-x),\qquad
 \psi=\chi*\widetilde\chi,\qquad
 \phi=-\Delta\psi.
$$
Then $\phi\in C_c^\infty(\mathbb R^d)$ is real and radial,
and $\int\phi=0$. We use the Fourier transform convention
$$
 \widehat h(\xi)=\int h(x)e^{-2\pi ix\cdot\xi}\,dx.
$$
Then
\begin{equation}\label{eq:positive-fourier}
 \widehat\phi(\xi)
 =4\pi^2|\xi|^2|\widehat\chi(\xi)|^2\geq0.
\end{equation}
Moreover, $\supp\phi\subset\supp\psi\subset\overline{B(0,2)}$.

The constants $h_0,\epsilon_0,c_\ast$ of Lemma~\ref{lem:J-positive}
below depend only on $d$, $\phi$, and the integer $k\geq2$ in that
lemma, and not on the placement of the atoms. In the bilinear construction we fix them with $k=2$. We then
choose $\rho$, and finally $\theta$ and $s$. Thus all fixed small
constants are chosen independently of $N$.

\subsection{A positivity lemma}

After the coarea formula, the averages of pairs of atoms in
Section~\ref{sec:atom-pairs}, and of $k$-tuples of atoms in
Section~\ref{sec:multilinear}, reduce to the following integral.
For an integer $k\geq2$, $y=(y_1,\ldots,y_k)\in(\mathbb R^d)^k$,
$h\in\mathbb R$ and $\epsilon\geq0$, put
\begin{equation}\label{eq:coarea-integral-definition}
 J^{(k)}_\epsilon(y,h)
 :=\int_{(\mathbb R^d)^k}\prod_{i=1}^k\phi(\alpha_i)\,
 \delta\Bigl(h-\sum_{i=1}^ky_i\cdot\alpha_i
 +\frac{\epsilon}{2}\sum_{i=1}^k|\alpha_i|^2\Bigr)
 \,d\alpha_1\cdots d\alpha_k.
\end{equation}
Here $\delta$ is the one-dimensional Dirac distribution,
and the integral is interpreted by the coarea formula.
We write $J_\epsilon(y_1,y_2,h)=J^{(2)}_\epsilon((y_1,y_2),h)$, and put
$$
 K_k=\Bigl\{y\in(\mathbb R^d)^k:
 \frac1{4\sqrt k}\leq|y_i|\leq2,\ 1\leq i\leq k\Bigr\}.
$$

\begin{lemma}\label{lem:J-positive}
Let $k\geq2$. There are constants $0<h_0\leq1/2$,
$0<\epsilon_0\leq1/(32\sqrt{dk})$ and $c_\ast>0$, depending only on
$d$, $k$ and $\phi$, such that
$$
 J^{(k)}_\epsilon(y,h)\geq c_\ast
$$
whenever $y\in K_k$, $|h|\leq h_0$, and $0\leq\epsilon\leq\epsilon_0$.
\end{lemma}

\begin{proof}
Fix $k\geq2$. All constants in this proof may depend on $d$, $k$ and
$\phi$. We prove positivity at $\epsilon=h=0$ and then bound the
perturbations in $h$ and $\epsilon$.

To begin with, note that Fourier inversion in the delta variable gives
$$
 J^{(k)}_0(y,0)
 =\int_{\mathbb R}\prod_{i=1}^k\widehat\phi(ty_i)\,dt.
$$
Since $\widehat\chi(0)\neq0$ and $|y_i|\leq2$, there are
fixed $t_0,c_\chi>0$ such that
$$
 |\widehat\chi(ty_i)|\geq c_\chi
 \qquad
 (|t|\leq t_0,\ y\in K_k,\ 1\leq i\leq k).
$$
By \eqref{eq:positive-fourier}, every factor is nonnegative. Using
also $|y_i|\geq1/(4\sqrt k)$ for $|t|\leq t_0$, we obtain
$$
 J^{(k)}_0(y,0)
 \geq
 \Bigl(\frac{\pi^2c_\chi^2}{4k}\Bigr)^k
 \int_{-t_0}^{t_0}t^{2k}\,dt
 =:c_{\chi,t_0}.
$$

The inequality $|e^{2\pi ith}-1|\leq2\pi|th|$ gives
\begin{align}\label{eq:J 0 h -small}
 |J^{(k)}_0(y,h)-J^{(k)}_0(y,0)|
 &\leq2\pi|h|\int_{\mathbb R}|t|
     \prod_{i=1}^k|\widehat\phi(ty_i)|\,dt
 \leq C|h|.
\end{align}
The last integral is uniformly bounded on $K_k$, since
$\phi\in C_c^\infty$ and $|y_i|\geq1/(4\sqrt k)$ imply
$|\widehat\phi(ty_i)|\leq C(1+|t|)^{-2}$, and $k\geq2$.

To estimate $J^{(k)}_\epsilon(y,h)$ for $\epsilon\geq0$,
Fourier inversion, initially in the distributional sense, gives
$$
 J^{(k)}_\epsilon(y,h)
 =\int_{\mathbb R}e^{2\pi ith}
 \prod_{i=1}^kI_{\epsilon,y_i}(t)\,dt,
$$
where
$$
 I_{\epsilon,y}(t)
 :=\int\phi(\alpha)e^{-2\pi ity\cdot\alpha}
 e^{\pi it\epsilon|\alpha|^2}\,d\alpha=\int\phi(\alpha)e^{2\pi itP_{\epsilon,y}(\alpha)}\,d\alpha
$$
with $P_{\epsilon,y}(\alpha)= -y\cdot \alpha+\frac{\epsilon}{2}|\alpha|^2$.

For $n=0,1$ and every integer $M\geq1$, we claim that
\begin{equation}\label{claim I epsilon n}
 |\partial_\epsilon^n I_{\epsilon,y}(t)|
 \leq C_{M}(1+|t|)^{n-M}
 \qquad\bigl(t\in\R,\ 0\leq\epsilon\leq\tfrac1{32\sqrt{dk}},\
 \tfrac1{4\sqrt k}\leq|y|\leq2\bigr).
\end{equation}
These bounds also justify the absolute convergence of the Fourier
integral for $J^{(k)}_\epsilon$.

First suppose $|t|\geq1$.
On the fixed neighborhood $\{|\alpha|<4\sqrt d\}$
of $\supp\phi$,
$$
|\nabla_\alpha P_{\epsilon,y}(\alpha)|= |-{y}+\epsilon\alpha|
 \geq |y|-\epsilon|\alpha|
 \geq\frac1{4\sqrt k}-4\epsilon\sqrt d
 \geq\frac1{8\sqrt k}.
$$
Consider the differential operator
$$
 \mathcal D_{\epsilon,y,t}
 =\frac{ \nabla_\alpha P_{\epsilon,y}(\alpha) \cdot\nabla_\alpha}
 {2\pi it |\nabla_\alpha P_{\epsilon,y}(\alpha)|^2}.
$$
By differentiation,
$
 \mathcal D_{\epsilon,y,t}
 e^{2\pi it{P_{\epsilon,y}(\alpha)}}
 =
 e^{2\pi it{P_{\epsilon,y}(\alpha)}}
$. Its formal transpose, for integration without complex conjugation, is
$$  \mathcal D_{\epsilon,y,t}^{\mathrm t} (\phi)=-\operatorname{div}_\alpha \bigg( \frac{\phi\cdot \nabla_\alpha P_{\epsilon,y}(\alpha)}{2\pi i t |\nabla_\alpha P_{\epsilon,y}(\alpha)|^2} \bigg).
 $$
On the support of the amplitudes, the coefficients of this
operator and their derivatives are uniformly bounded after removing
the factor $t^{-1}$: the phase gradient is bounded below and
$(\epsilon,y)$ ranges over a compact set. Integration by parts $M$
times, with no boundary terms because the amplitudes are compactly
supported, therefore gives
\begin{align*}
| I_{\epsilon,y}(t)|
 &=\bigg|\int (\mathcal D_{\epsilon,y,t}^{\mathrm t})^M \big(\phi(\alpha)\big)\, e^{2\pi itP_{\epsilon,y}(\alpha)}\,d\alpha\bigg| \leq \frac{C_{M}}{|t|^M}\leq \frac{2^MC_{M}}{(1+|t|)^M}.
\end{align*}
Similarly,
\begin{align*}
| \partial_\epsilon I_{\epsilon,y}(t)|
 &=\bigg|\int \pi i t \ (\mathcal D_{\epsilon,y,t}^{\mathrm t})^M \big( |\alpha|^2 \phi(\alpha)\big)e^{2\pi itP_{\epsilon,y}(\alpha)}\,d\alpha\bigg| \leq \frac{C_{M}}{|t|^{M-1}}\leq \frac{2^{M}C_{M}}{(1+|t|)^{M-1}}.
\end{align*}

For $|t|\leq1$, direct integration gives
$$
 |I_{\epsilon,y}(t)|\leq\|\phi\|_1,\qquad
 |\partial_\epsilon I_{\epsilon,y}(t)|
 \leq\pi|t|\int|\alpha|^2|\phi(\alpha)|\,d\alpha.
$$
Together with the estimates for $|t|\geq1$, these bounds prove
\eqref{claim I epsilon n}, with a constant depending on $M$.

Take $M\geq2$ and differentiate the product of the $k$ factors.
Then \eqref{claim I epsilon n} gives
$$
 \Bigl|
 \partial_\epsilon
 \prod_{i=1}^kI_{\epsilon,y_i}(t)
 \Bigr|
 \leq C_{M}(1+|t|)^{1-kM}.
$$
Integrating this derivative in $\epsilon$ and then in $t$ gives
$$
 |J^{(k)}_\epsilon(y,h)-J^{(k)}_0(y,h)|
 \leq C\epsilon\int_{\mathbb R}(1+|t|)^{1-kM}\,dt
 \leq C\epsilon,
$$
uniformly for $y\in K_k$, $h\in\mathbb R$ and
$0\leq\epsilon\leq1/(32\sqrt{dk})$.
Together with \eqref{eq:J 0 h -small}, after taking a common
constant $C$, we have
$$
 |J^{(k)}_\epsilon(y,h)-J^{(k)}_0(y,0)|
 \leq C(\epsilon+|h|).
$$

Choose $0<h_0\leq1/2$ and
$
 0<\epsilon_0\leq
 \frac1{32\sqrt{dk}}
$
so that $C(h_0+\epsilon_0)\leq \frac{c_{\chi,t_0}}{2}$, and put
$c_\ast=c_{\chi,t_0}/2$.
For $y\in K_k$, $|h|\leq h_0$, and
$0\leq\epsilon\leq\epsilon_0$, we obtain
$$
 J^{(k)}_\epsilon(y,h)\geq J^{(k)}_0(y,0) -  |J^{(k)}_\epsilon(y,h)-J^{(k)}_0(y,0)|
 \geq c_{\chi,t_0}-C(\epsilon+|h|)
 \geq c_\ast.
$$
This proves the lemma.
\end{proof}

\subsection{The number and centers of the atoms}
\label{sec:labels-coefficients}

Let $h_0$, $\epsilon_0$ and $c_\ast$ be the constants of
Lemma~\ref{lem:J-positive} with $k=2$. In particular,
$\epsilon_0\leq1/(32\sqrt{2d})$.
Fix
$$
 0<{\rho}\leq(\log2)
 \min\left\{{2^{-10}},\epsilon_0\right\}.
$$
Since $\theta^2=\rho/\log(N+1)\leq\rho/\log2$ by
\eqref{eq:theta-choice},
\begin{equation}\label{eq:theta-small}
 \theta^2\leq\min\{2^{-10},\epsilon_0\}
\end{equation}
for every $N\geq1$.
We use \eqref{eq:theta-small} in the form $2^6\theta^2\leq2^{-4}$ in
Subsection~\ref{sec:atoms-functions}, and in the form
$\theta^2\leq\epsilon_0\leq1/8$ in Section~\ref{sec:atom-pairs}.

For each $Q_1\in\mathcal Q_0$, put
\begin{equation}\label{eq:atom-numbers}
{m_{Q_1}=\mathcal L^{(2d-1)(W-Z_{Q_1})}},\qquad
{n_{Q_1}=\mathcal L^{(2d-1)(W+Z_{Q_1})}}.
\end{equation}
These are positive integers.  They are the numbers of atoms used
in $f$ and $g$, respectively, inside $Q_1$. Through
$Z_{Q_1}$ they depend on the labels, as do the coefficients,
centers, and blocks defined below. We keep this dependence implicit.
Their product does not depend on the labels:
$$
 m_{Q_1}n_{Q_1}={\mathcal L^{2(2d-1)W}}.
$$
Let
\begin{equation}\label{eq:atom-size}
 s=\theta^2R_0\,{\mathcal L^{-2(2d-1)W}}.
\end{equation}
It satisfies
\begin{equation}\label{eq:fit-identity}
 m_{Q_1}n_{Q_1}s=\theta^2R_0,
 \qquad s\leq\theta^2R_0.
\end{equation}

For $Q_1\in\mathcal Q_0$, use the base points chosen in
Section~\ref{sec:cube-tree} and define the centers by
$$
\begin{aligned}
 u_{Q_1,r}&=a_{Q_1}+{2^6}n_{Q_1}rs\,e_d,
 &&0\leq r<m_{Q_1},\\
 v_{Q_1,t}&=b_{Q_1}+{2^4}ts\,e_d,
 &&0\leq t<n_{Q_1}.
\end{aligned}
$$
The indices $r,t$ number the atoms inside $Q_1$.
The first group has spacing ${2^6}n_{Q_1}s$ and the second has
spacing ${2^4}s$.  Only the last coordinate changes.
By \eqref{eq:fit-identity}, their maximum displacements satisfy
\begin{equation}\label{eq:centers-diameter}
\begin{aligned}
 \max_r|u_{Q_1,r}-a_{Q_1}|
 &\leq {2^6}m_{Q_1}n_{Q_1}s
 ={2^6}\theta^2R_0,\\
 \max_t|v_{Q_1,t}-b_{Q_1}|
 &\leq {2^4}n_{Q_1}s\leq {2^4}\theta^2R_0.
\end{aligned}
\end{equation}
Thus both groups remain close to their base points, even when their
numbers are large.  The diameter of either group satisfies the same
bound, since its centers lie on one line in increasing order.
Figure~\ref{fig:atom-arrangement} shows the arrangement of the two
groups inside $Q_1$.

\begin{figure}[t]
\centering
\begin{tikzpicture}[x=1cm,y=1cm,font=\scriptsize]
  \draw[black!70] (0,0) rectangle (4,4);
  \node[anchor=north west,font=\small] at (.12,3.92) {$Q_1$};
  \filldraw[fill=black!4,draw=black!45,dashed]
    (1.75,1.75) rectangle (2.25,2.25);
  \draw[blue!65!black,very thick] (1.91,1.90)--(1.91,2.10);
  \draw[orange!85!black,very thick] (2.10,1.86)--(2.10,1.91);
  \fill[blue!65!black] (1.91,1.90) circle (1.4pt);
  \fill[orange!85!black] (2.10,1.86) circle (1.4pt);
  \draw[blue!65!black] (.72,2.75)--(1.91,1.90);
  \node[anchor=east,text=blue!65!black] at (.73,2.75)
    {$a_{Q_1}$};
  \draw[orange!85!black] (3.25,1.10)--(2.10,1.86);
  \node[anchor=west,text=orange!85!black] at (3.22,1.10)
    {$b_{Q_1}$};
  \draw[black!45] (.70,.83)--(1.75,1.75);
  \node[anchor=east] at (.75,.79) {$\frac18Q_1$};
  \draw[<->,black!55] (0,-.30)--(4,-.30)
    node[midway,fill=white] {$R_0$};
  \draw[->,black!55] (3.55,2.30)--(3.55,3.40)
    node[above] {$e_d$};

    \node[font=\small,text=blue!65!black] at (7.35,4.35)
    {Atoms in $f$};
  \draw[blue!65!black] (6.30,.55)--(6.30,1.85);
  \draw[blue!65!black] (6.30,3.05)--(6.30,3.45);
  \foreach \yy in {.55,1.85,3.45}{
    \filldraw[fill=blue!10,draw=blue!65!black]
      (6.19,\yy-.11) rectangle (6.41,\yy+.11);
    \fill[blue!65!black] (6.30,\yy) circle (1.2pt);
  }
  \node[text=blue!65!black,anchor=west] at (6.64,.55)
    {$u_{Q_1,0}=a_{Q_1}$};
  \node[text=blue!65!black,anchor=west] at (6.64,1.85)
    {$u_{Q_1,1}$};
  \node[text=blue!65!black,anchor=west] at (6.64,3.45)
    {$u_{Q_1,m_{Q_1}-1}$};
  \node[text=blue!65!black,font=\large] at (6.30,2.46)
    {$\vdots$};
  \draw[<->,blue!65!black] (5.85,.55)--(5.85,1.85)
    node[midway,left,rotate=90,anchor=south]
      {$2^6n_{Q_1}s$};
  \draw[->,black!55] (8.95,.30)--(8.95,3.90)
    node[above] {$e_d$};

    \node[font=\small,text=orange!85!black] at (12.25,4.35)
    {Atoms in $g$};
  \draw[orange!85!black] (11.20,.55)--(11.20,1.85);
  \draw[orange!85!black] (11.20,3.05)--(11.20,3.45);
  \foreach \yy in {.55,1.85,3.45}{
    \filldraw[fill=orange!12,draw=orange!85!black]
      (11.09,\yy-.11) rectangle (11.31,\yy+.11);
    \fill[orange!85!black] (11.20,\yy) circle (1.2pt);
  }
  \node[text=orange!85!black,anchor=west] at (11.54,.55)
    {$v_{Q_1,0}=b_{Q_1}$};
  \node[text=orange!85!black,anchor=west] at (11.54,1.85)
    {$v_{Q_1,1}$};
  \node[text=orange!85!black,anchor=west] at (11.54,3.45)
    {$v_{Q_1,n_{Q_1}-1}$};
  \node[text=orange!85!black,font=\large] at (11.20,2.46)
    {$\vdots$};
  \draw[<->,orange!85!black] (10.75,.55)--(10.75,1.85)
    node[midway,left,rotate=90,anchor=south] {$2^4s$};
  \draw[->,black!55] (13.85,.30)--(13.85,3.90)
    node[above] {$e_d$};

  \node[text=black!60] at (9.90,-.30)
    {The two columns are enlarged separately.};
\end{tikzpicture}
\caption{Atom centers in a level-0 cube $Q_1$.
The base points lie in $\frac18Q_1$, and both columns run in the
$e_d$ direction. The small squares contain the atom supports defined in
Section~\ref{sec:atoms-functions}.
Only the first and last few atoms are shown. Every support lies in
$Q_1$. The whole column need not lie in $\frac18Q_1$.
The drawing shows families with more than one atom.}
\label{fig:atom-arrangement}
\end{figure}
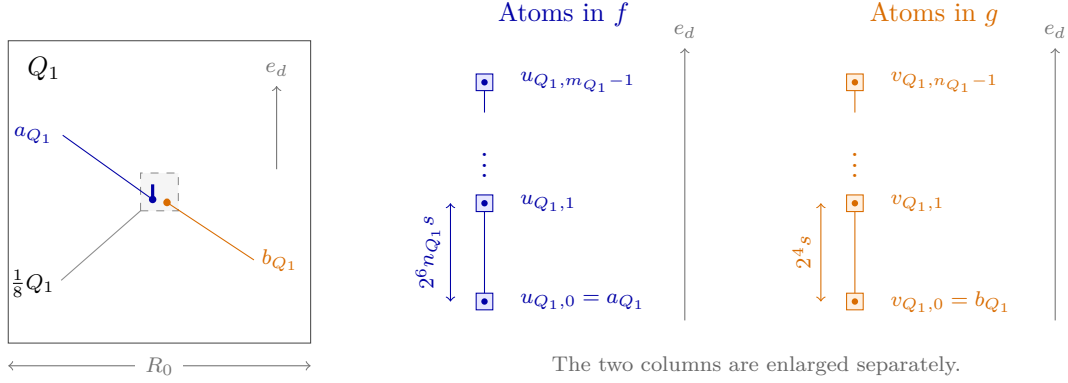

\subsection{The atom blocks, their supports, and their Hardy norms}
\label{sec:atoms-functions}

Give each level-0 cube a total coefficient $\theta R_0^d$ in each
function.  Define
\begin{equation}\label{eq:atom-coefficients}
 \mu_{Q_1}=\frac{\theta R_0^d}{m_{Q_1}},\qquad
 \nu_{Q_1}=\frac{\theta R_0^d}{n_{Q_1}}.
\end{equation}
Define the blocks by
\begin{equation}\label{eq:atom-blocks}
\begin{aligned}
 F_{Q_1,r}(x)&=\mu_{Q_1}s^{-d}
       \phi\left(\frac{x-u_{Q_1,r}}s\right),
 &&0\leq r<m_{Q_1},\\
 G_{Q_1,t}(x)&=\nu_{Q_1}s^{-d}
       \phi\left(\frac{x-v_{Q_1,t}}s\right),
 &&0\leq t<n_{Q_1}.
\end{aligned}
\end{equation}
Both blocks have integral zero. Put $Q_\phi=[-2,2]^d$, so that
$\supp\phi\subset\overline{B(0,2)}\subset Q_\phi$. Then
$$
 \supp F_{Q_1,r}\subset u_{Q_1,r}+sQ_\phi,\qquad
 \supp G_{Q_1,t}\subset v_{Q_1,t}+sQ_\phi.
$$
These support cubes have side $4s$. Their placement is shown in
Figure~\ref{fig:atom-arrangement}.

By \eqref{eq:theta-small}, $2^6\theta^2\leq2^{-4}$ and
$2\theta^2\leq2^{-4}$.
Let $c_{Q_1}$ be the center of $Q_1$.
For $x\in\supp F_{Q_1,r}$ and each coordinate $i$,
$$
 |x_i-(c_{Q_1})_i|
 \leq |(a_{Q_1})_i-(c_{Q_1})_i|
      +|(u_{Q_1,r}-a_{Q_1})_i|+|x_i-(u_{Q_1,r})_i|
 \leq 3R_0/16<R_0/2.
$$
The three terms are at most $R_0/16$ each, by
$a_{Q_1}\in {\frac18Q_1}$, \eqref{eq:centers-diameter}, and
$|x_i-(u_{Q_1,r})_i|\leq2s\leq2\theta^2R_0$.
The same calculation, with $b_{Q_1},v_{Q_1,t}$, applies to
$G_{Q_1,t}$.  Therefore, $$\supp F_{Q_1,r}, \ \ \supp G_{Q_1,t} \subset Q_1$$
and these inclusions hold for every base point $a_{Q_1}$, $b_{Q_1}$ in ${\frac18Q_1}$.

Choose a fixed $C_\phi>0$ such that
$C_\phi\geq |Q_\phi|\|\phi\|_\infty$.  Then both 
$$
 \frac{F_{Q_1,r}}{C_\phi\mu_{Q_1}}\qquad{\rm and}\qquad
 \frac{G_{Q_1,t}}{C_\phi\nu_{Q_1}}
$$
are standard $(1,\infty)$-atoms: each has integral zero, is supported
in its cube of side $4s$, and 
$$
 \left\|\frac{F_{Q_1,r}}{C_\phi\mu_{Q_1}}\right\|_\infty
 =\frac{s^{-d}\|\phi\|_\infty}{C_\phi}
 \leq\frac1{s^d|Q_\phi|}
 =\frac1{|u_{Q_1,r}+sQ_\phi|}
$$
and the same inequality holds for $ \frac{G_{Q_1,t}}{C_\phi\nu_{Q_1}}$. Thus, for the atomic norm,
$$
 \|F_{Q_1,r}\|_{H^1(\R^d)}\leq C_\phi\mu_{Q_1},\qquad
 \|G_{Q_1,t}\|_{H^1(\R^d)}\leq C_\phi\nu_{Q_1}.
$$
The atomic norm is equivalent to the usual real Hardy space norm;
see \cite{Coifman,Latter,SteinHA}.  The equivalence constants depend
only on $d$.

\subsection{The functions and their Hardy norms}

For any signs $\varepsilon_{Q_1,r},\delta_{Q_1,t}\in\{-1,1\}$, let
\begin{equation}\label{eq:signed-functions}
 f=\sum_{Q_1}\sum_{r=0}^{m_{Q_1}-1}
       \varepsilon_{Q_1,r}F_{Q_1,r},\qquad
 g=\sum_{Q_1}\sum_{t=0}^{n_{Q_1}-1}
       \delta_{Q_1,t}G_{Q_1,t}.
\end{equation}
Each block belongs to $C_c^\infty(\mathbb R^d)$ and the sums are
finite. Hence $f,g\in C_c^\infty(\mathbb R^d)$,
are supported in $[0,1)^d$, and have integral zero.
By \eqref{eq:atom-coefficients}, for every $Q_1\in\mathcal Q_0$,
$$
 m_{Q_1}\mu_{Q_1}=n_{Q_1}\nu_{Q_1}=\theta R_0^d.
$$
Since $\#\mathcal Q_0=R_0^{-d}$,
$$
 \sum_{Q_1}m_{Q_1}\mu_{Q_1}=\theta,\qquad
 \sum_{Q_1}n_{Q_1}\nu_{Q_1}=\theta.
$$
The triangle inequality gives
\begin{equation}\label{eq:H1-bound}
\begin{aligned}
 \|f\|_{H^1(\mathbb R^d)}
 &\leq C\sum_{Q_1}\sum_{r=0}^{m_{Q_1}-1}\mu_{Q_1}=C\theta,\\
 \|g\|_{H^1(\mathbb R^d)}
 &\leq C\sum_{Q_1}\sum_{t=0}^{n_{Q_1}-1}\nu_{Q_1}=C\theta.
\end{aligned}
\end{equation}
Here $C=C_dC_\phi$, where $C_d$ is the constant in the equivalence of
the atomic and real Hardy norms. Thus $C$ depends only on $d$ and the
fixed function $\phi$.
In particular, the later choices of labels, base points, and signs
preserve \eqref{eq:H1-bound}.

\section{Atom pairs and local lower bounds}
\label{sec:atom-pairs}

This section selects the atom pairs matched to each radius $R_j$ and
proves the local lower bounds. For a level-$j$ cube $Q$, the labels
select ordered pairs of level-$0$ cubes inside $Q$. The selection
rule is designed so that the product of the coefficients of a
selected atom pair is exactly $sR_j^{2d-1}$, which cancels the
normalization of the coarea formula. This is what makes the lower
bound in Proposition~\ref{cor:positive-region} independent of the
pair. We then attach to each selected pair a region $S_e$, part of a
thin spherical shell, on which the bilinear average of the pair is
bounded below, and show that the regions of different atom pairs of
the same cube pair are disjoint.

\subsection{Cube pairs}
\label{sec:placement}

Put $\tau=1/(2\sqrt d)$.
Fix $1\leq j\leq N$ and a level-$j$ cube $Q$, with center
$c_Q$ and side $R_j$.
Inside $Q$, choose two auxiliary cubes $Q^L$ and $Q^R$, depending on
$Q$, each of side
${2^{-4}} \tau  R_j$, centered respectively at
$$
 c_Q-\frac{{2^{-4}} R_j}{\sqrt2}e_1
 \qquad{\rm and}\qquad
 c_Q+\frac{{2^{-4}} R_j}{\sqrt2}e_1.
$$

\begin{lemma}\label{lem:first-common-ancestor}
Fix $1\leq j\leq N$ and a level-$j$ cube $Q$. 
Let $(Q_1,Q_2)\in\mathcal Q_0\times\mathcal Q_0$ satisfy
\begin{equation}\label{eq:admissible-pair}
 Q_1\subset Q^L,
 \qquad Q_2\subset Q^R.
\end{equation}
Then
\begin{equation}\label{eq:geometric-first-ancestor}
 Q_1^{(j)}=Q_2^{(j)}=Q,
 \qquad Q_1^{(j-1)}\ne Q_2^{(j-1)}.
\end{equation}
\end{lemma}

\begin{proof}
The sides of $Q^L,Q^R$ are parallel to the coordinate axes.
Their intervals in coordinates $2,\ldots,d$ agree, and their intervals in the first coordinate
are separated by $(\sqrt2-\tau){2^{-4}} R_j$.  Therefore
$$
 \operatorname{dist}(Q^L,Q^R)
 =(\sqrt2-\tau){2^{-4}} R_j
 >\frac{R_j}{\mathcal L},
$$
by the lower bound $\mathcal L\geq2^7\sqrt d$ fixed in
Section~\ref{sec:cube-tree}.
A level-$(j-1)$ subcube of $Q$ has an interval in the first coordinate
of length $R_j/\mathcal L$, so it cannot meet both $Q^L$ and
$Q^R$.  Thus $Q_1,Q_2$ lie in different level-$(j-1)$
subcubes of $Q$.  Since both are contained in $Q$, this proves
\eqref{eq:geometric-first-ancestor}.
\end{proof}

The same choice gives $\mathcal L\geq2^6/\tau$, so $ {2^{-4}} \tau R_j/R_0= {2^{-4}}\tau \mathcal L^j\geq4$.
Each coordinate interval contains between ${2^{-4}} \tau R_j/R_0-2$ and
${2^{-4}}\tau  R_j/R_0+1$ whole intervals of length $R_0$ from the subdivision.  Hence
$$
 \left({\frac{2^{-5}\tau R_j}{R_0}}\right)^d
 \leq\#\{Q_1\in\mathcal Q_0:Q_1\subset Q^L\}\leq
 \left({\frac{2^{-3}\tau R_j}{R_0}}\right)^d.$$
Similarly, 
$
 \#\{Q_2\in\mathcal Q_0:Q_2\subset Q^R\}\approx({2^{-4}} \tau  )^d{\mathcal L^{dj}}.
$

Let
\begin{align}\label{P Q}
 \mathcal P_Q=\{(Q_1,Q_2)\in\mathcal Q_0\times\mathcal Q_0:
 Q_1\subset Q^L,\ Q_2\subset Q^R\}
\end{align}
be the set of cube pairs satisfying \eqref{eq:admissible-pair}.
Multiplying the counts gives
\begin{equation}\label{eq:number-admissible-cube-pairs}
 {c} {\mathcal L^{2dj}}\leq
 \#\mathcal P_Q
 \leq {C} {\mathcal L^{2dj}}.
\end{equation}
Fix $1\leq j\leq N$ and a level-$j$ cube $Q$.
Let $Q_1,Q_2\in\mathcal Q_0$ satisfy
\eqref{eq:admissible-pair}.  Lemma~\ref{lem:first-common-ancestor} gives
$$
 Q_1^{(j-1)}\ne Q_2^{(j-1)}.
$$
We select the ordered pair $(Q_1,Q_2)$ when
\begin{equation}\label{eq:selection-condition}
 Z_{Q_1}-Z_{Q_2}=j.
\end{equation}
Put $r_j={\mathcal L^{(2d-1)(2W-j)}}$. For such a pair,
\eqref{eq:atom-numbers} gives
\begin{equation}\label{eq:selected-atom-pair-count}
 m_{Q_1}n_{Q_2}=r_j.
\end{equation}
In particular, this product depends only on $j$, not on the
individual values of $Z_{Q_1}$ and $Z_{Q_2}$.

Combining \eqref{eq:selected-atom-pair-count} with
\eqref{eq:dyadic-radii} and \eqref{eq:atom-size} gives
\begin{align}\label{eq:count-identity}
 {\mathcal L^{2dj}}r_j
 &=\mathcal L^{2dj}\mathcal L^{(2d-1)(2W-j)}
 =\mathcal L^{j+2(2d-1)W}
 =\theta^2\frac{R_j}{s}.
\end{align}
The coefficient definitions \eqref{eq:atom-coefficients} then give
\begin{equation}\label{eq:coefficient-product}
 \mu_{Q_1}\nu_{Q_2}
 =\frac{\theta^2R_0^{2d}}{r_j}
 =sR_j^{2d-1}
\end{equation}
whenever $Z_{Q_1}-Z_{Q_2}=j$.

For $Q\in\mathcal Q_j$, let
$$
 \mathcal P_{j,Q}
 =\{(Q_1,Q_2)\in\mathcal P_Q: Z_{Q_1}-Z_{Q_2}=j\}
$$
be the set of selected cube pairs, and define
\begin{equation}\label{eq:selected-atom-indices-Q}
 \mathcal E_{j,Q}
 =\bigl\{((Q_1,r),(Q_2,t)):
 (Q_1,Q_2)\in\mathcal P_{j,Q},\
 0\leq r<m_{Q_1},\ 0\leq t<n_{Q_2}\bigr\}.
\end{equation}
Let
\begin{equation}\label{eq:selected-atom-indices}
 \mathcal E_j
 =\bigcup_{Q\in\mathcal Q_j}\mathcal E_{j,Q}.
\end{equation}
We write $e=((Q_1,r),(Q_2,t))$ for the elements of these sets and
call them the selected atom pairs.
These sets depend on the labels through \eqref{eq:selection-condition}
and \eqref{eq:atom-numbers}, and the regions $S_e$ defined below
depend also on the base points through the centers $u,v$. We keep
these dependences implicit.

\subsection{Local lower bounds}
\label{sec:local-lower-bound}

Let $0<h_0\leq1/2$ be the constant of Lemma~\ref{lem:J-positive}
with $k=2$, fixed in Subsection~\ref{sec:labels-coefficients}.
This constant is independent of $N$.
For $e=((Q_1,r),(Q_2,t))\in\mathcal E_j$, put
$u=u_{Q_1,r}$, $v=v_{Q_2,t}$, and
\begin{equation}\label{eq:pair-midpoint}
 z_e=(u+v)/2.
\end{equation}
For $R>0$ and $x,u,v\in\mathbb R^d$, define
$$
 \Phi_R(x;u,v)=|x-u|^2+|x-v|^2-R^2.
$$
Put
$$
 r_{u,v}^2=\frac{R^2}{2}-\frac{|u-v|^2}{4}.
$$
We suppress the dependence on $R$, and $r_{u,v}>0$ whenever
$|u-v|<\sqrt2R$, which is the case below. Completing the square gives
\begin{equation}\label{eq:pair-sphere-radius}
 \Phi_R(x;u,v)=2\bigl(|x-z_e|^2-r_{u,v}^2\bigr).
\end{equation}
Set $R=R_j$ in this formula and define
\begin{equation}\label{eq:atom-pair-region}
 S_e
 =\left\{z_e+\lambda\zeta:
 \begin{array}{l}
 \lambda>0,\quad\zeta\in\mathbb S^{d-1},\quad
              |\zeta-e_d|\leq2^{-5},\\[5pt]
 |\lambda^2-r_{u,v}^2|\leq(h_0/2)sR_j
 \end{array}\right\}.
\end{equation}
Figure~\ref{fig:atom-pair-region} shows this part of a thin spherical shell.
By \eqref{eq:fit-identity}, $R_j\geq R_0$, and \eqref{eq:theta-small},
$$
 0<\frac{s}{R_j}
 \leq\theta^2
 \leq\epsilon_0
 \qquad(1\leq j\leq N).
$$

\begin{figure}[t]
\centering
\begin{tikzpicture}[x=1cm,y=1cm,font=\small]
  \fill[blue!18]
    (70:2.43) arc[start angle=70,end angle=110,radius=2.43]
    -- (110:2.17)
    arc[start angle=110,end angle=70,radius=2.17] -- cycle;

  \draw[black!40,dashed] (0,0) circle[radius=2.30];
  \draw[blue!65!black,thick]
    (70:2.43) arc[start angle=70,end angle=110,radius=2.43];
  \draw[blue!65!black,thick]
    (70:2.17) arc[start angle=70,end angle=110,radius=2.17];
  \draw[blue!65!black] (70:2.17)--(70:2.43);
  \draw[blue!65!black] (110:2.17)--(110:2.43);
  \draw[black!35,dotted] (0,0)--(70:2.43);
  \draw[black!35,dotted] (0,0)--(110:2.43);

  \draw[->,black!65] (0,0)--(0,3.10)
    node[above] {$e_d$};
  \draw[->,orange!85!black] (0,0)--(78:2.80)
    node[above right] {$\zeta$};
  \fill[orange!85!black] (78:2.37) circle[radius=1.8pt];
  \draw[orange!85!black] (78:2.37)--(1.18,2.40);
  \node[anchor=west,text=orange!85!black] at (1.22,2.40)
    {$x=z_e+\lambda\zeta$};
  \node[text=orange!85!black] at (.51,1.35) {$\lambda$};

  \draw[black!65] (-.20,-.02)--(.20,.02);
  \fill (-.20,-.02) circle[radius=1.3pt];
  \fill (.20,.02) circle[radius=1.3pt];
  \fill (0,0) circle[radius=1.5pt];
  \node[above left] at (0,.05) {$z_e$};
  \draw[black!55] (-.20,-.02)--(-.65,-.42);
  \node[below left] at (-.65,-.42) {$u=u_{Q_1,r}$};
  \draw[black!55] (.20,.02)--(.65,-.42);
  \node[below right] at (.65,-.42) {$v=v_{Q_2,t}$};

  \draw[black!55] (0,0)--(-75:2.30)
    node[pos=.66,below left] {$r_{u,v}$};
  \node[anchor=east,text=black!65] at (-1.55,-1.75)
    {$\Phi_{R_j}(x;u,v)=0$};
  \draw[black!40] (-1.53,-1.62)--(225:2.30);

  \node[text=blue!65!black] at (-1.65,2.82) {$S_e$};
  \draw[->,blue!65!black] (-1.42,2.73)--(103:2.32);
  \draw[black!50] (72:2.43)--(1.65,3.10)--(3.45,3.10);
  \node[anchor=west] at (3.50,3.10)
    {$|x-z_e|^2=r_{u,v}^2+\dfrac{h_0}{2}sR_j$};
  \draw[black!50] (70:2.17)--(1.70,1.65)--(3.45,1.65);
  \node[anchor=west] at (3.50,1.65)
    {$|x-z_e|^2=r_{u,v}^2-\dfrac{h_0}{2}sR_j$};
  \node[anchor=west] at (3.50,.75)
    {Directions: $|\zeta-e_d|\leq2^{-5}$.};
\end{tikzpicture}
\caption{A two-dimensional section of $S_e$ in $\mathbb R^d$.
The shaded region is part of a thin shell centered at $z_e=(u+v)/2$.
The dashed circle has radius $r_{u,v}$.
The cap opening and shell thickness are enlarged.}
\label{fig:atom-pair-region}
\end{figure}
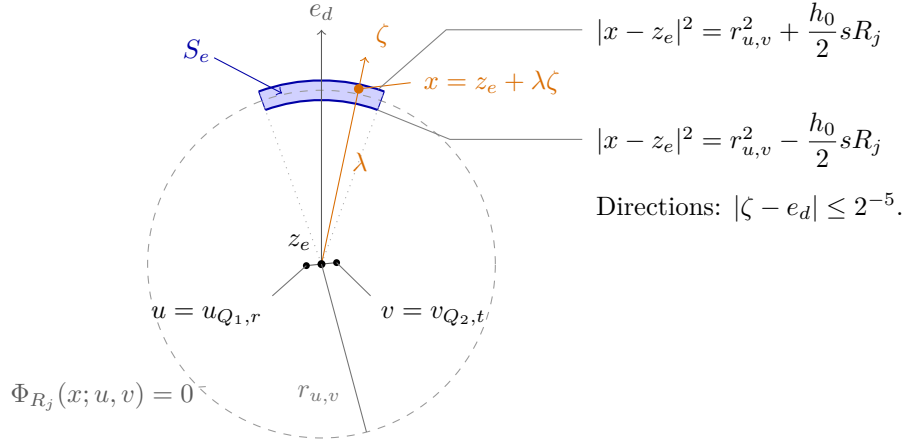

Thus $S_e$ is part of a thin shell about the sphere
$\Phi_{R_j}(x;u,v)=0$, centered at $z_e$ and of radius $r_{u,v}$.
Its directions lie in the fixed upper cap
$|\zeta-e_d|\leq{2^{-5}}$.  In particular,
\begin{equation}\label{size Phi}
|\Phi_{R_j}(x;u,v)|\leq h_0sR_j
\end{equation}
for $x\in S_e$.

Since $u\in Q^L$ and $v\in Q^R$, the positions and side
lengths of these cubes give $|u-v|\leq2^{-3}R_j$.
Hence $r_{u,v}\approx R_j$.
We claim that
\begin{equation}\label{eq:region-size-and-lower-bound}
 |S_e|\approx sR_j^{d-1}.
\end{equation}

Indeed, $|u-v|\leq2^{-3}R_j$ gives
$127\cdot2^{-8}R_j^2\leq r_{u,v}^2\leq R_j^2/2$.
Using $h_0\leq1$ and $s/R_j\leq\epsilon_0\leq1/8$, we obtain
\begin{equation}\label{ruv + -}
 \frac14R_j^2
 \leq r_{u,v}^2-\frac{h_0}{2}sR_j
 \leq r_{u,v}^2+\frac{h_0}{2}sR_j
 \leq\frac9{16}R_j^2.
\end{equation}
The angular cap in \eqref{eq:atom-pair-region} has fixed positive
surface measure. Polar coordinates centered at $z_e$ therefore give
$$
 |S_e|\approx
 \int_{(r_{u,v}^2-h_0sR_j/2)^{1/2}}
      ^{(r_{u,v}^2+h_0sR_j/2)^{1/2}}\lambda^{d-1}\,d\lambda
 \approx R_j^{d-1}\frac{h_0sR_j}{R_j}
 \approx sR_j^{d-1}.
$$
Here \eqref{ruv + -} gives $\lambda\approx R_j$, and rationalizing
the endpoints gives a radial width comparable to $h_0s$.
The constants depend only on $d$ and the fixed $h_0$.

\subsubsection{The lower bound on \texorpdfstring{$S_e$}{Se}}

\begin{proposition}\label{cor:positive-region}
There is a constant $c_0>0$ independent of $N,j,e$, the labels,
and the base points, such that
 for every $e=((Q_1,r),(Q_2,t))\in\mathcal E_j$,
$$
 \A_{R_j}(F_{Q_1,r},G_{Q_2,t})(x)
 \geq c_0,
 \qquad x\in S_e.
$$
\end{proposition}

\begin{proof}
Fix $e=((Q_1,r),(Q_2,t))\in\mathcal E_j$ and $x\in S_e$. Set
$
 R=R_j,\ u=u_{Q_1,r},\ v=v_{Q_2,t},\
 \mu=\mu_{Q_1},\ \nu=\nu_{Q_2},
$
and put
$$
 y_1=\frac{x-u}{R},\qquad y_2=\frac{x-v}{R},\qquad
 h=\frac{\Phi_R(x;u,v)}{2sR},\qquad \epsilon=\frac{s}{R}.
$$
We first prove that
\begin{equation}\label{n1 n2 h epsilon}
 \frac14\leq |y_1|,|y_2|\leq2,
 \qquad |h|\leq\frac{h_0}{2},
 \qquad 0<\epsilon\leq\epsilon_0.
\end{equation}
Write $x=z_e+\lambda\zeta$ as in \eqref{eq:atom-pair-region}. Then
\eqref{ruv + -} gives $R/2\leq\lambda\leq3R/4$. Since $|u-v|\leq R/8$ and
$$
 x-u=\lambda\zeta+(v-u)/2,\qquad
 x-v=\lambda\zeta-(v-u)/2,
$$
the triangle inequality bounds both $|y_1|$ and $|y_2|$
between $7/16$ and $13/16$. The bound on $h$ follows from
\eqref{size Phi}, and $\epsilon\leq\theta^2\leq\epsilon_0$ by
\eqref{eq:fit-identity} and $R\geq R_0$. This proves \eqref{n1 n2 h epsilon}.

We now compute the average exactly.
Changing variables $p=x-Ry$, $q=x-Rz$ on $\mathbb S^{2d-1}$, and
then applying the coarea formula, we obtain
\begin{align*}
 \A_R(F_{Q_1,r},G_{Q_2,t})(x)
 &=\frac1{\mathcal H^{2d-1}(\mathbb S^{2d-1})R^{2d-1}}
 \int_{|x-p|^2+|x-q|^2=R^2}
 F_{Q_1,r}(p)G_{Q_2,t}(q)\,
 d\mathcal H^{2d-1}(p,q)\\
 &=\frac2{\mathcal H^{2d-1}(\mathbb S^{2d-1})R^{2d-2}}
 \iint F_{Q_1,r}(p)G_{Q_2,t}(q)
 \delta(\Phi_R(x;p,q))\,dp\,dq.
\end{align*}
The second identity follows because
$|\nabla_{p,q}\Phi_R|=2R$ on the level set.

Substitute $p=u+s\alpha$ and $q=v+s\beta$.
The factors $s^{-2d}$ from the two blocks and $s^{2d}$
from $dp\,dq$ cancel. Then
\begin{align*}
 \Phi_R(x;u+s\alpha,v+s\beta)
 &=\Phi_R(x;u,v)
 -2sR(y_1\cdot\alpha+y_2\cdot\beta)
 +s^2(|\alpha|^2+|\beta|^2)\\
 &=2sR\left(
 h-y_1\cdot\alpha-y_2\cdot\beta
 +\frac{\epsilon}{2}(|\alpha|^2+|\beta|^2)
 \right).
\end{align*}

The factor $(2sR)^{-1}$ from the delta distribution therefore gives
\begin{equation}\label{eq:exact-coarea}
 \A_R(F_{Q_1,r},G_{Q_2,t})(x)
 =\frac{\mu\nu}{\mathcal H^{2d-1}(\mathbb S^{2d-1})sR^{2d-1}}
 J_\epsilon(y_1,y_2,h),
\end{equation}
where $J_\epsilon(y_1,y_2,h)=J^{(2)}_\epsilon((y_1,y_2),h)$ is the
integral in \eqref{eq:coarea-integral-definition}.

By \eqref{n1 n2 h epsilon}, we have $(y_1,y_2)\in K_2$, since
$1/4\geq1/(4\sqrt2)$, and also $|h|\leq h_0/2$ and
$0<\epsilon\leq\epsilon_0$. Lemma~\ref{lem:J-positive} with $k=2$
therefore gives $J_\epsilon(y_1,y_2,h)\geq c_\ast$, and
\eqref{eq:exact-coarea} implies
$$
 \A_R(F_{Q_1,r},G_{Q_2,t})(x)
 \geq
 \frac{c_\ast}{\mathcal H^{2d-1}(\mathbb S^{2d-1})}
 \frac{\mu\nu}{sR^{2d-1}}.
$$
Since $e\in\mathcal E_j$, \eqref{eq:coefficient-product} gives
$\mu\nu=sR^{2d-1}$. This proves the proposition with
$c_0=c_\ast/\mathcal H^{2d-1}(\mathbb S^{2d-1})$.
\end{proof}

\subsubsection{Disjointness for a fixed pair of cubes}

\begin{lemma}\label{lem:separation}
Let
$e=((Q_1,r),(Q_2,t))$ and
$e'=((Q_1,r'),(Q_2,t'))$ belong to $\mathcal E_j$, with
$(r,t)\neq(r',t')$.  Then $S_e\cap S_{e'}=\varnothing$.
\end{lemma}

\begin{proof}
Fix $(Q_1,Q_2)$ satisfying \eqref{eq:admissible-pair} and
\eqref{eq:selection-condition}.  As in
\eqref{eq:selected-atom-pair-count}, $m_{Q_1}n_{Q_2}=r_j$.
Recall $z_e$ and $r_{u,v}$ from \eqref{eq:pair-midpoint} and
\eqref{eq:pair-sphere-radius}, with $R=R_j$.
Write $x=(x',x_d)$ and $z_e=(z_e',z_{e,d})$.
We compare the upper spheres of the different atom pairs of this
cube pair at the same horizontal point $x'$.
For $x'\in\mathbb R^{d-1}$ with $|x'-z_e'|<r_{u,v}$, put
$$
 \Gamma_{u,v}(x')
 =z_{e,d}+\sqrt{r_{u,v}^2-|x'-z_e'|^2}.
$$
The upper half of the sphere $\Phi_{R_j}(x;u,v)=0$ is the graph of
$\Gamma_{u,v}$.

Consider the common horizontal domain
$$
 D_Q=\left\{x'\in\mathbb R^{d-1}:
 |x'-c_Q'|\leq{5\cdot2^{-4}}R_j\right\}.
$$
For every $e\in\mathcal E_{j,Q}$, the auxiliary cubes give
$$
 |u-c_Q|,|v-c_Q|
 \leq\left(\frac1{\sqrt2}+\frac{\sqrt d\,\tau}{2}\right)
 2^{-4}R_j<2^{-4}R_j,
$$
so $|z_e'-c_Q'|\leq2^{-4}R_j$. If $x=z_e+\lambda\zeta\in S_e$,
then \eqref{ruv + -} gives $\lambda\leq R_j$, and the cap condition
gives $|\zeta'|\leq2^{-5}$. Thus
$$
 |x'-c_Q'|\leq |z_e'-c_Q'|+\lambda|\zeta'|
 \leq3\cdot2^{-5}R_j<5\cdot2^{-4}R_j.
$$
Hence $D_Q$ contains the horizontal projection of every such $S_e$.

The cubes $Q^L,Q^R$ give
$|z_e'-c_Q'|\leq{2^{-4}} R_j$ and $|u-v|\leq{2^{-3}} R_j$.
On this domain $D_Q$, put $H=(r_{u,v}^2-|x'-z_e'|^2)^{1/2}$.  Then
$$
 H^2\geq
 {91\cdot2^{-8}}R_j^2
 \geq\frac14R_j^2.
$$
These bounds also hold when $u$ moves along a line segment in
$Q^L$, or $v$ moves along a line segment in $Q^R$, since
both cubes are convex.
For $x'\in D_Q$, differentiation gives
$$
 \partial_{u_d}\Gamma_{u,v}(x')
 =\frac12+\frac{v_d-u_d}{4H},
 \qquad
 \partial_{v_d}\Gamma_{u,v}(x')
 =\frac12+\frac{u_d-v_d}{4H}.
$$
The centers of $Q^L,Q^R$ have the same $d$-th coordinate,
and both cubes have side ${2^{-4}}\tau R_j$. Hence
$|u_d-v_d|\leq{2^{-4}}\tau R_j$.
Together with $H\geq R_j/2$, this gives
$$
 \frac{|u_d-v_d|}{4H}
 \leq\frac{{2^{-4}}\tau R_j}{4(R_j/2)}
 ={2^{-5}}\tau
 \leq\frac14.
$$
Therefore
\begin{equation}\label{eq:graph-derivatives}
 \frac14\leq\partial_{u_d}\Gamma_{u,v}(x')\leq\frac34,
 \qquad
 \frac14\leq\partial_{v_d}\Gamma_{u,v}(x')\leq\frac34.
\end{equation}

We next bound the vertical distance from a point of $S_e$
to its upper sphere.
Take $x\in S_e$, and write
$x=(x',x_d)=z_e+\lambda\zeta$.
The projection estimate above gives $x'\in D_Q$, so
$H\geq R_j/2$.
Moreover,
$
 \zeta_d
 =1-|\zeta-e_d|^2/2
 \geq1- {2^{-11}}>0,
$
and hence $x_d-z_{e,d}=\lambda\zeta_d>0$.

Since $\Gamma_{u,v}(x')=z_{e,d}+H$, we have
\begin{align*}
 \Phi_{R_j}(x;u,v)
 &=2\bigl((x_d-z_{e,d})^2-H^2\bigr)
 =2\bigl(x_d-\Gamma_{u,v}(x')\bigr)(x_d-z_{e,d}+H).
\end{align*}
Note that
$
 2(x_d-z_{e,d}+H)\geq2H\geq R_j.
$
Therefore, combining with \eqref{size Phi}, we obtain 
$$
 |x_d-\Gamma_{u,v}(x')|
 =\frac{|\Phi_{R_j}(x;u,v)|}{2(x_d-z_{e,d}+H)}
 \leq h_0s
 \leq s.
$$
Thus, for each fixed $x'\in D_Q$, the vertical section of $S_e$
is contained in
$$
 [\,\Gamma_{u,v}(x')-s,\ \Gamma_{u,v}(x')+s\,].
$$

We now compare the two graphs. Figure~\ref{fig:height-ordering}
illustrates the two cases below. The selection condition
$Z_{Q_1}-Z_{Q_2}=j$ gives
$
 n_{Q_2}=n_{Q_1}{\mathcal L^{-(2d-1)j}}\leq n_{Q_1}.
$
If $r=r'$ and $t\neq t'$,
\eqref{eq:graph-derivatives} and the mean value theorem show that the
vertical distance between the graphs is at least $2^2s$.
If $r\neq r'$, interchange the pairs if necessary so that $r>r'$.
Since $|t-t'|<n_{Q_2}\leq n_{Q_1}$,
\begin{align*}
 \big|\Gamma_{u_{Q_1,r},v_{Q_2,t}}-\Gamma_{u_{Q_1,r'},v_{Q_2,t'}}\big|
 &\geq
 \big|\Gamma_{u_{Q_1,r},v_{Q_2,t}}-\Gamma_{u_{Q_1,r'},v_{Q_2,t}}\big|-
 \big|\Gamma_{u_{Q_1,r'},v_{Q_2,t}}-\Gamma_{u_{Q_1,r'},v_{Q_2,t'}}\big|\\
 &\geq 2^4n_{Q_1}(r-r')s-3\cdot2^2|t-t'|s
 \geq2^2n_{Q_1}s>2s.
\end{align*}
At each $x'$, the two graph heights differ by more than $2s$,
while each vertical section lies within $s$ of its graph.
Hence $S_e\cap S_{e'}=\varnothing$, proving the lemma.
\end{proof}

\begin{figure}[htbp]
\centering
\begin{tikzpicture}[x=1cm,y=1cm,font=\scriptsize]
    \node[font=\small] at (3.0,4.85) {$r=r',\quad t>t'$};
  \draw[->,black!55] (1.5,.25)--(1.5,4.05)
    node[above] {$x_d$};
  \filldraw[fill=blue!12,draw=blue!65!black]
    (1.38,.74) rectangle (1.62,1.16);
  \draw[blue!65!black] (1.28,.95)--(1.93,.95);
  \fill[blue!65!black] (1.5,.95) circle (1.2pt);
  \node[anchor=west,text=blue!65!black] at (2.0,.95)
    {$\Gamma_{u_{Q_1,r},v_{Q_2,t'}}(x')$};
  \filldraw[fill=orange!12,draw=orange!85!black]
    (1.38,3.04) rectangle (1.62,3.46);
  \draw[orange!85!black] (1.28,3.25)--(1.93,3.25);
  \fill[orange!85!black] (1.5,3.25) circle (1.2pt);
  \node[anchor=west,text=orange!85!black] at (2.0,3.25)
    {$\Gamma_{u_{Q_1,r},v_{Q_2,t}}(x')$};
  \draw[<->,black!60] (.95,.74)--(.95,1.16)
    node[midway,left] {$2s$};
  \draw[<->,black!60] (.95,3.04)--(.95,3.46)
    node[midway,left] {$2s$};
  \draw[<->,black!60] (1.85,1.16)--(1.85,3.04);
  \node[anchor=west,text=black!60] at (1.95,2.10)
    {disjoint intervals};
  \node[align=center] at (3.0,-.15)
    {Height difference $\geq2^2(t-t')s>2s$.};

    \node[font=\small] at (10.5,4.85) {$r>r'$};
  \draw[->,black!55] (9.0,.25)--(9.0,4.05)
    node[above] {$x_d$};
  \filldraw[fill=blue!12,draw=blue!65!black]
    (8.88,.74) rectangle (9.12,1.16);
  \draw[blue!65!black] (8.78,.95)--(9.43,.95);
  \fill[blue!65!black] (9.0,.95) circle (1.2pt);
  \node[anchor=west,text=blue!65!black] at (9.5,.95)
    {$\Gamma_{u_{Q_1,r'},v_{Q_2,t'}}(x')$};
  \filldraw[fill=orange!12,draw=orange!85!black]
    (8.88,2.79) rectangle (9.12,3.21);
  \draw[orange!85!black] (8.78,3.00)--(9.43,3.00);
  \fill[orange!85!black] (9.0,3.00) circle (1.2pt);
  \node[anchor=west,text=orange!85!black] at (9.5,3.00)
    {$\Gamma_{u_{Q_1,r},v_{Q_2,t}}(x')$};
  \draw[<->,black!60] (8.45,.74)--(8.45,1.16)
    node[midway,left] {$2s$};
  \draw[<->,black!60] (8.45,2.79)--(8.45,3.21)
    node[midway,left] {$2s$};
  \draw[<->,black!60] (9.35,1.16)--(9.35,2.79);
  \node[anchor=west,text=black!60] at (9.45,1.975)
    {disjoint intervals};
  \node[align=center] at (10.5,-.15)
    {Height difference $\geq2^2n_{Q_1}s>2s$.};
\end{tikzpicture}
\caption{Height comparisons at the same horizontal point $x'\in D_Q$.
Each shaded interval extends $s$ above and below
$\Gamma_{u,v}(x')$ and contains the vertical section of the
corresponding region.
}
\label{fig:height-ordering}
\end{figure}
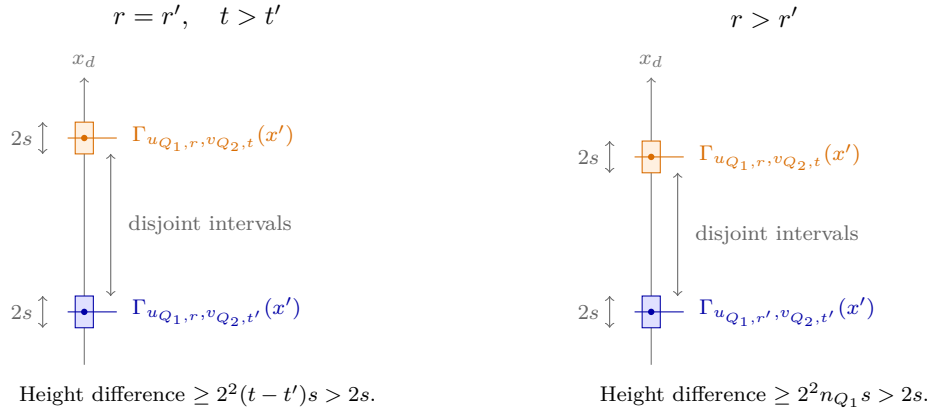

\section{Counting the label assignments}
\label{sec:region-counting}
To obtain a union of positive measure, we need a lower bound
for the total measure of the regions, counted with multiplicity,
and an upper bound for their overlap.
We first count the label assignments that select a cube pair.
Section~\ref{sec:region-moments} then uses this count to estimate
the counting function and its square.

Fix $1\leq j\leq N$, and let $Q_1,Q_2\in\mathcal Q_0$ satisfy
$$
 Q_1^{(j)}=Q_2^{(j)},\qquad
 Q_1^{(j-1)}\ne Q_2^{(j-1)}.
$$
By \eqref{eq:label-difference}, all labels of their common containing
cubes cancel in $Z_{Q_1}-Z_{Q_2}$.
The remaining coordinates are the labels of the $2j$ distinct cubes
$Q_1^{(r-1)},Q_2^{(r-1)}$, $1\leq r\leq j$.
Summation over these coordinates reduces the count to the function
$$
 D_j(\alpha,\beta)=\sum_{r=1}^jL_r(\alpha_r-\beta_r),
 \qquad (\alpha,\beta)\in\{0,1\}^j\times\{0,1\}^j.
$$
For $k\in\mathbb Z$, put
\begin{equation}\label{eq:label-difference-density}
 \kappa_j(k)=4^{-j}
 \#\{(\alpha,\beta)\in\{0,1\}^{2j}:D_j(\alpha,\beta)=k\}.
\end{equation}
Thus $\Av_\omega \one_{\{Z_{Q_1}-Z_{Q_2}=k\}}=\kappa_j(k)$ for
such a pair $(Q_1,Q_2)$.

The next lemma gives a local limit estimate for $D_j$.
 We include a short proof, since the weights $L_r$
are unbounded and we only need the two-sided bound for $j\kappa_j(j)$
in Lemma~\ref{lem:local-limit}.

\begin{lemma}\label{prop:lattice-fourier-estimate}
For $\kappa_j$ defined in \eqref{eq:label-difference-density}, put
$\sigma_j^2=\frac12\sum_{r=1}^jL_r^2$. Then
$$
 \sup_{k\in\mathbb Z}
 \left|
 \sigma_j\kappa_j(k)
 -\frac1{\sqrt{2\pi}}e^{-k^2/(2\sigma_j^2)}
 \right|\longrightarrow0
$$
as $j\to\infty$.
In particular,
$
 \kappa_j(j)\sim\frac{2e^{-2}}{\sqrt{2\pi}}\,\frac1j
$
as $j\to\infty$.
\end{lemma}

\begin{proof}
Since $L_r^2=r+O(\sqrt r)$, summation gives  $
 \sigma_j^2
 =\frac{j^2}{4}+O(j^{3/2}).$
Hence
\begin{equation}\label{eq:label-variance1}
 \sigma_j
 =\frac{j}{2}+O(j^{1/2}).
\end{equation}

Direct summation over the binary coordinates gives
\begin{align}\label{eq:label-characteristic-function}
 \varphi_j(t)&:=\sum_{k\in\mathbb Z}\kappa_j(k)e^{ikt}
 =\frac{1}{4^j}\sum_{(\alpha,\beta)\in\{0,1\}^{2j}}
   e^{it \sum_{r=1}^jL_r (\alpha_r-\beta_r)}
 =\prod_{r=1}^j\frac{2+e^{iL_rt}+e^{-iL_rt}}4\notag\\
 &=\prod_{r=1}^j\cos^2\left(\frac{L_rt}{2}\right).
\end{align}

Fourier inversion therefore gives
$$
 \kappa_j(k)=\frac1{2\pi}
 \int_{-\pi}^{\pi}e^{-ikt}\varphi_j(t)\,dt,\qquad k\in\mathbb Z. 
$$
Define
$
 \Psi_j(u)=\one_{\{|u|\leq\pi\sigma_j\}}
 \varphi_j(u/\sigma_j).
$
Changing variables by $u=\sigma_jt$, we obtain
\begin{align*}
 \sigma_j\kappa_j(k)
 &=\frac1{2\pi}\int_{\mathbb R}
 e^{-i(k/\sigma_j)u}\Psi_j(u)\,du.
\end{align*}
The Gaussian Fourier transform gives 
$$ 
\frac1{\sqrt{2\pi}}e^{-k^2/(2\sigma_j^2)}
 =\frac1{2\pi}\int_{\mathbb R}
 e^{-i(k/\sigma_j)u}e^{-u^2/2}\,du.
$$
Since the phase $e^{-i(k/\sigma_j)u}$ has absolute value one,
$$
 \sup_{k\in\mathbb Z}
 \left|\sigma_j\kappa_j(k)
 -\frac1{\sqrt{2\pi}}e^{-k^2/(2\sigma_j^2)}\right|
 \leq\frac1{2\pi}
 \|\Psi_j-e^{-u^2/2}\|_{L^1(\mathbb R)}.
$$
It is thus enough to prove that this $L^1$ norm tends to zero.
Fix $0<\beta\leq1$.  Splitting at $|u|=\beta\sigma_j/\sqrt j$
and applying the triangle inequality outside this interval gives
\begin{align*}
 \int_{\mathbb R}|\Psi_j(u)-e^{-u^2/2}|\,du
 &\leq\int_{|u|\leq \beta\sigma_j/\sqrt j}
 |\Psi_j(u)-e^{-u^2/2}|\,du\\
 &\qquad+\int_{|u|>\beta\sigma_j/\sqrt j}|\Psi_j(u)|\,du
 +\int_{|u|>\beta\sigma_j/\sqrt j}e^{-u^2/2}\,du\\
 &=:T_1+T_2+T_3.
\end{align*}
The Gaussian tail $T_3$ tends to zero because
$\beta\sigma_j/\sqrt j\to\infty$. We estimate the other two terms below.

\medskip
\noindent\emph{Estimate for $T_1$.}  For every fixed $A>0$, we show that
$\varphi_j(u/\sigma_j)\to e^{-u^2/2}$ uniformly on $[-A,A]$.
Since $\pi\sigma_j\to\infty$, this also gives pointwise convergence
of $\Psi_j$ on $\mathbb R$.

Taylor expansion gives
\begin{equation}\label{log formula}
 \log\cos^2(v/2)=-\frac{v^2}{4}+O(v^4)
 \qquad(|v|\leq1).
\end{equation}
For $|u|\leq A$, we have
$\max_{r\leq j}L_r|u|/\sigma_j=O_A(j^{-1/2})$ by
\eqref{eq:label-variance1}. Applying \eqref{log formula} to each
factor in \eqref{eq:label-characteristic-function} gives
$$
 \log\varphi_j(u/\sigma_j)
 =-\frac{u^2}{4\sigma_j^2}\sum_{r=1}^jL_r^2
   +O\left(\frac{u^4}{\sigma_j^4}\sum_{r=1}^jL_r^4\right)
 =-\frac{u^2}{2}+O_A(j^{-1}).
$$
Here $\sum_rL_r^2=2\sigma_j^2$ and
$\sum_{r\leq j}L_r^4\leq\sum_{r\leq j}r^2=O(j^3)$.
Exponentiating proves the uniform convergence on $[-A,A]$.

We next find an integrable majorant.
For $|v|\leq1$,
$$
 -\log\cos^2(v/2)
 =2\int_0^{|v|/2}\tan w\,dw
 \geq2\int_0^{|v|/2}w\,dw=\frac{v^2}{4}.
$$
Thus $\cos^2(v/2)\leq e^{-v^2/4}$ for $|v|\leq1$.

If $|t|\leq \beta j^{-1/2}$, then $|L_rt|\leq \beta\leq1$
for all $r\leq j$.  Multiplying the bounds for the individual
factors gives
$$
 |\varphi_j(t)| = \prod_{r=1}^j \cos^2\Big( \frac{L_rt}{2} \Big)
 \leq\exp\left(-\frac{t^2}{4}\sum_{r=1}^jL_r^2\right)
 \leq e^{-\frac{\sigma_j^2t^2}{2}}.
$$
Recalling that $u=t\sigma_j$, we have
$$
 |\Psi_j(u)|\leq e^{-\frac{u^2}{2}}
 \qquad\left(|u|\leq\frac{\beta\sigma_j}{\sqrt j}\right).
$$
Hence
$$
\one_{\{|u|\leq \beta\sigma_j/\sqrt j\}}|\Psi_j(u)-e^{-u^2/2}|\leq 2e^{-u^2/2}.
$$
Dominated convergence therefore gives
$$
 \int_{|u|\leq \beta\sigma_j/\sqrt j}
 |\Psi_j(u)-e^{-u^2/2}|\,du\longrightarrow0.
$$
This proves the required estimate for $T_1$.

\medskip
\noindent\emph{Estimate for $T_2$.} By the definition of $\Psi_j$ and the change of variables
$u=\sigma_jt$,
$$
 \int_{|u|>\beta\sigma_j/\sqrt j}|\Psi_j(u)|\,du
 =\sigma_j\int_{\beta/\sqrt j<|t|\leq\pi}
 |\varphi_j(t)|\,dt
 \leq2\pi\sigma_j
 \sup_{\beta/\sqrt j\leq|t|\leq\pi}|\varphi_j(t)|.
$$
Since $\sigma_j\approx j$, to show that $T_2$ tends to zero it is
enough to find $\tilde\beta>0$ such that, for all sufficiently
large $j$,
\begin{equation}\label{eq:tail-decay}
 \sup_{\beta j^{-1/2}\leq|t|\leq\pi}
 |\varphi_j(t)|\leq e^{-\tilde\beta \sqrt j}.
\end{equation}

For sufficiently large $j$, put $M_j=\lfloor\sqrt j\rfloor-1$.
Each weight $q\leq M_j$ occurs $2q+1$ times, at the indices
$q^2\leq r<(q+1)^2$. These groups lie entirely in
$\{1,\ldots,j\}$. Discarding the remaining factors, which are at
most one, gives
\begin{equation}\label{varphi upper bound}
 |\varphi_j(t)|\leq\prod_{q=1}^{M_j}|\cos(qt/2)|^{4q+2}.
\end{equation}
For $x\in\mathbb R$, let
$\|x\|_{\mathbb R/\mathbb Z}=\min_{n\in\mathbb Z}|x-n|$.
We claim that 
\begin{equation}\label{cos upper bound}
 |\cos(\pi x)|^2
 \leq e^{-\pi^2\|x\|_{\mathbb R/\mathbb Z}^2}.
\end{equation}
Both sides are one-periodic and even, so it suffices to consider
$0\leq x\leq1/2$. For $0\leq x<1/2$, we have
$\|x\|_{\mathbb R/\mathbb Z}=x$ and
$$
 -\log|\cos(\pi x)|^2
 =2\int_0^{\pi x}\tan v\,dv
 \geq2\int_0^{\pi x}v\,dv
 =\pi^2x^2.
$$
Here we used $\tan v\geq v$ for $0\leq v<\pi/2$.
Taking exponentials proves \eqref{cos upper bound}.
At $x=1/2$, the cosine is zero, so the bound also holds.

Consequently, from \eqref{varphi upper bound} and \eqref{cos upper bound}
\begin{align}\label{eq:grouped-product}
 |\varphi_j(t)|
  &\leq\exp\left(-c\sum_{q=1}^{M_j}(2q+1)
 \left\|\frac{qt}{2\pi}\right\|_{\mathbb R/\mathbb Z}^2\right)
 \leq\exp\left(-c\sum_{q=1}^{M_j}
 \left\|\frac{qt}{2\pi}\right\|_{\mathbb R/\mathbb Z}^2\right).
\end{align}
The last inequality uses $2q+1\geq1$. To bound the distance sum,
we prove
\begin{equation}\label{eq:spacing-estimate}
 \sum_{q=1}^M\|q\alpha\|_{\mathbb R/\mathbb Z}^2
 \geq cM\min\{1,M^2\|\alpha\|_{\mathbb R/\mathbb Z}^2\}
 \qquad(M\geq2,\ \alpha\in\mathbb R).
\end{equation}

Put $\delta=\|\alpha\|_{\mathbb R/\mathbb Z}\in[0,1/2]$.
Periodicity and evenness give $\|q\alpha\|=\|q\delta\|$
for every integer $q$, where the norms below denote distance
to the nearest integer.
We consider three cases.

\smallskip
\noindent\emph{Case 1: $M\delta\leq1/2$.} For $1\leq q\leq M$, we have $0\leq q\delta\leq1/2$,
so $\|q\alpha\|=q\delta$.  Hence
$$
 \sum_{q=1}^M\|q\alpha\|^2
 =\delta^2\frac{M(M+1)(2M+1)}6
 \geq\frac13M^3\delta^2
$$
since $M^2\delta^2\leq1/4$.  This case also includes $\delta=0$.

\smallskip
\noindent\emph{Case 2: $M\delta>1/2$ and $\delta>1/4$.}
The triangle inequality on $\mathbb R/\mathbb Z$ gives
$$
 \delta
 =\|(q+1)\alpha-q\alpha\|
 \leq\|q\alpha\|+\|(q+1)\alpha\|.
$$
Using $a^2+b^2\geq(a+b)^2/2$, we obtain
$$
 \|q\alpha\|^2+\|(q+1)\alpha\|^2
 \geq\frac{\delta^2}{2}\geq\frac1{32}.
$$
Thus, 
$$
 \sum_{q=1}^M\|q\alpha\|^2
 \geq\frac{\lfloor M/2\rfloor}{32}\geq\frac{M}{96},
$$
since $\lfloor M/2\rfloor\geq M/3$. This implies
\eqref{eq:spacing-estimate} in Case 2.

\smallskip
\noindent\emph{Case 3: $M\delta>1/2$ and $0<\delta\leq1/4$.}
We use blocks of $m=\lceil(2\delta)^{-1}\rceil$ terms. Then
$$
 2\leq m\leq M,
 \qquad
 \frac12\leq m\delta\leq\frac12+\delta\leq\frac34.
$$
We first show that every block of $m$ consecutive terms has
sum at least $cm$, independently of its initial point.

The one-periodic function $F(x)=\|x\|_{\mathbb R/\mathbb Z}^2$
is Lipschitz with constant one: its derivative in each $(n-1/2,n+1/2)$ is $2(x-n)$, and $F$ is
continuous at the endpoints.  Thus, for every $x\in\mathbb R$,
\begin{align*}
 \left|\delta\sum_{q=1}^mF(x+q\delta)
 -\int_x^{x+m\delta}F(y)\,dy\right|
 &\leq\sum_{q=1}^m\int_{x+(q-1)\delta}^{x+q\delta}
             |x+q\delta-y|\,dy
 =\frac{m\delta^2}{2}\leq\frac{3\delta}{8}.
\end{align*}
Since $m\delta\geq1/2$ and $F\geq0$, we obtain that 
$$
 \int_x^{x+m\delta}F(y)\,dy
 \geq\min_{a\in\mathbb R}\int_a^{a+1/2}F(y)\,dy =\int_{-1/4}^{1/4}F(y)\,dy
 =\int_{-1/4}^{1/4}y^2\,dy
 =\frac1{96}.
$$
Combining the integral lower bound with the error estimate gives
$$
 \delta\sum_{q=1}^mF(x+q\delta)
 \geq\frac1{96}-\frac{3\delta}{8}.
$$

If $0<\delta\leq1/72$, then $3\delta/8\leq1/192$, so the above inequality yields that 
$$
 \sum_{q=1}^mF(x+q\delta)
 \geq \bigg(\frac1{96}-\frac{3\delta}{8}\bigg)\frac{1}{\delta}
 \geq\frac1{192\delta}\geq\frac m{144}
$$
since $m\delta\leq3/4$.

If $1/72<\delta\leq1/4$, use instead
$
 \delta
 \leq\|x+q\delta\|+\|x+(q+1)\delta\|.
$
Each disjoint adjacent pair contributes at least
$\delta^2/2\geq1/(2\cdot72^2)$ to the sum of squares.
Since $\lfloor m/2\rfloor\geq m/3$, we obtain
$$
 \sum_{q=1}^mF(x+q\delta)
 \geq\frac{\lfloor m/2\rfloor}{2\cdot72^2}
 \geq\frac{m}{6\cdot72^2}.
$$

Thus, in both subcases,
$$
 \sum_{q=1}^mF(x+q\delta)\geq cm
 \qquad(x\in\mathbb R),
$$
where $c>0$ is absolute and independent of the initial point.

There are $\lfloor M/m\rfloor$ complete blocks, covering at least
$M/2$ indices because $m\leq M$. Apply the preceding bound to each
block with initial point $x=\ell m\delta$, and discard the
nonnegative remainder. This gives
$$
 \sum_{q=1}^M\|q\alpha\|^2
 \geq cm\lfloor M/m\rfloor\geq cM/2.
$$
Together with Cases 1 and 2, this proves
\eqref{eq:spacing-estimate} with an absolute constant.

Apply \eqref{eq:spacing-estimate} with $M=M_j$ and
$\alpha=t/(2\pi)$.  For $|t|\leq\pi$, we have
$\|t/(2\pi)\|=|t|/(2\pi)$, so
$$
 \sum_{q=1}^{M_j}
 \left\|\frac{qt}{2\pi}\right\|_{\mathbb R/\mathbb Z}^2
 \geq cM_j\min\{1,M_j^2t^2\}.
$$
Since $M_j\approx\sqrt j$, the condition
$|t|\geq \beta j^{-1/2}$ gives $M_j^2t^2\geq c\beta^2$
for large $j$.  Hence
$$
 \sum_{q=1}^{M_j}
 \left\|\frac{qt}{2\pi}\right\|_{\mathbb R/\mathbb Z}^2
 \geq \tilde\beta \sqrt j
 \qquad(\beta j^{-1/2}\leq|t|\leq\pi).
$$
Combining this lower bound with \eqref{eq:grouped-product} gives
$|\varphi_j(t)|\leq e^{-c\tilde\beta\sqrt j}$ for
$\beta j^{-1/2}\leq|t|\leq\pi$, which is \eqref{eq:tail-decay}
after renaming $\tilde\beta$.
This proves the required estimate for $T_2$.

Combining the three estimates gives
$$\|\Psi_j-e^{-u^2/2}\|_{L^1(\mathbb R)}\to0.$$
Substitution into the Fourier inversion bound proves the uniform
counting estimate in the lemma.
Finally, $j/\sigma_j\to2$, so we have
$$
 \sigma_j\kappa_j(j)\longrightarrow\frac{e^{-2}}{\sqrt{2\pi}},
 \qquad
 j\kappa_j(j)\longrightarrow\frac{2e^{-2}}{\sqrt{2\pi}}.
$$
This proves the stated asymptotic and completes the proof.
\end{proof}

For $j\geq1$, put
\begin{equation}\label{eq:cj-def}
 c_j=\kappa_j(j).
\end{equation}
By the identity following \eqref{eq:label-difference-density}, $c_j$
is the proportion of label choices for which a given pair of cubes
with first common ancestor at level $j$ satisfies
$Z_{Q_1}-Z_{Q_2}=j$.

\begin{lemma}\label{lem:local-limit}
There are absolute constants $c,C>0$ such that
\begin{equation}\label{eq:cj}
 \frac cj\leq c_j\leq\frac Cj
 \qquad (j\geq1).
\end{equation}
\end{lemma}

\begin{proof}
Lemma~\ref{prop:lattice-fourier-estimate} gives the estimate
for $j>J_0$, where $J_0$ is an absolute constant. To treat the
remaining indices, it suffices to show $\kappa_j(j)>0$ for every $j$.

Put $W_n=\sum_{r=1}^nL_r$ and $W_0=0$. Every integer from $0$ to
$W_n$ is a sum of a subset of $L_1,\ldots,L_n$. Indeed, if this
holds for $n-1$, the subset sums at step $n$ cover the integer intervals
$$
 [0,W_{n-1}]\quad\hbox{and}\quad[L_n,L_n+W_{n-1}].
$$
There is no missing integer between them, because
$L_n\leq n\leq W_{n-1}+1$. This proves the assertion by induction.
Since $W_j\geq j$, choose a subset $I\subset\{1,\ldots,j\}$ with
$\sum_{r\in I}L_r=j$. Taking $\alpha_r=\one_I(r)$ and $\beta_r=0$
gives $D_j(\alpha,\beta)=j$, so
$$
 \kappa_j(j)\geq4^{-j}>0.
$$
For $1\leq j\leq J_0$, we therefore have
$$
 \frac{4^{-J_0}}{j}\leq\kappa_j(j)\leq\frac{J_0}{j}.
$$
Combining this with the large-$j$ estimate proves \eqref{eq:cj}.
\end{proof}

\section{Counting the regions and controlling their overlap}
\label{sec:region-moments}

Recall the selected indices $\mathcal E_{j,Q}$ and $\mathcal E_j$
from \eqref{eq:selected-atom-indices-Q} and
\eqref{eq:selected-atom-indices}, and the regions $S_e$ from
\eqref{eq:atom-pair-region}. For $e\in\mathcal E_{j,Q}$ we always
write $e=((Q_1,r),(Q_2,t))$, and for a second pair
$\widetilde e=((\widetilde Q_1,\widetilde r),(\widetilde Q_2,\widetilde t))$.
A condition on $Q_i$ or $\widetilde Q_i$ under a sum over
$e,\widetilde e$ refers to these components. The dependence of
these objects on the labels and base points will be kept implicit.

\begin{definition}
For $1\leq j\leq N$ and $Q\in\mathcal Q_j$, define the counting function
\begin{align}
 \mathcal N_{j,Q}(x)
 &=\sum_{e\in\mathcal E_{j,Q}}\one_{S_e}(x)=\sum_{(Q_1,Q_2)\in\mathcal P_{j,Q}}
   \sum_{r=0}^{m_{Q_1}-1}\sum_{t=0}^{n_{Q_2}-1}
   \one_{S_{((Q_1,r),(Q_2,t))}}(x),
 \label{eq:cube-counting-functions}
 \\ 
 \mathcal N_j(x)
 &=\sum_{Q\in\mathcal Q_j}\mathcal N_{j,Q}(x)
 =\sum_{e\in\mathcal E_j}\one_{S_e}(x),
 \qquad
 \mathcal N(x)=\sum_{j=1}^{N}\mathcal N_j(x).
 \label{eq:counting-functions}
\end{align}
\end{definition}

To estimate the measure of the union of the regions $S_e$, we
prove a lower bound for the integral of the counting function
$\mathcal N$ and an upper bound for the integral of $\mathcal N^2$,
after averaging over the labels and the base points
(Proposition~\ref{prop:moments}). 
The lower bound follows from the counts in
Sections~\ref{sec:atom-pairs} and \ref{sec:region-counting}.
The upper bound controls the overlap of the regions.

\begin{proposition}\label{prop:moments}
There are constants $c,C>0$, independent of $N$, such that
\begin{align}
 &c{c_j\theta^2}\leq
 \Av_\omega\Av_{a,b}\int_{\mathbb R^d}\mathcal N_j(x)\,dx
 \leq C{c_j\theta^2}
 \qquad(1\leq j\leq N),
 \label{eq:first-moment}\\
& \Av_\omega\Av_{a,b}\int_{\mathbb R^d}\mathcal N^2(x)\,dx
 \leq C({\rho}+{\rho}^2).
 \label{eq:second-moment}
\end{align}
In particular,
$\Av_\omega\Av_{a,b}\int_{\mathbb R^d}\mathcal N(x)\,dx{\approx\theta^2\sum_{j=1}^{N}c_j\approx{\rho}}$.
\end{proposition}

We prove the first estimate in Section~\ref{sec:first-integral}.
Section~\ref{sec:discrete-coarea} gives the endpoint integral
estimates needed for the second estimate, which is proved in
Section~\ref{sec:second-integral}.

\subsection{Proof of \eqref{eq:first-moment}}
\label{sec:first-integral}

\begin{lemma}
For every $1\leq j\leq N$ and $Q\in\mathcal Q_j$,
\begin{equation}\label{eq:first-cube-integral}
 \Av_\omega\Av_{a,b}\int_{\mathbb R^d}\mathcal N_{j,Q}(x)\,dx
{\approx c_j\theta^2R_j^d}.
\end{equation}
\end{lemma}

\begin{proof}
Fix $j,Q$. For every selected atom pair,
$\Av_{a,b}|S_e|\approx sR_j^{d-1}$ uniformly by
\eqref{eq:region-size-and-lower-bound}. Its cube pair contributes
$r_j$ atom pairs by \eqref{eq:selected-atom-pair-count}.
By Lemma~\ref{lem:first-common-ancestor} and
\eqref{eq:label-difference-density}, for each $(Q_1,Q_2)\in\mathcal P_Q$ (defined as in \eqref{P Q}), 
the label average is
\begin{equation}\label{eq:single-pair-label-count}
 \Av_\omega\one_{\{Z_{Q_1}-Z_{Q_2}=j\}}=c_j.
\end{equation}
Summing first over atom indices and then averaging the labels gives
\begin{align*}
 \Av_\omega\Av_{a,b}\int\mathcal N_{j,Q}(x)\,dx
 &\approx c_jr_jsR_j^{d-1}\,\#\mathcal P_Q
 \approx c_j\mathcal L^{2dj}r_jsR_j^{d-1}
 =c_j\theta^2R_j^d.
\end{align*}
The second comparison uses \eqref{eq:number-admissible-cube-pairs},
and the equality uses \eqref{eq:count-identity}.
\end{proof}

There are $R_j^{-d}$ cubes in $\mathcal Q_j$. Summing
\eqref{eq:first-cube-integral} over these cubes proves
\eqref{eq:first-moment}. Summing again over $1\leq j\leq N$ gives
$\Av_\omega\Av_{a,b}\int_{\mathbb R^d}\mathcal N(x)\,dx{\approx\theta^2\sum_{j=1}^{N}c_j\approx{\rho}}$.

\subsection{Containing regions and integration in the base points}
\label{sec:discrete-coarea}

Before estimating products of counting functions, we locate all
regions belonging to one cube and bound their averages over the
base points.

\begin{definition}
For $Q\in\mathcal Q_j$, put
\begin{equation}\label{eq:containing-region}
 \mathcal O_{j,Q}
 =\left\{x=(x',x_d)\in\mathbb R^d:
 |x'-c_Q'|\leq{5\cdot2^{-4}}R_j,
 R_j/3\leq x_d-c_{Q,d}\leq R_j
 \right\}.
\end{equation}
\end{definition}

\begin{lemma}\label{lem: O disjoint}
For every $Q,\widetilde Q\in\mathcal Q_j$ with $Q\ne\widetilde Q$, we have $
 \mathcal O_{j,Q}\cap\mathcal O_{j,\widetilde Q}=\varnothing.
$
\end{lemma}

\begin{proof}
Each horizontal coordinate interval of $\mathcal O_{j,Q}$ has
length $5R_j/8<R_j$, and its vertical interval has length
$2R_j/3<R_j$. These regions are translates of one fixed set by
the cube centers, whose coordinate differences are integer multiples
of $R_j$. Distinct centers therefore give disjoint regions.
\end{proof}

For $x\in\mathbb R^d$ and $1\leq j\leq N$, put
$
 \mathcal G_j(x)
 =\{(u,v)\in\mathbb R^d\times\mathbb R^d:
 |\Phi_{R_j}(x;u,v)|\leq h_0sR_j\}.
$

\begin{lemma}
For every $e\in\mathcal E_{j,Q}$, with
$u=u_{Q_1,r}$ and $v=v_{Q_2,t}$, and every $x\in\mathbb R^d$,
\begin{equation}\label{eq:region-strip-containment}
 \one_{S_e}(x)
 \leq \one_{\mathcal O_{j,Q}}(x)
       \one_{\mathcal G_j(x)}(u,v).
\end{equation}
Moreover, we have
\begin{equation}\label{eq:output-overlap}
 |\mathcal O_{j,Q}|\approx R_j^d\qquad{\rm and}\qquad
 \sum_{Q\in\mathcal Q_j}|\mathcal O_{j,Q}|\leq \frac{2}{3}.
\end{equation}
These bounds are independent of the labels and base points.
\end{lemma}

\begin{proof}
Let $x\in S_e$ and write $x=z_e+\lambda\zeta$ as in
\eqref{eq:atom-pair-region}. The calculation in the proof of
Proposition~\ref{cor:positive-region} gives
$R_j/2\leq\lambda\leq3R_j/4$.
The positions of $Q^L,Q^R$ give $|z_e-c_Q|\leq{2^{-4}} R_j$.
Since $|\zeta-e_d|\leq{2^{-5}}$, we have
$|\zeta'|\leq{2^{-5}}$ and $\zeta_d\geq1-{2^{-11}}$.
Consequently,
$$
 |x'-c_Q'|
 \leq |z_e'-c_Q'|+\lambda|\zeta'|
 \leq{3\cdot2^{-5}}R_j
 \leq{5\cdot2^{-4}}R_j,
$$
and
$$
 {\bigl(2^{-1}-2^{-4}-2^{-12}\bigr)}R_j
 \leq x_d-c_{Q,d}
 \leq{\bigl(2^{-1}+2^{-2}+2^{-4}\bigr)}R_j.
$$
The left-hand side exceeds $R_j/3$, and the right-hand side
{is less than $R_j$, since $13\cdot2^{-4}<1$.}
Thus $x\in\mathcal O_{j,Q}$. Also,
\eqref{eq:pair-sphere-radius} and
\eqref{eq:atom-pair-region} give
$
 |\Phi_{R_j}(x;u,v)|
 =2|\lambda^2-r_{u,v}^2|
 \leq h_0sR_j.
$
This proves \eqref{eq:region-strip-containment}.

The definition gives $|\mathcal O_{j,Q}|\approx R_j^d$ and
$$
 |\mathcal O_{j,Q}|\leq\frac23\left(\frac58\right)^{d-1}R_j^d
 \leq\frac23R_j^d.
$$
Summing over the $R_j^{-d}$ cubes proves \eqref{eq:output-overlap}.
\end{proof}

\begin{lemma}
Fix $Q\in\mathcal Q_j$ and $x\in\mathcal O_{j,Q}$. For every
$u\in Q^L$,
\begin{equation}\label{eq:one-variable-coarea}
 \int_{Q^R}
 \one_{\mathcal G_j(x)}(u,v)\,dv
 \leq CsR_j^{d-1}.
\end{equation}
The same estimate holds with the endpoints interchanged.
In particular,
\begin{equation}\label{eq:one-scale}
 \iint_{Q^L\times Q^R}
 \one_{\mathcal G_j(x)}(u,v)\,du\,dv
 \leq CsR_j^{2d-1}.
\end{equation}
\end{lemma}

\begin{proof}
By the definition of $\Phi_{R_j}$,
$\partial_{v_d}\Phi_{R_j}(x;u,v)=-2(x_d-v_d)$.

For $u\in Q^L$ and $v\in Q^R$, the positions of these cubes give
$|u_d-c_{Q,d}|,|v_d-c_{Q,d}|\leq{2^{-4}} R_j$.
Since $x\in\mathcal O_{j,Q}$, we also have
$x_d-c_{Q,d}\geq R_j/3$. Hence
$$
 x_d-v_d
 =(x_d-c_{Q,d})-(v_d-c_{Q,d})
 \geq{\frac{13}{3}2^{-4}}R_j>R_j/4.
$$

For fixed $x,u,v'$, the function $v_d\mapsto\Phi_{R_j}(x;u,v)$ is
strictly decreasing, with derivative of absolute value at least
$R_j/2$. Its inverse
image of the interval $[-h_0sR_j,h_0sR_j]$ has length at most
$$
 \frac{2h_0sR_j}{R_j/2}=4h_0s\leq4s.
$$

Since $Q^R$ has side length ${2^{-4}}\tau R_j$, integrating over the remaining
$d-1$ coordinates of $v$ proves
\eqref{eq:one-variable-coarea}. Integrating also over $u\in Q^L$
proves \eqref{eq:one-scale}.

If $v$ is fixed, use
$\partial_{u_d}\Phi_{R_j}=-2(x_d-u_d)$.
The same lower bound for its absolute value proves the estimate
with the endpoints interchanged.
\end{proof}

\begin{lemma}
For every $1\leq j\leq N$, $Q\in\mathcal Q_j$, and
$x\in\mathbb R^d$,
\begin{equation}\label{eq:averaged-cube-count}
 \Av_\omega\Av_{a,b}\mathcal N_{j,Q}(x)
 \leq 
 C{c_j\theta^2}\one_{\mathcal O_{j,Q}}(x).
\end{equation}
\end{lemma}

\begin{proof}
Fix $\omega\in\{0,1\}^{\mathcal D}$ and an atom pair
$e=((Q_1,r),(Q_2,t))\in\mathcal E_{j,Q}$.  The indicator $\one_{S_e}(x)$ depends
only on $a_{Q_1}$ and $b_{Q_2}$. The normalized integrals over
all other base-point coordinates equal one. Hence, noting that $|{\frac18Q_1}|=|{\frac18Q_2}|=8^{-d}R_0^d$,
we have
\begin{align}
 \Av_{a,b} \one_{S_e}(x)
 &=\fint_{{\frac18Q_1}}\fint_{{\frac18Q_2}}
 \one_{S_e}(x)\,db_{Q_2}\,da_{Q_1}\notag\\
 &=8^{2d}R_0^{-2d}
 \int_{{\frac18Q_1}+{2^6}n_{Q_1}rs\,e_d}
 \int_{{\frac18Q_2}+{2^4}ts\,e_d}
 \one_{S_e}(x)\,dv\,du
 \notag\\
  &\leq
  8^{2d}R_0^{-2d}\one_{\mathcal O_{j,Q}}(x)
 \iint_{Q_1\times Q_2}
 \one_{\mathcal G_j(x)}(u,v)\,du\,dv.
 \label{eq:translated-pair-bound}
\end{align}
For the second equality, we use the changes of variables
$
 u=a_{Q_1}+{2^6}n_{Q_1}rs\,e_d,
 \ 
 v=b_{Q_2}+{2^4}ts\,e_d.
$
This change of variables has Jacobian one.
By \eqref{eq:centers-diameter}, the translated domains ${\frac18Q_1}+{2^6}n_{Q_1}rs\,e_d$ and ${\frac18Q_2}+{2^4}ts\,e_d$ lie in
$Q_1$ and $Q_2$, respectively.
Applying \eqref{eq:region-strip-containment} and enlarging
the integration domains to $Q_1\times Q_2$ gives the inequality above.

The enlargement to $Q_1\times Q_2$ removes the dependence of the
upper bound on the atom indices and labels. Thus, for fixed
$Q_1,Q_2$, the right-hand side of \eqref{eq:translated-pair-bound}
is independent of $\omega,r,t$. This is the order of averaging we
will also use for products of counting functions.
On the selection set, $m_{Q_1}n_{Q_2}=r_j$. Thus
$$
 \Av_\omega\Bigl[\one_{\{Z_{Q_1}-Z_{Q_2}=j\}}
 \sum_{r=0}^{m_{Q_1}-1}\sum_{t=0}^{n_{Q_2}-1}1\Bigr]
 =r_j\,\Av_\omega\one_{\{Z_{Q_1}-Z_{Q_2}=j\}}
 =c_jr_j,
$$
where the last equality is \eqref{eq:single-pair-label-count}.
By \eqref{eq:cube-counting-functions} and
\eqref{eq:translated-pair-bound},
\begin{align*}
 \Av_\omega\Av_{a,b}\mathcal N_{j,Q}(x)
 &\leq
 Cc_jr_jR_0^{-2d}\one_{\mathcal O_{j,Q}}(x)
 \sum_{(Q_1,Q_2)\in\mathcal P_Q}
 \iint_{Q_1\times Q_2}
 \one_{\mathcal G_j(x)}(u,v)\,du\,dv\\
 &\leq
 Cc_jR_0^{-2d}r_jsR_j^{2d-1}
 \one_{\mathcal O_{j,Q}}(x)\\
 &=Cc_j{\mathcal L^{2dj}}r_j\frac{s}{R_j}
 \one_{\mathcal O_{j,Q}}(x)
 =Cc_j\theta^2\one_{\mathcal O_{j,Q}}(x).
\end{align*}
The second inequality follows by combining the disjoint domains
$Q_1\times Q_2\subset Q^L\times Q^R$ and applying
\eqref{eq:one-scale}. The last equality uses
\eqref{eq:count-identity}.
This proves \eqref{eq:averaged-cube-count}.
\end{proof}

\subsection{Proof of \eqref{eq:second-moment}}
\label{sec:second-integral}

We distinguish pairs from different levels, pairs from different
cubes at the same level, and pairs from the same cube. We first
record which labels occur in each selection condition.

\begin{definition}
If $Q_1,Q_2\in\mathcal Q_0$ have first common ancestor at level $j$,
put
$$
 K(Q_1,Q_2)
 =\{Q_1^{(h)}:0\leq h\leq j-1\}
  \cup\{Q_2^{(h)}:0\leq h\leq j-1\}.
$$
\end{definition}

By \eqref{eq:label-difference}, $Z_{Q_1}-Z_{Q_2}$ depends only on the
labels $\omega(P)$ with $P\in K(Q_1,Q_2)$.

If $Q_1^{(j)}=Q_2^{(j)}=Q$, every cube in
$K(Q_1,Q_2)$ is strictly contained in $Q$.
Let $Q_1,Q_2,\widetilde Q_1,\widetilde Q_2\in\mathcal Q_0$,
and suppose the two pairs have first common ancestors at levels
$j$ and $k$, respectively. If
$$
 K(Q_1,Q_2)\cap K(\widetilde Q_1,\widetilde Q_2)=\varnothing,
$$
the two selection conditions depend on disjoint sets of labels.
By \eqref{eq:label-factorization} and \eqref{eq:cj-def},
\begin{align}
 \Av_\omega\bigl[\one_{\{Z_{Q_1}-Z_{Q_2}=j\}}
 \one_{\{Z_{\widetilde Q_1}-Z_{\widetilde Q_2}=k\}}\bigr]
=
 \left(\Av_\omega \one_{\{Z_{Q_1}-Z_{Q_2}=j\}}\right)
 \left(\Av_\omega \one_{\{Z_{\widetilde Q_1}
                         -Z_{\widetilde Q_2}=k\}}\right)
 =c_jc_k.
 \label{eq:cross-level-label-count}
\end{align}
In particular, this average equals $c_j^2$ when $k=j$.

We also need the following geometric separation between levels.

\begin{lemma}\label{lem:radial-separation}
Let $j<k$, $Q\in\mathcal Q_j$ and
$\widetilde Q\in\mathcal Q_k$. Let
$
 e=((Q_1,r),(Q_2,t))\in\mathcal E_{j,Q},
$ and
$ \widetilde e=
 ((\widetilde Q_1,\widetilde r),
  (\widetilde Q_2,\widetilde t))
 \in\mathcal E_{k,\widetilde Q}.
$
Then
$$
 \widetilde Q_1\subset Q
 \ \ \text{or}\ \
 \widetilde Q_2\subset Q
 \ \Longrightarrow\ 
 S_e\cap S_{\widetilde e}=\varnothing.
$$
In particular, for $j<k$,
$
 \{Q_1,Q_2\}\cap
 \{\widetilde Q_1,\widetilde Q_2\}\ne\varnothing
 \ \ \Longrightarrow\ \ 
 S_e\cap S_{\widetilde e}=\varnothing.
$
\end{lemma}

\begin{proof}
Put $u=u_{Q_1,r}$ and choose
$
 w\in
 \{u_{\widetilde Q_1,\widetilde r},
   v_{\widetilde Q_2,\widetilde t}\}\cap Q.
$
Then $u,w\in Q$, so $|u-w|\leq\sqrt d\,R_j$.
If $x\in S_e\cap S_{\widetilde e}$, then
\eqref{n1 n2 h epsilon} gives
$$
\frac{1}{4} \leq\frac{|x-u|}{R_j}\leq2,
 \qquad
\frac{1}{4} \leq\frac{|x-w|}{R_k}\leq2.
$$
Since $R_k/R_j\geq\mathcal L>4(2+\sqrt d)$, we obtain
$$
 \frac{R_k}{4}
 \leq |x-w|
 \leq |x-u|+|u-w|
 \leq (2+\sqrt d)R_j
 <\frac{R_k}{4},
$$
a contradiction.

If the two pairs share a level-$0$ cube, then
$\widetilde Q_i=Q_\ell\subset Q$ for some
$i,\ell\in\{1,2\}$. The first assertion therefore applies.
This proves the lemma.
\end{proof}

\begin{proof}[Proof of \eqref{eq:second-moment}]
Write
\begin{align*}
 \mathrm I_{j,k}
 &=\Av_\omega\Av_{a,b}\int_{\mathbb R^d}\mathcal N_j(x)\mathcal N_k(x)\,dx
 &&(1\leq j<k\leq N),\\
 \mathrm {II}_{j,Q,\widetilde Q}
 &=\Av_\omega\Av_{a,b}\int_{\mathbb R^d}
       \mathcal N_{j,Q}(x)\mathcal N_{j,\widetilde Q}(x)\,dx
 &&(Q,\widetilde Q\in\mathcal Q_j,\ Q\ne\widetilde Q),\\
 \mathrm {III}_{j,Q}
 &=\Av_\omega\Av_{a,b}\int_{\mathbb R^d}\mathcal N_{j,Q}^2(x)\,dx
 &&(Q\in\mathcal Q_j).
\end{align*}
Expanding the square gives
\begin{equation}\label{eq:second-integral-decomposition}
 \Av_\omega\Av_{a,b}\int_{\mathbb R^d}\mathcal N^2(x)\,dx
 =2\sum_{1\leq j<k\leq N}\mathrm I_{j,k}
 +\sum_{j=1}^{N}
    \sum_{Q,\widetilde Q\in\mathcal Q_j:\,Q\ne\widetilde Q}
       \mathrm {II}_{j,Q,\widetilde Q}
 +\sum_{j=1}^{N}\sum_{Q\in\mathcal Q_j}\mathrm {III}_{j,Q}.
\end{equation}
The sum over $Q\ne\widetilde Q$ is over ordered pairs.
We now prove
\begin{equation}\label{3cases}
\begin{aligned}
 &\mathrm I_{j,k}\leq C{c_jc_k\theta^4}
 &&(1\leq j<k\leq N),\\
 &\sum_{Q,\widetilde Q\in\mathcal Q_j:\,Q\ne\widetilde Q}
       \mathrm {II}_{j,Q,\widetilde Q}=0
 &&(1\leq j\leq N),\\
 &\mathrm {III}_{j,Q}\leq C({c_j\theta^2}+{c_j^2\theta^4})R_j^d
 &&(Q\in\mathcal Q_j,\ 1\leq j\leq N).
\end{aligned}
\end{equation}

\smallskip
\noindent\emph{Case 1: different levels, $j<k$. \ We prove $\mathrm I_{j,k}\leq C{c_jc_k\theta^4}$.}

Fix $Q\in\mathcal Q_j$ and $\widetilde Q\in\mathcal Q_k$.
Expand the two local counting functions over $e\in\mathcal E_{j,Q}$
and $\widetilde e\in\mathcal E_{k,\widetilde Q}$.
If either $\widetilde Q_1$ or $\widetilde Q_2$ lies in $Q$,
Lemma~\ref{lem:radial-separation} gives
$S_e\cap S_{\widetilde e}=\varnothing$, so these terms contribute zero.

Fix four remaining level-$0$ cubes. They are distinct.
With the labels and atom indices fixed, the two indicators depend
on disjoint base-point coordinates. Thus $\Av_{a,b} [\one_{S_e}(x)\one_{S_{\widetilde e}}(x)]
 =(\Av_{a,b} \one_{S_e}(x))
   (\Av_{a,b} \one_{S_{\widetilde e}}(x))
$. Hence,
\begin{align}
 \Av_{a,b} [\one_{S_e}(x)\one_{S_{\widetilde e}}(x)]
 &= \fint_{{\frac18Q_2}} \fint_{{\frac18Q_1}}
 \one_{S_e}(x)\,da_{Q_1}\,db_{Q_2}
 \cdot  \fint_{{\frac18\widetilde Q_2}}\fint_{ {\frac18\widetilde Q_1}}
 \one_{S_{\widetilde e}}(x)\,
 da_{\widetilde Q_1}\,db_{\widetilde Q_2}
 \notag\\
 &\leq CR_0^{-4d}
 \one_{\mathcal O_{j,Q}}(x)
 \one_{\mathcal O_{k,\widetilde Q}}(x)
 \left(\iint_{Q_1\times Q_2}
 \one_{\mathcal G_j(x)}(u,v)\,du\,dv\right)
 \notag\\
 &\qquad\times
 \left(\iint_{\widetilde Q_1\times\widetilde Q_2}
 \one_{\mathcal G_k(x)}(\widetilde u,\widetilde v)\,
 d\widetilde u\,d\widetilde v\right).
 \label{eq:four-base-point-bound}
\end{align}
The last inequality applies \eqref{eq:translated-pair-bound}
to each factor. 

We now sum over the atom indices and average over the labels.
We claim that
$$
 K(Q_1,Q_2)\cap K(\widetilde Q_1,\widetilde Q_2)=\varnothing.
$$
Suppose that $P$ belongs to both sets.
Since $P\in K(Q_1,Q_2)$, we have $P\subsetneq Q$.
Since $P\in K(\widetilde Q_1,\widetilde Q_2)$,
it contains $\widetilde Q_1$ or $\widetilde Q_2$.
Hence one of these two level-$0$ cubes lies in $Q$,
a case already discarded. Thus the two sets are disjoint.

For each remaining four-tuple of cubes, the bound
\eqref{eq:four-base-point-bound} no longer depends on the labels
or atom indices. On the two selection sets the number of index
pairs is $r_jr_k$, by \eqref{eq:selected-atom-pair-count}.
The disjoint label-coordinate sets give the average $c_jc_k$
by \eqref{eq:cross-level-label-count}. Hence summing indices and
then averaging labels contributes exactly $c_jc_kr_jr_k$.

The product domains for distinct four-tuples are disjoint.
We may drop the restrictions $\widetilde Q_i\not\subset Q$ after
this averaging and enlarge their union to
$Q^L\times Q^R\times\widetilde Q^L\times\widetilde Q^R$.
We obtain
\begin{align*}
 &\Av_\omega\Av_{a,b}
    [\mathcal N_{j,Q}(x)\mathcal N_{k,\widetilde Q}(x)]\\
 &\leq Cc_jc_kr_jr_kR_0^{-4d}
   \one_{\mathcal O_{j,Q}}(x)\one_{\mathcal O_{k,\widetilde Q}}(x)\\
 &\quad\times\left(\iint_{Q^L\times Q^R}
    \one_{\mathcal G_j(x)}(u,v)\,du\,dv\right)
   \left(\iint_{\widetilde Q^L\times\widetilde Q^R}
    \one_{\mathcal G_k(x)}(\widetilde u,\widetilde v)
                        \,d\widetilde u\,d\widetilde v\right)\\
 &\leq Cc_jc_kr_jr_kR_0^{-4d}s^2R_j^{2d-1}R_k^{2d-1}
   \one_{\mathcal O_{j,Q}}(x)\one_{\mathcal O_{k,\widetilde Q}}(x)\\
 &=Cc_jc_k\theta^4
   \one_{\mathcal O_{j,Q}}(x)\one_{\mathcal O_{k,\widetilde Q}}(x).
\end{align*}
The second inequality uses \eqref{eq:one-scale} at the two radii.
For the last equality, \eqref{eq:count-identity} gives
$$
 r_jR_0^{-2d}sR_j^{2d-1}=\theta^2,
 \qquad r_kR_0^{-2d}sR_k^{2d-1}=\theta^2.
$$

At each level the containing regions are disjoint
(Lemma~\ref{lem: O disjoint}). Summing over $Q,\widetilde Q$ and
using \eqref{eq:output-overlap}, we conclude that
\begin{align*}
 \mathrm I_{j,k}
 &\leq Cc_jc_k\theta^4\int
   \bigg(\sum_{Q\in\mathcal Q_j}\one_{\mathcal O_{j,Q}}\bigg)
   \bigg(\sum_{\widetilde Q\in\mathcal Q_k}
                          \one_{\mathcal O_{k,\widetilde Q}}\bigg)\,dx
 \leq Cc_jc_k\theta^4\sum_{Q\in\mathcal Q_j}|\mathcal O_{j,Q}|
 \leq Cc_jc_k\theta^4.
\end{align*}
This proves Case 1.

\smallskip

\noindent\emph{Case 2: the same level and different cubes.}

By \eqref{eq:region-strip-containment},
$\mathcal N_{j,Q}$ is supported in $\mathcal O_{j,Q}$.
Lemma \ref{lem: O disjoint} gives
$
 \mathcal N_{j,Q}(x)\mathcal N_{j,\widetilde Q}(x)=0
 \  (Q\ne\widetilde Q)
$
for every choice of labels and base points.
Hence,
$$
 \mathrm{II}_{j,Q,\widetilde Q}
 =\Av_\omega\Av_{a,b}\int_{\mathbb R^d}
       \mathcal N_{j,Q}(x)\mathcal N_{j,\widetilde Q}(x)\,dx
 =0,
$$
and hence
$
 \sum_{Q,\widetilde Q\in\mathcal Q_j:\,Q\ne\widetilde Q}
 \mathrm{II}_{j,Q,\widetilde Q}=0$.
This proves the estimate in Case 2.

\smallskip
\noindent\emph{Case 3: the same level and the same cube.}

Fix $Q\in\mathcal Q_j$. By \eqref{eq:cube-counting-functions},
\begin{align*}
 \mathrm{III}_{j,Q}
 &=\Av_\omega\Av_{a,b}\int_{\mathbb R^d}
       \bigg(\sum_{e\in\mathcal E_{j,Q}}\one_{S_e}(x)\bigg)^2\,dx
 =\Av_\omega\Av_{a,b}\sum_{e,\widetilde e\in\mathcal E_{j,Q}}
       \int_{\mathbb R^d}
       \one_{S_e}(x)\one_{S_{\widetilde e}}(x)\,dx\\
 &=\Av_\omega\Av_{a,b}\sum_{e,\widetilde e\in\mathcal E_{j,Q}}
       |S_e\cap S_{\widetilde e}|.
\end{align*}
The sums are finite, and the double sum is over ordered pairs,
including $e=\widetilde e$.
Since $Q_1,\widetilde Q_1\subset Q^L$,
$Q_2,\widetilde Q_2\subset Q^R$, and $Q^L\cap Q^R=\varnothing$,
we therefore divide the sum into three parts:
$$
\begin{aligned}
 \text{Part (i)}\quad&
 Q_1=\widetilde Q_1,\quad Q_2=\widetilde Q_2;\\
 \text{Part (ii)}\quad&
 Q_1=\widetilde Q_1,\ Q_2\ne\widetilde Q_2
 \quad\text{or}\quad
 Q_1\ne\widetilde Q_1,\ Q_2=\widetilde Q_2;\\
 \text{Part (iii)}\quad&
 Q_1\ne\widetilde Q_1,\quad Q_2\ne\widetilde Q_2.
\end{aligned}
$$

\smallskip
\noindent\emph{Part {\rm(i):} $Q_1=\widetilde Q_1,\ Q_2=\widetilde Q_2$}

By Lemma~\ref{lem:separation}, $S_e\cap S_{\widetilde e}$ is empty
unless $(r,t)=(\widetilde r,\widetilde t)$.
Thus, by \eqref{eq:first-cube-integral},
this part is
$$
 \Av_\omega\Av_{a,b}\sum_{e\in\mathcal E_{j,Q}}|S_e|
 =\Av_\omega\Av_{a,b}\int_{\mathbb R^d}\mathcal N_{j,Q}(x)\,dx
 \leq C{c_j\theta^2}R_j^d.
$$

\smallskip
\noindent\emph{Part {\rm(ii):} $ Q_1=\widetilde Q_1,\ Q_2\ne\widetilde Q_2
 \ \text{or}\ 
 Q_1\ne\widetilde Q_1,\ Q_2=\widetilde Q_2$}

First suppose $Q_1=\widetilde Q_1$ and
$Q_2\ne\widetilde Q_2$. Fix the labels and
$e=((Q_1,r),(Q_2,t))\in\mathcal E_{j,Q}$.
Keep $a_{Q_1}$ and $b_{Q_2}$ fixed, so that $S_e$ is fixed.
For every selected partner $\widetilde Q_2$, the selection
conditions give $Z_{Q_1}-Z_{\widetilde Q_2}=j$. Since $Z_{Q_1}-Z_{Q_2}=j$, we then have
$
 Z_{\widetilde Q_2}=Z_{Q_2}$, which further gives
$ n_{\widetilde Q_2}=n_{Q_2}.
$
Hence all these partners have the same atom-index ranges
$$
 0\leq\widetilde r<m_{Q_1},\qquad
 0\leq\widetilde t<n_{Q_2}.
$$

Fix one such $(\widetilde r,\widetilde t)$ and write
$\widetilde e=((Q_1,\widetilde r),(\widetilde Q_2,\widetilde t))$.
The left endpoint $u_{Q_1,\widetilde r}$ does not depend on
$\widetilde Q_2$.
For each selected $\widetilde Q_2$, change variables
$v=b_{\widetilde Q_2}+{2^4}\widetilde t s\,e_d$.
By \eqref{eq:centers-diameter}, the translated domain lies in
$\widetilde Q_2$.
For $x\in S_e\subset\mathcal O_{j,Q}$,
\eqref{eq:region-strip-containment} gives
\begin{align}\label{part ii sum widetilde Q2}
 \sum_{\substack{\widetilde Q_2\in\mathcal Q_0:\ \widetilde Q_2\ne Q_2,\\ (Q_1,\widetilde Q_2)\in\mathcal P_{j,Q}}}
 \fint_{{\frac18\widetilde Q_2}}
 \one_{S_{\widetilde e}}(x)
 \,db_{\widetilde Q_2}
 &\leq CR_0^{-d}
 \sum_{\substack{\widetilde Q_2\in\mathcal Q_0:\ \widetilde Q_2\ne Q_2,
 \\ (Q_1,\widetilde Q_2)\in\mathcal P_{j,Q}}}
 \int_{\widetilde Q_2}
 \one_{\mathcal G_j(x)}(u_{Q_1,\widetilde r},v)\,dv\\
 &\leq CR_0^{-d}\int_{Q^R}
 \one_{\mathcal G_j(x)}(u_{Q_1,\widetilde r},v)\,dv
 \leq CR_0^{-d}sR_j^{d-1}\notag\\
& =C{\mathcal L^{dj}}\frac{s}{R_j}.\notag
\end{align}
The second inequality combines integrals of the same function
over disjoint cubes contained in $Q^R$.
The last inequality uses \eqref{eq:one-variable-coarea}, since
$u_{Q_1,\widetilde r}\in Q_1\subset Q^L$.

The preceding bound is uniform in
$0\leq\widetilde r<m_{Q_1}$ and
$0\leq\widetilde t<n_{Q_2}$. Hence, by \eqref{part ii sum widetilde Q2},
\begin{align}\label{part ii partial sum}
 &\sum_{\widetilde r=0}^{m_{Q_1}-1}
  \sum_{\widetilde t=0}^{n_{Q_2}-1}
  \sum_{\substack{\widetilde Q_2\in\mathcal Q_0:\ \widetilde Q_2\ne Q_2,\\ (Q_1,\widetilde Q_2)\in\mathcal P_{j,Q}}}
 \fint_{{\frac18\widetilde Q_2}}
 |S_e\cap S_{\widetilde e}|
 \,db_{\widetilde Q_2}\\
 &=\sum_{\widetilde r=0}^{m_{Q_1}-1}
   \sum_{\widetilde t=0}^{n_{Q_2}-1}
   \int_{S_e}\Biggl[
   \sum_{\substack{\widetilde Q_2\in\mathcal Q_0:\ \widetilde Q_2\ne Q_2,\\ (Q_1,\widetilde Q_2)\in\mathcal P_{j,Q}}}
   \fint_{{\frac18\widetilde Q_2}}
   \one_{S_{\widetilde e}}(x)
   \,db_{\widetilde Q_2}
   \Biggr]\,dx\notag\\
 &\leq\sum_{\widetilde r=0}^{m_{Q_1}-1}
       \sum_{\widetilde t=0}^{n_{Q_2}-1}
       \int_{S_e}C{\mathcal L^{dj}}\frac{s}{R_j}\,dx
 =C{\mathcal L^{dj}}\frac{s}{R_j}m_{Q_1}n_{Q_2}|S_e|\notag\\
 &=C{\mathcal L^{dj}}r_j\frac{s}{R_j}|S_e|
 =C\theta^2{\mathcal L^{-dj}}|S_e|,\notag
\end{align}
where the last two equalities use $m_{Q_1}n_{Q_2}=r_j$
and \eqref{eq:count-identity}, respectively.
Averaging \eqref{part ii partial sum} over the two base points
of $e$ gives
$$
 \Av_{a,b}\sum_{\substack{\widetilde e\in\mathcal E_{j,Q}:\\
       \widetilde Q_1=Q_1,\ \widetilde Q_2\ne Q_2}}
 |S_e\cap S_{\widetilde e}|
 \leq C\theta^2\mathcal L^{-dj}\Av_{a,b}|S_e|.
$$
Summing over $e$ and then averaging the labels, we obtain
\begin{align*}
 &\Av_\omega\Av_{a,b}
 \sum_{\substack{
 e,\widetilde e\in\mathcal E_{j,Q}:\\
 \widetilde Q_1=Q_1,\ \widetilde Q_2\ne Q_2}}
 |S_e\cap S_{\widetilde e}|
 =\Av_\omega \bigg(\sum_{e\in\mathcal E_{j,Q}}\Av_{a,b} 
 \sum_{\substack{
 \widetilde e\in\mathcal E_{j,Q}:\\
 \widetilde Q_1=Q_1,\ \widetilde Q_2\ne Q_2}}
 |S_e\cap S_{\widetilde e}|\bigg)
\\
 &\qquad\leq C\theta^2{\mathcal L^{-dj}}
       \Av_\omega \sum_{e\in\mathcal E_{j,Q}}\Av_{a,b} |S_e|
=C\theta^2{\mathcal L^{-dj}}
       \Av_\omega\Av_{a,b}\sum_{e\in\mathcal E_{j,Q}}|S_e|\\
&\qquad =C\theta^2{\mathcal L^{-dj}}
       \Av_\omega\Av_{a,b}\int_{\mathbb R^d}\mathcal N_{j,Q}(x)\,dx
 \leq C\theta^2{\mathcal L^{-dj}}{c_j\theta^2}R_j^d
 \leq C{c_j\theta^2}R_j^d.
\end{align*}
The last line uses \eqref{eq:first-cube-integral}
and $\theta^2{\mathcal L^{-dj}}\leq1$.

If $Q_2=\widetilde Q_2$ and $Q_1\ne\widetilde Q_1$,
the selection conditions give
$Z_{\widetilde Q_1}=Z_{Q_1}$ and hence
$m_{\widetilde Q_1}=m_{Q_1}$ and $n_{\widetilde Q_1}=n_{Q_1}$.
The same argument, with $u$ and $v$ interchanged,
gives the bound $C\theta^2{\mathcal L^{-dj}}{c_j\theta^2}R_j^d$.

This proves the required bound when exactly one endpoint cube
is shared.

\smallskip
\noindent\emph{Part {\rm(iii):} $Q_1\ne\widetilde Q_1,\ \ Q_2\ne\widetilde Q_2$}

We first consider the terms for which
$
 K(Q_1,Q_2)\cap K(\widetilde Q_1,\widetilde Q_2)=\varnothing.
$
Apply \eqref{eq:four-base-point-bound} with $k=j$ and
$\widetilde Q=Q$. As in Case 1, enlarge the translated base-point
domains before summing the atom indices and averaging the labels.
The number of index pairs on the selection sets is $r_j^2$, and
the label average is $c_j^2$. Combining the disjoint product domains
and applying \eqref{eq:one-scale} gives
\begin{align}\label{part iii 1}
 \Av_\omega\Av_{a,b}
 \sum_{\substack{e,\widetilde e\in\mathcal E_{j,Q}:\\
    Q_1\ne\widetilde Q_1,\ Q_2\ne\widetilde Q_2\\
    K(Q_1,Q_2)\cap K(\widetilde Q_1,\widetilde Q_2)=\varnothing}}
 |S_e\cap S_{\widetilde e}|
 &\leq Cc_j^2r_j^2R_0^{-4d}\int_{\mathcal O_{j,Q}}
   \left(\iint_{Q^L\times Q^R}
      \one_{\mathcal G_j(x)}(u,v)\,du\,dv\right)^2dx
 \notag\\
 &\leq Cc_j^2r_j^2R_0^{-4d}s^2R_j^{4d-2}R_j^d
 =Cc_j^2\theta^4R_j^d.
\end{align}
Here $|\mathcal O_{j,Q}|\lesssim R_j^d$ bounds the measure,
and \eqref{eq:count-identity} gives the last equality.

It remains to consider the terms for which
$K(Q_1,Q_2)\cap K(\widetilde Q_1,\widetilde Q_2)\ne\varnothing$.
Choose a cube $P$ in this intersection. By definition,
$P=Q_i^{(h)}=\widetilde Q_\ell^{(h)}$ for some $0\leq h\leq j-1$ and
$i,\ell\in\{1,2\}$.
Taking the level-$(j-1)$ ancestor gives
$Q_i^{(j-1)}=\widetilde Q_\ell^{(j-1)}$.
By Lemma~\ref{lem:first-common-ancestor}, no level-$(j-1)$
subcube of $Q$ meets both $Q^L$ and $Q^R$.
Hence $i=\ell$, so we have
$$
 Q_1^{(j-1)}=\widetilde Q_1^{(j-1)}
 \quad\text{or}\quad
 Q_2^{(j-1)}=\widetilde Q_2^{(j-1)}.
$$

We first estimate the terms satisfying $Q_1^{(j-1)}=\widetilde Q_1^{(j-1)}$.
Note that
$$
\begin{aligned}
 \Av_\omega\bigl[m_{Q_1}n_{Q_2}m_{\widetilde Q_1}n_{\widetilde Q_2}
 \one_{\{Z_{Q_1}-Z_{Q_2}=j\}}
 \one_{\{Z_{\widetilde Q_1}-Z_{\widetilde Q_2}=j\}}\bigr]
 &=r_j^2\,\Av_\omega\bigl[\one_{\{Z_{Q_1}-Z_{Q_2}=j\}}
 \one_{\{Z_{\widetilde Q_1}-Z_{\widetilde Q_2}=j\}}\bigr]\\
 &\leq r_j^2\,\Av_\omega\one_{\{Z_{Q_1}-Z_{Q_2}=j\}}
 =c_jr_j^2.
\end{aligned}
$$
Applying \eqref{eq:four-base-point-bound} and the label bound just obtained, as in \eqref{part iii 1}, gives
\begin{align*}
 &\Av_\omega\Av_{a,b}
 \sum_{\substack{
 e,\widetilde e\in\mathcal E_{j,Q}:\\
 Q_1\ne\widetilde Q_1,\ Q_2\ne\widetilde Q_2\\
 Q_1^{(j-1)}=\widetilde Q_1^{(j-1)}}}
 |S_e\cap S_{\widetilde e}|\\
 &\leq CR_0^{-4d}
 \int_{\mathcal O_{j,Q}}
 \sum_{\substack{
 (Q_1,Q_2),(\widetilde Q_1,\widetilde Q_2)\in\mathcal P_Q:\\
 Q_1\ne\widetilde Q_1,\ Q_2\ne\widetilde Q_2,\
 Q_1^{(j-1)}=\widetilde Q_1^{(j-1)}}}
 c_jr_j^2 \\
 &\qquad\qquad\times
 \left(\iint_{Q_1\times Q_2}
 \one_{\mathcal G_j(x)}(u,v)
 \,du\,dv\right)
 \left(\iint_{\widetilde Q_1\times\widetilde Q_2}
 \one_{\mathcal G_j(x)}(\widetilde u,\widetilde v)
 \,d\widetilde u\,d\widetilde v\right)\,dx\\
 &\leq Cc_jr_j^2R_0^{-4d}
 \int_{\mathcal O_{j,Q}}
 \sum_{Q'\in\mathcal Q_{j-1}:\,Q'\subset Q}
 \iint_{(Q'\cap Q^L)\times(Q'\cap Q^L)}
 \left(\int_{Q^R}
 \one_{\mathcal G_j(x)}(u,v)\,dv\right)\\
 &\qquad\qquad\times
 \left(\int_{Q^R}
 \one_{\mathcal G_j(x)}(\widetilde u,\widetilde v)\,d\widetilde v\right)
 \,du\,d\widetilde u\,dx\\
 &\leq Cc_jr_j^2R_0^{-4d}
 |\mathcal O_{j,Q}|s^2R_j^{2d-2}
 \sum_{Q'\in\mathcal Q_{j-1}:\,Q'\subset Q}|Q'|^2\\
 &\leq \frac{Cc_j}{{\mathcal L^d}}
 r_j^2R_0^{-4d}s^2R_j^{4d-2}R_j^d
 =\frac{Cc_j}{{\mathcal L^d}}\theta^4R_j^d.
\end{align*}
The second inequality groups the disjoint product domains
by the common ancestor
$Q'=Q_1^{(j-1)}=\widetilde Q_1^{(j-1)}$ and enlarges their union.
The third inequality uses \eqref{eq:one-variable-coarea}
for each right endpoint.
The last two steps use \eqref{eq:count-identity},
$|\mathcal O_{j,Q}|\leq CR_j^d$ and
$$
 \sum_{Q'\in\mathcal Q_{j-1}:\,Q'\subset Q}|Q'|^2
 ={\mathcal L^d}\left(\frac{R_j}{\mathcal L}\right)^{2d}
 ={\frac{R_j^{2d}}{\mathcal L^d}}.
$$

The same argument, with $u$ and $v$ interchanged, gives
the same bound when
$Q_2^{(j-1)}=\widetilde Q_2^{(j-1)}$.
Moreover, Lemma~\ref{lem:local-limit} and ${\mathcal L^d}\geq N\geq j$
give
$$
 \frac{c_j}{{\mathcal L^d}}
 =c_j^2{\frac1{\mathcal L^dc_j}}
 \leq Cc_j^2{\frac{j}{\mathcal L^d}}
 \leq Cc_j^2.
$$
Consequently,
\begin{align*}
 \Av_\omega\Av_{a,b}
 \sum_{\substack{
 e,\widetilde e\in\mathcal E_{j,Q}:\\
 Q_1\ne\widetilde Q_1,\ Q_2\ne\widetilde Q_2\\
 K(Q_1,Q_2)\cap K(\widetilde Q_1,\widetilde Q_2)\ne\varnothing}}
 |S_e\cap S_{\widetilde e}|
 &\leq
 \sum_{i=1}^2\Av_\omega\Av_{a,b}
 \sum_{\substack{
 e,\widetilde e\in\mathcal E_{j,Q}:\\
 Q_1\ne\widetilde Q_1,\ Q_2\ne\widetilde Q_2\\
 Q_i^{(j-1)}=\widetilde Q_i^{(j-1)}}}
 |S_e\cap S_{\widetilde e}|
 \\
 &\leq \frac{Cc_j}{{\mathcal L^d}}\theta^4R_j^d
 \leq C{c_j^2\theta^4}R_j^d.
\end{align*}

Combining all parts of Case 3 proves
$ \mathrm {III}_{j,Q}
 \leq C({c_j\theta^2}+{c_j^2\theta^4})R_j^d.$
This completes the three estimates in \eqref{3cases}.

Combining \eqref{eq:second-integral-decomposition} and \eqref{3cases} and recalling
that $\#\mathcal Q_j=R_j^{-d}$, 
we have
\begin{align*}
 \Av_\omega\Av_{a,b}\int_{\mathbb R^d}\mathcal N^2(x)\,dx
 &\leq
 2C\sum_{1\leq j<k\leq N}{c_jc_k\theta^4}
 +C\sum_{j=1}^{N}
       \sum_{Q\in\mathcal Q_j}({c_j\theta^2}+{c_j^2\theta^4})R_j^d\\
 &{=C\theta^2\sum_{j=1}^{N}c_j
   +C\theta^4\left(
       2\sum_{1\leq j<k\leq N}c_jc_k
       +\sum_{j=1}^{N}c_j^2
      \right)}\\
 &{=C\left[
     \theta^2\sum_{j=1}^{N}c_j
     +\theta^4\left(\sum_{j=1}^{N}c_j\right)^2
   \right]}\\
 &{\leq C({\rho}+{\rho}^2)}.
\end{align*}
This proves \eqref{eq:second-moment} and completes the proof
of Proposition~\ref{prop:moments}.
\end{proof}

\section{Completion of the bilinear proof}
\label{sec:proof-completion}

We first prove a sign estimate for fixed labels and base points.
We then choose labels and base points so that the union of the
selected regions has measure bounded below, and compare the norms.

\subsection{A common choice of signs}

Fix the labels and base points.
Let
$$
\mathcal I_f=\{(Q_1,r):Q_1\in \mathcal Q_0, 0\le r< m_{Q_1}\}
\qquad{\rm and}\qquad
\mathcal I_g=\{(Q_2,t):Q_2\in \mathcal Q_0, 0\le t< n_{Q_2}\} 
$$
be the finite index sets of the atoms in $f$ and $g$, respectively.
For $i=(Q_1,r)\in\mathcal I_f$ and $k=(Q_2,t)\in\mathcal I_g$,
write $F_i=F_{Q_1,r}$ and $G_k=G_{Q_2,t}$.
Then \eqref{eq:signed-functions} becomes
$$ 
f=\sum_{i\in\mathcal I_f}\varepsilon_iF_i
 \qquad{\rm and}\qquad
 g=\sum_{k\in\mathcal I_g}\delta_kG_k.
$$
We will choose these signs after fixing one atom pair at each
point of the union
\begin{equation}\label{eq:union-definition}
 U:=\{x:\mathcal N(x)>0\}
   =\bigcup_{j=1}^{N}\bigcup_{e\in\mathcal E_j}S_e.
\end{equation}

The family of regions is finite. Fix an ordering and, for each
$x\in U$, let $e(x)$ be the first index for which $x\in S_{e(x)}$.
Write
$$
 e(x)=(i(x),k(x))\in\mathcal E_{j(x)},
$$
where $j(x)$ is the level of $e(x)$.
Each $S_e$ is Borel, so $e(x)$ and $j(x)$ are measurable.
By Proposition~\ref{cor:positive-region}, with the constant
$c_0>0$ there, which is independent of $N$, the labels, and the base
points,
\begin{equation}\label{eq:chosen-pair-lower-bound}
 \A_{R_{j(x)}}(F_{i(x)},G_{k(x)})(x)\geq c_0,
 \qquad x\in U.
\end{equation}

For $H:\{-1,1\}^{\mathcal I_f}\to\mathbb R$ and
$H':\{-1,1\}^{\mathcal I_g}\to\mathbb R$, define the normalized
finite sums
$$
 \Av_{\varepsilon}H
 =2^{-\#\mathcal I_f}
   \sum_{\varepsilon\in \{-1,1\}^{\mathcal I_f}}H(\varepsilon),\qquad
 \Av_{\delta}H'
 =2^{-\#\mathcal I_g}
   \sum_{\delta\in
                   \{-1,1\}^{\mathcal I_g}}H'(\delta).
$$
Bilinearity gives
\begin{equation}\label{eq:rademacher-sum}
 \A_{R_{j(x)}}(f,g)(x)
 =\sum_{i\in\mathcal I_f}\sum_{k\in\mathcal I_g}
 \varepsilon_i\delta_k
 \A_{R_{j(x)}}(F_i,G_k)(x).
\end{equation}
This sum contains every pair of atoms at the chosen radius.

\begin{lemma}\label{lem:sign-estimates}
For every $x\in U$,
\begin{align}
 \Av_{\varepsilon}\Av_{\delta}
 |\A_{R_{j(x)}}(f,g)(x)|^2
 &\geq c_0^2,
 \label{eq:rademacher-second-lower}\\
 \Av_{\varepsilon}\Av_{\delta}
 |\A_{R_{j(x)}}(f,g)(x)|^4
 &\leq9\left(
 \Av_{\varepsilon}\Av_{\delta}
 |\A_{R_{j(x)}}(f,g)(x)|^2\right)^2.
 \label{eq:rademacher-fourth}
\end{align}
\end{lemma}

\begin{proof}
Fix $x\in U$. For $i,i'\in\mathcal I_f$ and
$k,k'\in\mathcal I_g$, direct summation gives
$$
 \Av_{\varepsilon}\Av_{\delta}
 (\varepsilon_i\varepsilon_{i'}\delta_k\delta_{k'})
 =\one_{\{i=i'\}}\one_{\{k=k'\}}.
$$
Expanding the square in \eqref{eq:rademacher-sum}, we obtain
\begin{align*}
 \Av_{\varepsilon}\Av_{\delta}
 |\A_{R_{j(x)}}(f,g)(x)|^2
 &=\sum_{i\in\mathcal I_f}\sum_{k\in\mathcal I_g}
   |\A_{R_{j(x)}}(F_i,G_k)(x)|^2
 \geq|\A_{R_{j(x)}}(F_{i(x)},G_{k(x)})(x)|^2
 \geq c_0^2.
\end{align*}
The last step uses \eqref{eq:chosen-pair-lower-bound}.
This proves \eqref{eq:rademacher-second-lower}.

To prove \eqref{eq:rademacher-fourth}, first fix $\delta$ and put
$$
 a_i=\sum_{k\in\mathcal I_g}\delta_k
       \A_{R_{j(x)}}(F_i,G_k)(x),
 \qquad i\in\mathcal I_f.
$$
These numbers are real. In the fourth-power expansion, a term
has zero average whenever some sign occurs an odd number of times.
Hence
\begin{align*}
 \Av_\varepsilon\left(\sum_i\varepsilon_i a_i\right)^4
 &=3\left(\sum_i a_i^2\right)^2-2\sum_i a_i^4
 \leq3\left(\sum_i a_i^2\right)^2.
\end{align*}
Applying the same identity to the $\delta$ sum defining each $a_i$
gives
$$
 \Av_\delta a_i^4
 \leq3\left(\sum_{k\in\mathcal I_g}
       |\A_{R_{j(x)}}(F_i,G_k)(x)|^2\right)^2.
$$
We now expand the square and apply Cauchy--Schwarz to each pair
of indices:
\begin{align*}
 \Av_\delta\left(\sum_i a_i^2\right)^2
 &=\sum_{i,i'}\Av_\delta(a_i^2a_{i'}^2)
 \leq\sum_{i,i'}
       (\Av_\delta a_i^4)^{1/2}
       (\Av_\delta a_{i'}^4)^{1/2}\\
 &\leq3\left(\sum_{i\in\mathcal I_f}
                 \sum_{k\in\mathcal I_g}
       |\A_{R_{j(x)}}(F_i,G_k)(x)|^2\right)^2.
\end{align*}
Combining these estimates with the quadratic identity above, we obtain
\begin{align*}
 \Av_{\varepsilon}\Av_{\delta}
 |\A_{R_{j(x)}}(f,g)(x)|^4
 &\leq3\Av_\delta\left(\sum_i a_i^2\right)^2
 \leq9\left(\sum_{i\in\mathcal I_f}
                 \sum_{k\in\mathcal I_g}
       |\A_{R_{j(x)}}(F_i,G_k)(x)|^2\right)^2\\
 &=9\left(\Av_{\varepsilon}\Av_{\delta}
       |\A_{R_{j(x)}}(f,g)(x)|^2\right)^2.
\end{align*}
This proves \eqref{eq:rademacher-fourth}. See also \cite{Bonami}.
\end{proof}

The above estimates yield a common choice of signs. 

\begin{lemma}\label{lem:sign-selection}
For every $x\in U$,
\begin{equation}\label{eq:pointwise-sign-probability}
 \Av_{\varepsilon}\Av_{\delta}
 \one_{\{|\A_{R_{j(x)}}(f,g)(x)|>{c_0\over\sqrt2}\}}\geq\frac1{36}.
\end{equation}
Consequently, one choice of all signs satisfies
\begin{equation}\label{eq:global-sign-level}
 \left|\left\{x\in U:
 |\A_{R_{j(x)}}(f,g)(x)|>{c_0\over\sqrt2}\right\}\right|\geq\frac{|U|}{36}.
\end{equation}
\end{lemma}

\begin{proof}
Fix $x\in U$ and put
$$
 \Gamma_x=\left\{(\varepsilon,\delta)\in \{-1,1\}^{\mathcal I_f}\times
                   \{-1,1\}^{\mathcal I_g}:
 |\A_{R_{j(x)}}(f,g)(x)|>{c_0\over\sqrt2}\right\}.
$$
Split the finite sum over $\Gamma_x$ and its complement, then
apply Cauchy--Schwarz and Lemma~\ref{lem:sign-estimates}:
\begin{align*}
 \Av_{\varepsilon}\Av_{\delta}
 |\A_{R_{j(x)}}(f,g)(x)|^2
 &\leq\frac{c_0^2}{2}
 +\Av_{\varepsilon}\Av_{\delta}
 \left[\one_{\Gamma_x}|\A_{R_{j(x)}}(f,g)(x)|^2\right]\\
 &\leq\frac{c_0^2}{2}
 +\left(\Av_{\varepsilon}\Av_{\delta}\one_{\Gamma_x}\right)^{1/2}
  \left(\Av_{\varepsilon}\Av_{\delta}
  |\A_{R_{j(x)}}(f,g)(x)|^4\right)^{1/2}\\
 &\leq\left(\frac12
       +3\left(\Av_{\varepsilon}\Av_{\delta}
                         \one_{\Gamma_x}\right)^{1/2}\right)
       \Av_{\varepsilon}\Av_{\delta}
       |\A_{R_{j(x)}}(f,g)(x)|^2.
\end{align*}
By \eqref{eq:rademacher-second-lower}, the last finite average
is positive. Dividing by it gives
$$
 \Av_{\varepsilon}\Av_{\delta}\one_{\Gamma_x}\geq\frac1{36}.
$$
This proves \eqref{eq:pointwise-sign-probability}.

The choices $e(x)$ and $j(x)$ were made before the signs.
Interchanging the finite sign sum and the integral in $x$, we get
\begin{align*}
 \Av_{\varepsilon}\Av_{\delta}
 \Big|\Big\{x\in U:|\A_{R_{j(x)}}(f,g)(x)|>{c_0\over\sqrt2}\Big\}\Big|
 &=\int_U\Av_{\varepsilon}\Av_{\delta}
 \one_{\{|\A_{R_{j(x)}}(f,g)(x)|>{c_0\over\sqrt2}\}}\,dx
 \geq\frac{|U|}{36}.
\end{align*}
At least one of the finitely many sign choices reaches this
average. This proves \eqref{eq:global-sign-level} and the lemma.
\end{proof}

\subsection{A positive-measure union and the norm comparison}

We first choose the labels and base points using the two estimates
in Proposition~\ref{prop:moments}.  We then apply
Lemma~\ref{lem:sign-selection} to that choice and compare the norms.

\begin{corollary}\label{cor:union}
Let $U$ be defined by \eqref{eq:union-definition}, with $\theta$ as
in \eqref{eq:theta-choice}. There are a constant $c_U>0$,
independent of $N$, and a choice of the cube labels and base points
such that
\begin{equation}\label{eq:union-lower}
 |U|\geq c_U.
\end{equation}
\end{corollary}

\begin{proof}
Lemma~\ref{lem:local-limit} gives
$$
 \sum_{j=1}^{N}c_j\approx\log(N+1).
$$
Thus \eqref{eq:theta-choice} gives $\theta^2\sum_{j=1}^{N}c_j\approx{\rho}$.
All the integrals below are finite, since the sums are finite and
the regions lie in a fixed bounded neighborhood of $[0,1]^d$.
Apply Cauchy--Schwarz to the normalized product sums and integrals
defining $\Av_\omega\Av_{a,b}$, together with Lebesgue integration in $x$:
\begin{align*}
 \bigl(\Av_\omega\Av_{a,b}\int\mathcal N\bigr)^2
 &=\bigl(\Av_\omega\Av_{a,b}\int\one_{\{\mathcal N>0\}}\mathcal N\bigr)^2
 \leq \Av_\omega\Av_{a,b}|\{x:\mathcal N(x)>0\}|\;
        \Av_\omega\Av_{a,b}\int\mathcal N^2.
\end{align*}
Proposition~\ref{prop:moments} therefore yields
$$
 \Av_\omega\Av_{a,b}|\{\mathcal N>0\}|
 \geq c\frac{\rho^2}{\rho+{\rho}^2}
 =:2c_U.
$$
The constant $c_U>0$ does not depend on $N$, since
${\rho}$ was fixed earlier.
If every choice of labels and base points had
$|\{\mathcal N>0\}|<c_U$, its normalized average would be at most
$c_U<2c_U$, a contradiction.  Hence a choice satisfying
\eqref{eq:union-lower} exists.  Fix it for the rest of the proof.
\end{proof}

Apply Lemma~\ref{lem:sign-selection} to the labels and base points
just fixed, and choose the signs given by \eqref{eq:global-sign-level}.
Recall that $E_N=\{R_1,\ldots,R_N\}$. Since $j(x)\in\{1,\ldots,N\}$,
\eqref{eq:global-sign-level} and Corollary~\ref{cor:union} give
\begin{equation}\label{eq:maximal-level}
 \Big|\Big\{x\in \mathbb R^d:\B_{E_N}(f,g)(x)>{c_0\over\sqrt2}\Big\}\Big|
 \geq \frac{c_U}{36}.
\end{equation}

The Hardy bounds \eqref{eq:H1-bound} hold for this choice, and
\eqref{eq:maximal-level} ensures that neither function is zero.
Combining \eqref{eq:weak-quasinorm}, \eqref{eq:maximal-level},
\eqref{eq:H1-bound}, and \eqref{eq:theta-choice}, we obtain
\begin{align*}
 \frac{\|\B_{E_N}(f,g)\|_{L^{1/2,\infty}(\mathbb R^d)}}
 {\|f\|_{H^1(\mathbb R^d)}\|g\|_{H^1(\mathbb R^d)}}
 &\geq \frac{c_0(c_U/36)^2}{C^2\sqrt2\theta^2}
 =\frac{c_0c_U^2}{36^2\sqrt2C^2\rho}\log(N+1).
\end{align*}
This proves Theorem~\ref{thm:main} for $d\geq2$, with
$c_d=c_0c_U^2/(36^2\sqrt2C^2\rho)$.

\section{The case \texorpdfstring{$d=1$}{d=1}}
\label{sec:dimension-one}

We give the changes needed to prove Theorem~\ref{thm:main} for $d=1$.
The averaging set is still the circle $\mathbb S^1\subset\mathbb R^2$.
Only the regions in the $x$ variable change.
Use the cube family in Subsection~\ref{sec:cube-tree} and the atom
construction in Section~\ref{sec:construction} with $d=1$.
Every cube $Q$ is now an interval, and $e_d=+1$.
There are no horizontal variables $x'$ or $z_e'$.
The label sums, atom numbers, coefficients, and centers retain
their definitions. In particular, \eqref{eq:coefficient-product}
becomes $\mu_{Q_1}\nu_{Q_2}=sR_j$ for a selected pair.

\smallskip
\noindent\emph{Regions and local lower bounds.}
For $e=((Q_1,r),(Q_2,t))\in\mathcal E_j$, use
$u=u_{Q_1,r}$, $v=v_{Q_2,t}$, and $z_e=(u+v)/2$ as before.
The cap in \eqref{eq:atom-pair-region} is
$\{+1\}\subset\mathbb S^0$, with $\mathcal H^0(\{+1\})=1$.
Thus
$$
 S_e=\left[
 z_e+\sqrt{r_{u,v}^2-(h_0/2)sR_j},\
 z_e+\sqrt{r_{u,v}^2+(h_0/2)sR_j}
 \right],
$$
where $r_{u,v}^2=R_j^2/2-(u-v)^2/4$.
Since $r_{u,v}\approx R_j$, rationalizing the difference of the
two endpoints gives $|S_e|\approx s$.
The coarea calculation \eqref{eq:exact-coarea} remains valid for
the two variables $y_1,y_2$, and Lemma~\ref{lem:J-positive} with $k=2$
holds for every $d\geq1$. As in Proposition~\ref{cor:positive-region},
since $\mu_{Q_1}\nu_{Q_2}=sR_j$, we obtain
$$
 \A_{R_j}(F_{Q_1,r},G_{Q_2,t})(x)
 \geq c_0,
 \qquad x\in S_e.
$$

\smallskip
\noindent\emph{Disjointness.}
The upper graph is the point
$\Gamma_{u,v}=z_e+r_{u,v}$, and $H=r_{u,v}\geq R_j/2$.
Here $e_1=e_d=1$, so the centers of the auxiliary intervals no
longer have the same last coordinate, as they did for $d\geq2$.
Instead, we use $|u-v|\leq{2^{-3}} R_j$ in the derivative calculation:
\begin{align*}
 \partial_u\Gamma_{u,v}&=\frac12+\frac{v-u}{4r_{u,v}},
 \qquad
 \partial_v\Gamma_{u,v}=\frac12+\frac{u-v}{4r_{u,v}},\qquad
 \frac{|u-v|}{4r_{u,v}}\leq2^{-4}.
\end{align*}
Both derivatives therefore lie between $1/4$ and $3/4$.
The thickness estimate and height comparisons in
Lemma~\ref{lem:separation} now separate these intervals
with the same spacings $2^6n_{Q_1}s$ and $2^4s$.

\smallskip
\noindent\emph{Endpoint integrals and conclusion.}
The containing region \eqref{eq:containing-region} is the interval
$\mathcal O_{j,Q}=[c_Q+R_j/3,c_Q+R_j]$.
For $x\in\mathcal O_{j,Q}$ and $v\in Q^R$,
$x-v\geq{\frac{13}{3}2^{-4}}R_j>R_j/4$.
Hence $|\partial_v\Phi_{R_j}(x;u,v)|=2|x-v|\geq R_j/2$,
which gives
$$
 \int_{Q^R}\one_{\mathcal G_j(x)}(u,v)\,dv
 \leq Cs.
$$
Integrating also over $u\in Q^L$ gives $CsR_j$.
These are precisely \eqref{eq:one-variable-coarea} and
\eqref{eq:one-scale} with $d=1$.
The label counting and the three cases in the proof of
Proposition~\ref{prop:moments} therefore give
$$
 \Av_\omega\Av_{a,b}\int_{\mathbb R}\mathcal N\approx{\rho},
 \qquad
 \Av_\omega\Av_{a,b}\int_{\mathbb R}\mathcal N^2\leq C({\rho}+{\rho}^2).
$$
Corollary~\ref{cor:union} and Lemma~\ref{lem:sign-selection}
then give one choice of labels, base points, and signs with
$|\{x\in \mathbb R:\B_{E_N}(f,g)(x)>\lambda_0\}|\geq c_U/36$,
where $E_N=\{R_1,\ldots,R_N\}$.
Together with \eqref{eq:H1-bound} and \eqref{eq:theta-choice},
this proves Theorem~\ref{thm:main} for $d=1$.

\section{The multilinear extension}
\label{sec:multilinear}

In this section we prove Theorem~\ref{rem:multilinear-endpoint}.
Fix $k\geq3$ and $d\geq1$. All constants in this section may depend
on $d$, $k$, and the fixed function $\phi$, but not on $N$.
We use the cube system and notation of Subsection~\ref{sec:cube-tree},
with the parameters redefined below. The weights
$L_r=\lfloor\sqrt r\rfloor$ and the label selection rule remain
unchanged. Thus Lemma~\ref{lem:local-limit} still gives
$c_j\approx j^{-1}$.
We construct $k$ functions. The first two are sums of atoms as in
the bilinear case. Each of the other $k-2$ functions has one atom in
every level-$0$ cube. 
We give the proof for $d\geq2$. The changes for $d=1$ are given in
Subsection~\ref{sec:multi-dimension-one}.

\subsection{Atoms and selected tuples}
\label{sec:multi-construction}

Let $h_0$, $\epsilon_0$ and $c_\ast$ be the constants of
Lemma~\ref{lem:J-positive} for this $k$.
Choose a dyadic integer
$\mathcal L\geq\max\{2^{12}k^2\sqrt d,N^{1/d}\}$.
For each $P\in\mathcal Q_0$, put
$$
\begin{gathered}
 \theta^k=\frac{\rho}{\log(N+1)},\quad
 m_P=\mathcal L^{(kd-1)(W-Z_P)},\quad
 n_P=\mathcal L^{(kd-1)(W+Z_P)},\quad
 s=\theta^kR_0\mathcal L^{-2(kd-1)W}.
\end{gathered}
$$
With this $\mathcal L$, put $R_j=\mathcal L^{j-N}$ and
$E_N=\{R_1,\ldots,R_N\}$, as in Subsection~\ref{sec:cube-tree}.
Here $\rho>0$ is fixed so that
$$
 \frac{\rho}{\log2}\leq\min\Bigl\{\frac1{2^{12}k},\ \epsilon_0\Bigr\}.
$$
Since $s\leq\theta^kR_0\leq\theta^kR_j$ and
$\theta^k\leq\rho/\log2$, this gives
\begin{equation}\label{eq:multi-small-scale}
 \frac{s}{R_j}\leq\theta^k\leq\min\Bigl\{\frac1{2^{12}k},\ \epsilon_0\Bigr\}
 \qquad(1\leq j\leq N).
\end{equation}
For $P\in\mathcal Q_0$, use the coefficients
$\mu_P=\theta R_0^d/m_P$, $\nu_P=\theta R_0^d/n_P$ and define
$F_{P,r},G_{P,t}$ by \eqref{eq:atom-blocks}, with centers
$$
 u_{P,r}=a_P+2^8k n_Prs\,e_d,\qquad
 v_{P,t}=b_P+2^4k ts\,e_d.
$$
For $3\leq i\leq k$, choose base points
$w_P^{(i)}\in\frac18P$ independently of each other and of $a,b$,
and put
$$
 H_P^{(i)}(x)=\theta R_0^d s^{-d}
                 \phi\bigl((x-w_P^{(i)})/s\bigr).
$$
Since $m_Pn_Ps=\theta^kR_0$, the largest displacement of a center
is at most $2^8k\theta^kR_0\leq R_0/16$, by \eqref{eq:multi-small-scale}.
The same coordinate estimate as in Subsection~\ref{sec:atoms-functions}
therefore shows that all centers and support cubes lie in $P$.
Each family has total coefficient
$\theta$. Let $f,g$ be the signed sums of the first two families,
and let
$$
 h_i=\sum_{P\in\mathcal Q_0}\varepsilon_P^{(i)}H_P^{(i)},
 \qquad \varepsilon_P^{(i)}\in\{-1,1\},\quad 3\leq i\leq k.
$$
The atomic norm bound gives
\begin{equation}\label{eq:multi-Hardy}
 \|f\|_{H^1(\mathbb R^d)},\ \|g\|_{H^1(\mathbb R^d)},\
 \|h_3\|_{H^1(\mathbb R^d)},\ldots,\|h_k\|_{H^1(\mathbb R^d)}\leq C\theta
\end{equation}
for every choice of labels, base points, and signs.
We write $\Av_{a,b,w}$ for normalized integration in all base points.

Put $\eta=(2^8k)^{-1}$. For $Q\in\mathcal Q_j$, choose cubes
$Q^1,\ldots,Q^k$ of side $\eta R_j/(4k\sqrt d)$, centered at
$$
 c_Q+\left(i-\frac{k+1}{2}\right)\frac{\eta R_j}{k}e_1,
 \qquad 1\leq i\leq k.
$$
The side length is at least $4R_0$, so each auxiliary cube contains
$\approx\mathcal L^{dj}$ level-$0$ cubes. The gaps in the first
coordinate exceed $R_j/\mathcal L$. Hence no level-$(j-1)$ cube
meets two auxiliary cubes. These assertions follow from the chosen
lower bound on $\mathcal L$, exactly as in
Lemma~\ref{lem:first-common-ancestor} and the count following it.
Select $(Q_1,\ldots,Q_k)$ if
$$
 Q_i\subset Q^i\quad(1\leq i\leq k),\qquad Z_{Q_1}-Z_{Q_2}=j.
$$
There is no condition on the labels at the remaining positions.
The first two cubes have first common ancestor $Q$, so their
selection condition has label average $c_j$.
Put $r_j=\mathcal L^{(kd-1)(2W-j)}$. On that condition,
\begin{equation}\label{eq:multi-count-identities}
 \begin{gathered}
 m_{Q_1}n_{Q_2}=r_j,\qquad
 \mathcal L^{kdj}r_j\frac{s}{R_j}=\theta^k,\qquad
 \mu_{Q_1}\nu_{Q_2}(\theta R_0^d)^{k-2}=sR_j^{kd-1}.
 \end{gathered}
\end{equation}
Let $\mathcal P_{j,Q}$ be the set of selected tuples
$(Q_1,\ldots,Q_k)$, and let
$$
 \mathcal E_{j,Q}
 =\bigl\{((Q_1,r),(Q_2,t),Q_3,\ldots,Q_k):
 (Q_1,\ldots,Q_k)\in\mathcal P_{j,Q},\
 0\leq r<m_{Q_1},\ 0\leq t<n_{Q_2}\bigr\}.
$$
As in Section~\ref{sec:atom-pairs}, we write $e$ for the elements of
$\mathcal E_{j,Q}$.

\subsection{Regions and local lower bounds}

For $e\in\mathcal E_{j,Q}$, write
$u_1=u_{Q_1,r}$, $u_2=v_{Q_2,t}$ and $u_i=w_{Q_i}^{(i)}$ for
$i\geq3$, and put
$$
 z_e=\frac1k\sum_{i=1}^k u_i,\qquad
 r_e^2=\frac{R_j^2}{k}-\frac1k\sum_{i=1}^k|u_i-z_e|^2.
$$
Define
$$
 S_e=\left\{z_e+\lambda\zeta:
  \lambda>0,\quad\zeta\in\mathbb S^{d-1},\quad|\zeta-e_d|\leq\eta,\quad
  |\lambda^2-r_e^2|\leq\frac{h_0}{k}sR_j\right\}.
$$
The auxiliary cubes give
$|u_i-c_Q|\leq\eta R_j$ and $|u_i-z_e|\leq2\eta R_j$.
By \eqref{eq:multi-small-scale}, $s/R_j\leq1/(64k)$, and the defining
radial interval satisfies
$$
 \frac{3R_j}{4\sqrt k}\leq\lambda\leq\frac{5R_j}{4\sqrt k}.
$$
Indeed, $R_j^2/k-4\eta^2R_j^2\leq r_e^2\leq R_j^2/k$,
and the perturbation of its square is at most $sR_j/k$.
Polar integration as in \eqref{eq:region-size-and-lower-bound}
therefore gives $|S_e|\approx sR_j^{d-1}$. For $x\in S_e$,
$y_i=(x-u_i)/R_j$ satisfies
\begin{equation}\label{eq:multi-normals}
 \frac1{2\sqrt k}\leq|y_i|\leq\frac2{\sqrt k}.
\end{equation}

Set $\Phi_R^{(k)}(x;u_1,\ldots,u_k)=\sum_i|x-u_i|^2-R^2$.
At $R=R_j$ this is $k(|x-z_e|^2-r_e^2)$, so
\begin{equation}\label{eq:multi-strip}
 |\Phi_{R_j}^{(k)}(x;u_1,\ldots,u_k)|\leq h_0sR_j
 \qquad(x\in S_e).
\end{equation}

\begin{proposition}
There is a constant $c_0>0$, independent of $N$, $j$, $e$, the labels,
and the base points, such that for every $e\in\mathcal E_{j,Q}$,
\begin{equation}\label{eq:multi-local-lower}
 \A_{R_j}^{(k)}(F_{Q_1,r},G_{Q_2,t},H_{Q_3}^{(3)},\ldots,H_{Q_k}^{(k)})(x)
 \geq c_0\qquad(x\in S_e).
\end{equation}
\end{proposition}

\begin{proof}
The change of variables in \eqref{eq:exact-coarea}, now in $kd$
variables, gives
$$
 \A_{R_j}^{(k)}(F_{Q_1,r},G_{Q_2,t},H_{Q_3}^{(3)},\ldots,H_{Q_k}^{(k)})(x)
 =\frac{\mu_{Q_1}\nu_{Q_2}(\theta R_0^d)^{k-2}}
 {\mathcal H^{kd-1}(\mathbb S^{kd-1})sR_j^{kd-1}}
 J_\epsilon^{(k)}(y,h),
$$
where $y=(y_1,\ldots,y_k)$, $\epsilon=s/R_j$,
$h=\Phi_{R_j}^{(k)}(x;u_1,\ldots,u_k)/(2sR_j)$, and $J^{(k)}_\epsilon$
is defined in \eqref{eq:coarea-integral-definition}.
By \eqref{eq:multi-normals}, $y\in K_k$. By \eqref{eq:multi-strip}
and \eqref{eq:multi-small-scale}, $|h|\leq h_0/2$ and
$0<\epsilon\leq\epsilon_0$. Lemma~\ref{lem:J-positive} therefore gives
$J^{(k)}_\epsilon(y,h)\geq c_\ast$. The last identity in
\eqref{eq:multi-count-identities} proves
\eqref{eq:multi-local-lower} with
$c_0=c_\ast/\mathcal H^{kd-1}(\mathbb S^{kd-1})$.
\end{proof}

\subsection{Containing regions, separation, and endpoint integrals}

In this section, we replace the containing regions \eqref{eq:containing-region} by
$$
 \mathcal O_{j,Q}
 =\left\{x:|x'-c_Q'|\leq4\eta R_j,\quad
 \frac{R_j}{2\sqrt k}\leq x_d-c_{Q,d}\leq\frac{2R_j}{\sqrt k}\right\}.
$$
The bounds on $z_e,\lambda$ and $\zeta$ above imply $S_e\subset\mathcal O_{j,Q}$.
Their horizontal widths in each coordinate are at most $8\eta R_j<R_j$,
and their vertical length is $3R_j/(2\sqrt k)<R_j$ for $k\geq3$.
Since cube centers lie on the $R_j$ lattice, the regions are disjoint
at each level. Also $|\mathcal O_{j,Q}|\approx R_j^d$ and
$\sum_Q|\mathcal O_{j,Q}|\lesssim1$.

For $x\in\mathbb R^d$, put
$
 \mathcal G^{(k)}_j(x)
 =\{(u_1,\ldots,u_k)\in(\mathbb R^d)^k:
 |\Phi_{R_j}^{(k)}(x;u_1,\ldots,u_k)|\leq h_0sR_j\},
$
the analogue of $\mathcal G_j(x)$.

\begin{lemma}\label{lem:multi-endpoint}
Let $Q\in\mathcal Q_j$, $x\in\mathcal O_{j,Q}$, and let
$I\subset\{1,\ldots,k\}$ have $p\geq1$ elements. Fix $u_i\in Q^i$ for
$i\notin I$. Then
$$
 \int_{\prod_{i\in I}Q^i}
 \one_{\mathcal G^{(k)}_j(x)}(u_1,\ldots,u_k)
 \prod_{i\in I}du_i
 \leq CsR_j^{pd-1}.
$$
\end{lemma}

\begin{proof}
For any endpoint $u_i\in Q^i$ and $x\in\mathcal O_{j,Q}$,
$$
 x_d-u_{i,d}\geq\frac{R_j}{2\sqrt k}-\eta R_j
 \geq\frac{R_j}{4\sqrt k}.
$$
Thus $\partial_{u_{i,d}}\Phi_{R_j}^{(k)}=-2(x_d-u_{i,d})$ has
constant sign and absolute value at least $R_j/(2\sqrt k)$.
Fix one $i\in I$. Integrating first in $u_{i,d}$, and then in its
other coordinates, gives, uniformly in all other endpoints,
\begin{equation}\label{eq:multi-one-variable}
 \int_{Q^i}\one_{\mathcal G^{(k)}_j(x)}(u)\,du_i
 \leq C_{d,k}sR_j^{d-1}.
\end{equation}
Integrating in the remaining $p-1$ endpoints of $I$ contributes at
most $CR_j^{(p-1)d}$. This proves the lemma.
\end{proof}

\begin{lemma}\label{lem:multi-separation}
Let $e,e'\in\mathcal E_{j,Q}$ have the same cube tuple and index pairs
$(r,t)\neq(r',t')$. Then $S_e\cap S_{e'}=\varnothing$.
\end{lemma}

\begin{proof}
For a fixed cube tuple, the upper sphere is
$
 \Gamma(x')=z_{e,d}+H(x')$ with $H(x'):=(r_e^2-|x'-z_e'|^2)^{1/2}.
$
On $|x'-c_Q'|\leq4\eta R_j$, we have
$H^2\geq R_j^2/k-29\eta^2R_j^2\geq R_j^2/(4k)$.
Differentiating while holding the other endpoints fixed gives
$$
 \partial_{u_{i,d}}\Gamma
 =\frac1k+\frac{z_{e,d}-u_{i,d}}{kH},\qquad
 \frac1{2k}\leq\partial_{u_{i,d}}\Gamma\leq\frac2k.
$$
These bounds hold throughout the line segments joining the atom
centers. Moreover,
$$
 |x_d-\Gamma(x')|
 =\frac{|\Phi_{R_j}^{(k)}(x;u)|}{k(x_d-z_{e,d}+H)}
 \leq\frac{2h_0}{\sqrt k}s\leq s\qquad(x\in S_e).
$$
If $r=r'$ and $t\ne t'$, the height difference is at least $8s$.
If $r>r'$, then $n_{Q_2}\leq n_{Q_1}$ by selection, and the height
difference is at least
$$
 128n_{Q_1}(r-r')s-32|t-t'|s\geq96n_{Q_1}s>2s.
$$
At each $x'$, the vertical sections of $S_e$ and $S_{e'}$ therefore
lie in disjoint intervals.
\end{proof}

\subsection{The counting functions}

Define $\mathcal N_{j,Q},\mathcal N_j,\mathcal N$ by summing the
indicators of $S_e$, $e\in\mathcal E_{j,Q}$, as in
\eqref{eq:cube-counting-functions} and \eqref{eq:counting-functions}.

\begin{proposition}\label{prop:multi-moments}
There are constants $c,C>0$, independent of $N$, such that
$$
 \Av_\omega\Av_{a,b,w}\int\mathcal N\geq c\rho,\qquad
 \Av_\omega\Av_{a,b,w}\int\mathcal N^2\leq C(\rho+\rho^2).
$$
\end{proposition}

\begin{proof}
The number of cube tuples is $\approx\mathcal L^{kdj}$. On selection
there are $r_j$ index pairs and the label average is $c_j$.
Consequently
\begin{equation}\label{eq:multi-first-moment}
 \Av_\omega\Av_{a,b,w}\int\mathcal N_{j,Q}
 \approx c_j\mathcal L^{kdj}r_jsR_j^{d-1}
 =c_j\theta^kR_j^d.
\end{equation}
Summing over $Q\in\mathcal Q_j$ and $1\leq j\leq N$ gives
$\Av_\omega\Av_{a,b,w}\int\mathcal N\approx\theta^k\sum_jc_j\approx\rho$.

\smallskip
\emph{The same level and the same cube.}
Fix $j$ and $Q\in\mathcal Q_j$. We classify ordered pairs
$(e,\widetilde e)\in\mathcal E_{j,Q}^2$, with cube tuples
$(Q_1,\ldots,Q_k)$ and $(\widetilde Q_1,\ldots,\widetilde Q_k)$, by
the set
$$
 A=\{i:Q_i=\widetilde Q_i\}
$$
of shared positions. Different positions cannot share a cube, because
the auxiliary cubes are disjoint. For $e\in\mathcal E_{j,Q}$ and
$A\subset\{1,\ldots,k\}$, let $\mathcal E_{j,Q}(e,A)$ be the set of
$\widetilde e\in\mathcal E_{j,Q}$ whose set of shared positions with
$e$ is $A$. Let $q=\#A$ and $p=k-q$.

If $q=k$, Lemma~\ref{lem:multi-separation} leaves only $e=\widetilde e$.
This contribution equals
$\Av_\omega\Av_{a,b,w}\int\mathcal N_{j,Q}$.
By \eqref{eq:multi-first-moment}, it is bounded by
$Cc_j\theta^kR_j^d$.

Suppose $1\leq q<k$, and fix $A$. Fix the labels, the first pair $e$,
and all base points of $e$. Then $S_e$ is fixed. The free positions
$i\notin A$ carry the base points of the cubes $\widetilde Q_i\ne Q_i$,
which are distinct from those of $e$. Write $\Av_{\mathrm{free}}$ for
the normalized integration in these $p$ base points. It contributes
the factor $8^{pd}R_0^{-pd}$. Translate each such base
point to its atom center. The Jacobian is one and, by the displacement
bound in Subsection~\ref{sec:multi-construction}, the translated domain
lies in $\widetilde Q_i$. Since $S_{\widetilde e}$ lies in the strip
\eqref{eq:multi-strip}, for $x\in S_e\subset\mathcal O_{j,Q}$ we obtain
\begin{equation}\label{eq:multi-free-bound}
 \Av_{\mathrm{free}}\one_{S_{\widetilde e}}(x)
 \leq CR_0^{-pd}\int_{\prod_{i\notin A}\widetilde Q_i}
 \one_{\mathcal G^{(k)}_j(x)}(\widetilde u)
 \prod_{i\notin A}d\widetilde u_i,
\end{equation}
where $\widetilde u_i$, $i\in A$, are the centers of $\widetilde e$ at the
shared positions. These shared centers depend on
$(\widetilde r,\widetilde t)$ only, not on the free cubes.

We next sum over $\widetilde e\in\mathcal E_{j,Q}(e,A)$. Two cases occur.

\emph{(a) The first or the second position is shared.}
If $1\in A$, the two selection conditions
$Z_{Q_1}-Z_{Q_2}=j$ and $Z_{Q_1}-Z_{\widetilde Q_2}=j$ give
$Z_{\widetilde Q_2}=Z_{Q_2}$. Hence
$m_{\widetilde Q_1}=m_{Q_1}$ and $n_{\widetilde Q_2}=n_{Q_2}$ for every
permitted partner, and the index range of $(\widetilde r,\widetilde t)$
is the same fixed set of $r_j$ pairs. If $2\in A$, the same argument
gives $Z_{\widetilde Q_1}=Z_{Q_1}$, with the same conclusion. Fix
$(\widetilde r,\widetilde t)$ in this range. The right-hand side of
\eqref{eq:multi-free-bound} is then the integral of one fixed function
over $\prod_{i\notin A}\widetilde Q_i$. Summing over the free cubes,
which are disjoint, and dropping the remaining selection restrictions
bounds the sum by the same integral over $\prod_{i\notin A}Q^i$.
Lemma~\ref{lem:multi-endpoint} with $I=\{1,\ldots,k\}\setminus A$
bounds this integral by $CsR_j^{pd-1}$. Summing finally over the $r_j$
pairs $(\widetilde r,\widetilde t)$ gives
$$
 \sum_{\widetilde e\in\mathcal E_{j,Q}(e,A)}
 \Av_{\mathrm{free}}\one_{S_{\widetilde e}}(x)
 \leq Cr_jR_0^{-pd}sR_j^{pd-1}.
$$

\emph{(b) Neither the first nor the second position is shared.}
Then $A\subset\{3,\ldots,k\}$, so the shared centers are the single
atoms at the positions in $A$ and do not depend on
$(\widetilde r,\widetilde t)$. Both $\widetilde Q_1$ and $\widetilde Q_2$
are free. After enlargement, the right-hand side of
\eqref{eq:multi-free-bound} does not depend on
$(\widetilde r,\widetilde t)$, and on every selected tuple there are
exactly $m_{\widetilde Q_1}n_{\widetilde Q_2}=r_j$ such pairs. We
therefore sum first over $(\widetilde r,\widetilde t)$ and then over the
free cubes. As in (a), this gives the same bound
$Cr_jR_0^{-pd}sR_j^{pd-1}$.

In both cases, \eqref{eq:multi-count-identities} gives
$r_jR_0^{-pd}sR_j^{pd-1}=\theta^k\mathcal L^{-qdj}$. Integrating over
$x\in S_e$, we obtain
$$
 \sum_{\widetilde e\in\mathcal E_{j,Q}(e,A)}
 \Av_{\mathrm{free}}|S_e\cap S_{\widetilde e}|
 \leq C\theta^k\mathcal L^{-qdj}|S_e|.
$$
Average the base points of $e$, sum over $e$, and average labels.
By \eqref{eq:multi-first-moment}, the contribution is bounded by
$C\theta^k\mathcal L^{-qdj}c_j\theta^kR_j^d
\leq Cc_j\theta^kR_j^d$. There are at most $2^k$ choices of $A$,
a factor absorbed into $C$.

It remains to consider $q=0$, that is, tuples with no shared cubes.
For fixed labels and indices, the two tuples have disjoint base-point
coordinates, so the base-point average factors. Translating and
enlarging the domains for both tuples, as in
\eqref{eq:four-base-point-bound}, gives for $x\in\mathcal O_{j,Q}$ the
bound
$$
 CR_0^{-2kd}
 \Bigl(\int_{\prod_iQ_i}\one_{\mathcal G^{(k)}_j(x)}(u)\,du\Bigr)
 \Bigl(\int_{\prod_i\widetilde Q_i}\one_{\mathcal G^{(k)}_j(x)}(\widetilde u)\,d\widetilde u\Bigr),
$$
which is independent of the labels and indices. The index sum
contributes $r_j^2$.

If $K(Q_1,Q_2)\cap K(\widetilde Q_1,\widetilde Q_2)=\varnothing$, the
label average is $c_j^2$ by \eqref{eq:cross-level-label-count}.
Combining the disjoint domains and applying
Lemma~\ref{lem:multi-endpoint} with $I=\{1,\ldots,k\}$ to each factor
gives
$$
 Cc_j^2r_j^2R_0^{-2kd}s^2R_j^{2kd-2}|\mathcal O_{j,Q}|
 \leq Cc_j^2\theta^{2k}R_j^d.
$$

Otherwise, choose $P$ in the intersection. As in the argument
following \eqref{part iii 1}, $P=Q_i^{(h)}=\widetilde Q_\ell^{(h)}$ for
some $0\leq h\leq j-1$ and $i,\ell\in\{1,2\}$, since no level-$(j-1)$
cube meets two auxiliary cubes, $i=\ell$. Hence
$$
 Q_1^{(j-1)}=\widetilde Q_1^{(j-1)}
 \quad\text{or}\quad
 Q_2^{(j-1)}=\widetilde Q_2^{(j-1)}.
$$
In either case, the joint label average is at most $c_j$, as in the
bilinear case. Suppose $Q_1^{(j-1)}=\widetilde Q_1^{(j-1)}=Q'$. The
other case is the same with the first and second positions
interchanged. In each of the two integrals above, integrate first in
the third endpoint by \eqref{eq:multi-one-variable}. The third
positions are free because no cube is shared. This gives the factor
$(CsR_j^{d-1})^2$. 
For each tuple, integrate the endpoints at all positions
$i\in\{1,\ldots,k\}\setminus\{1,3\}$ by volume. Together, these
integrals contribute at most $CR_j^{(2k-4)d}$. The first endpoints
$u_1,\widetilde u_1$ remain.
Grouping the disjoint domains $Q_1\times\widetilde Q_1$ by their common
ancestor $Q'$ and enlarging their union, we obtain the factor
$$
 \sum_{Q'\in\mathcal Q_{j-1}:\,Q'\subset Q}|Q'|^2
 =\mathcal L^{-d}R_j^{2d}.
$$
Thus the two integrals and the cube sum together contribute at most
$C\mathcal L^{-d}s^2R_j^{2kd-2}$. Including normalization, index count, label average, and the measure
of $\mathcal O_{j,Q}$ bounds the exceptional contribution by
$$
 Cc_jr_j^2R_0^{-2kd}\mathcal L^{-d}s^2R_j^{2kd-2}R_j^d
 =\frac{Cc_j}{\mathcal L^d}\theta^{2k}R_j^d
 \leq Cc_j^2\theta^{2k}R_j^d.
$$
Here $\mathcal L^d\geq N\geq j$ and $c_j\gtrsim j^{-1}$ give the
last inequality. Combining all cases yields
\begin{equation}\label{eq:multi-same-level}
 \Av_\omega\Av_{a,b,w}\int\mathcal N_{j,Q}^2
 \leq C(c_j\theta^k+c_j^2\theta^{2k})R_j^d.
\end{equation}

\smallskip
\emph{The same level and different cubes.}
Different cubes at the same level have disjoint containing regions,
so they contribute no cross terms.

\smallskip
\emph{Different levels.}
Let $j<j'$, with the smaller region belonging to $Q\in\mathcal Q_j$
and the larger one to $\widetilde Q\in\mathcal Q_{j'}$.
If an endpoint cube of the larger tuple lies in $Q$, its center $w$
and the smaller tuple's center $u_1$ both lie in $Q$. An intersection
point would satisfy
$$
 \frac{R_{j'}}{2\sqrt k}\leq|x-w|
 \leq|x-u_1|+|u_1-w|
 \leq\left(\frac2{\sqrt k}+\sqrt d\right)R_j,
$$
contradicting $R_{j'}/R_j\geq\mathcal L$.
In all remaining terms the tuples have disjoint endpoint cubes.
Their selection conditions also have disjoint label-coordinate sets:
a common coordinate cube would be strictly inside $Q$ and would
contain an endpoint cube of the larger tuple, a case just excluded.
For fixed labels the base-point integrals factor. Enlarge their domains
before summing indices and labels, giving $r_jr_{j'}$ and $c_jc_{j'}$,
respectively. Lemma~\ref{lem:multi-endpoint} with $I=\{1,\ldots,k\}$ at
each radius gives
$$
 \begin{aligned}
 \Av_\omega\Av_{a,b,w}
       [\mathcal N_{j,Q}(x)\mathcal N_{j',\widetilde Q}(x)]
 &\leq Cc_jc_{j'}r_jr_{j'}R_0^{-2kd}s^2R_j^{kd-1}R_{j'}^{kd-1}
       \one_{\mathcal O_{j,Q}}(x)\one_{\mathcal O_{j',\widetilde Q}}(x)\\
 &=Cc_jc_{j'}\theta^{2k}
       \one_{\mathcal O_{j,Q}}(x)\one_{\mathcal O_{j',\widetilde Q}}(x).
 \end{aligned}
$$
Summing cubes and integrating, using disjointness at each level,
yields $$\Av_\omega\Av_{a,b,w}\int\mathcal N_j\mathcal N_{j'}
\leq Cc_jc_{j'}\theta^{2k}.$$

\smallskip
\emph{Conclusion.}
Together with \eqref{eq:multi-first-moment} and
\eqref{eq:multi-same-level}, these bounds give
$$
 \Av_\omega\Av_{a,b,w}\int\mathcal N^2
 \leq C\Bigl[\theta^k\sum_{j=1}^Nc_j
 +\theta^{2k}\Bigl(\sum_{j=1}^Nc_j\Bigr)^2\Bigr]
 \leq C(\rho+\rho^2),
$$
since $\theta^k\sum_{j=1}^Nc_j\approx\rho$.
\end{proof}

As in Corollary~\ref{cor:union}, Proposition~\ref{prop:multi-moments}
gives a fixed choice of all labels and base points for which
$|U|\geq c_U$, where $U=\{\mathcal N>0\}$ and $c_U>0$ is independent
of $N$.

\subsection{Signs and the norm comparison}

At each $x\in U$, choose a selected region before choosing the signs.
The full average at its radius is a real sum of the form
$$
 X=\sum_{\alpha_1,\ldots,\alpha_k}
       a_{\alpha_1,\ldots,\alpha_k}
       \varepsilon^{(1)}_{\alpha_1}\cdots\varepsilon^{(k)}_{\alpha_k},
$$
where $\varepsilon^{(1)},\ldots,\varepsilon^{(k)}$ are the independent
sign families of the $k$ functions.

\begin{lemma}
For every such sum, $\Av X^2=\sum_\alpha a_\alpha^2$ and
$$
 \Av X^4\leq3^k\Bigl(\sum_\alpha a_\alpha^2\Bigr)^2.
$$
\end{lemma}

\begin{proof}
The identity for $\Av X^2$ follows from orthogonality. Write $X=\sum_i\varepsilon_i^{(1)}X_i$ and average the first
sign family to obtain $\Av X^4\leq3\sum_{i,i'}\Av(X_i^2X_{i'}^2)$.
Cauchy--Schwarz and induction on the other $k-1$ sign families
bound this by $3^k(\sum_\alpha a_\alpha^2)^2$. The base case is
the one-family identity used in Lemma~\ref{lem:sign-estimates}.
\end{proof}

By \eqref{eq:multi-local-lower}, $\Av X^2\geq c_0^2$. The proof of
Lemma~\ref{lem:sign-selection}, with $3$ replaced by $3^{k/2}$, gives
$$
 \Av\one_{\{|X|>\frac{c_0}{\sqrt2}\}}\geq\frac1{4\cdot3^k}.
$$
Integrating over $U$ selects one common choice of signs such that
$$
 \Big|\Big\{x:\B_{E_N}^{(k)}(f,g,h_3,\ldots,h_k)(x)>\frac{c_0}{\sqrt2}\Big\}\Big|
 \geq\frac{c_U}{4\cdot3^k}.
$$
In particular, all functions are nonzero.
Using \eqref{eq:multi-Hardy} and \eqref{eq:weak-quasinorm} with $p=1/k$
gives
$$
 \frac{\|\B_{E_N}^{(k)}(f,g,h_3,\ldots,h_k)\|_{L^{1/k,\infty}(\mathbb R^d)}}
 {\|f\|_{H^1(\mathbb R^d)}\|g\|_{H^1(\mathbb R^d)}\prod_{i=3}^k\|h_i\|_{H^1(\mathbb R^d)}}
 \geq\frac{c_0}{\sqrt2 C^k\theta^k}\Bigl(\frac{c_U} {4\cdot 3^k}\Bigr)^k
 =c_{d,k}\log(N+1),
$$
with $c_{d,k}=\frac{c_0}{\sqrt2}(c_U/(4\cdot  3^k))^k/(C^k\rho)$.
This proves
Theorem~\ref{rem:multilinear-endpoint} for $d\geq2$.

\subsection{The case \texorpdfstring{$d=1$}{d=1}}
\label{sec:multi-dimension-one}

The same proof applies when $d=1$, with the simplifications of
Section~\ref{sec:dimension-one}. The auxiliary cubes are intervals,
the cap $|\zeta-e_d|\leq\eta$ becomes $\{+1\}\subset\mathbb S^0$, and
$S_e$ is an interval of length comparable to $s$. There is no
horizontal variable, and the upper sphere is the point
$\Gamma=z_e+r_e$, with $H=r_e$. The bounds for
$\partial_{u_{i}}\Gamma$ in Lemma~\ref{lem:multi-separation} used only
$|z_e-u_i|\leq2\eta R_j$ and $H\geq R_j/(2\sqrt k)$, so they hold
without change. The containing region is the interval
$[c_Q+R_j/(2\sqrt k),\,c_Q+2R_j/\sqrt k]$, the integral in one endpoint
is bounded by $Cs$, and $\sum_{Q'}|Q'|^2=\mathcal L^{-1}R_j^2$. Since
$\mathcal L\geq N$, the exceptional terms are still bounded by
$Cc_j^2\theta^{2k}R_j$. Lemma~\ref{lem:J-positive} holds for every
$d\geq1$. This proves Theorem~\ref{rem:multilinear-endpoint}.

\section{Applications}
\label{sec:applications}

Our construction also gives the following consequences.  

\begin{remark}
In this
remark, we write $\A_R^{(2)}=\A_R$ and $\B_E^{(2)}=\B_E$.
First, let $d\geq1$, $k\geq2$, $a\in\mathbb Z$, and $h\in\mathbb N$.
For
$
 E_{a,h,N}=\{2^{a+hn}:0\leq n<N\},
$
there is a constant $c_{d,k}>0$, independent of $a,h,N$, such that
\[
 \sup_{f_1,\ldots,f_k\ne0}
 \frac{\|\B_{E_{a,h,N}}^{(k)}(f_1,\ldots,f_k)
       \|_{L^{1/k,\infty}(\mathbb R^d)}}
      {\prod_{i=1}^k\|f_i\|_{H^1(\mathbb R^d)}}
 \geq c_{d,k}\log(N+1).
\]
To see this, take $M=\lfloor\sqrt N\rfloor$ and
$
 m=\left\lceil\log_2\max\{C_{d,k},M^{1/d}\}\right\rceil,
$
where $C_{d,k}$ is large enough for the geometric restrictions in the
construction. Then $\mathcal L=2^{hm}$ satisfies its requirements with
$M$ in place of $N$. For sufficiently large $N$,
$m(M-1)\leq N-1$. After a dyadic dilation, the $M$ radii used in the
construction belong to $E_{a,h,N}$. The result follows since
$\log(M+1)\gtrsim\log(N+1)$. For the remaining values of $N$, a fixed
nonzero smooth mean-zero test function gives a positive lower bound at
one radius. Dyadic dilation makes this bound independent of $a$; after
decreasing $c_{d,k}$, it also gives the displayed estimate.

Second, there are fixed functions
$f_1,\ldots,f_k\in C^\infty(\mathbb R^d)\cap H^1(\mathbb R^d)$ such
that, for every $\lambda>0$,
\[
 \bigl|\{x:\B_{\mathrm{lac}}^{(k)}(f_1,\ldots,f_k)(x)
       >\lambda\}\bigr|=\infty.
\]
Indeed, the proof gives functions $f_{i,N}$ supported in a fixed
bounded set, with $\|f_{i,N}\|_{H^1}\lesssim\theta_N$, and a set of
measure bounded below independently of $N$ on which
$\B_{E_N}^{(k)}(f_{1,N},\ldots,f_{k,N})$ has a fixed positive lower
bound. These sets also lie in a fixed bounded set, and
$E_N\subset(0,1]$. Choose $N_n$ so that
$2^{-kn}\theta_{N_n}^{-k}\to\infty$, and set
\[
 f_i(x)=\sum_{n=1}^{\infty}
  2^{-n}\theta_{N_n}^{-1}f_{i,N_n}(x-z_n),
\]
where the points $z_n$ are chosen so that the translated lower-bound
sets are disjoint and each has distance greater than $1$ from the
supports of the other summands. The sums converge in $H^1$ and are
locally finite, so each $f_i$ is smooth. On the $n$th translated set,
the other summands do not contribute at the radii in $E_{N_n}$. The
maximal function is therefore bounded below there by a positive
constant times $2^{-kn}\theta_{N_n}^{-k}$. Since each set has measure
bounded below independently of $n$, the asserted conclusion follows.

Third, the $L^1$ weak endpoint also fails for nonnegative smooth
compactly supported functions. In fact, the positivity of the
spherical measure gives
\[
 |\A_R^{(k)}(f_1,\ldots,f_k)|
 \leq \A_R^{(k)}(|f_1|,\ldots,|f_k|).
\]
Choose a nonnegative $\psi\in C_c^\infty(\mathbb R^d)$ with
$\psi\geq|\phi|$. Replace $\phi$ by $\psi$ in each atom block and
sum with the absolute values of the coefficients. The resulting
nonnegative smooth compactly supported functions dominate the
absolute values of the original functions and have $L^1$ norms
$O(\theta)$. Thus the same logarithmic lower bound holds for them.

The same construction gives failure of the weak Hardy endpoint for
the dyadic $q$-variation, for every $q>0$. We define
\[
 V^q(\{a_\ell\}_{\ell\in\mathbb Z})
 =\sup_{\ell_0<\cdots<\ell_m}
 \bigg(\sum_{j=1}^m|a_{\ell_j}-a_{\ell_{j-1}}|^q\bigg)^{1/q},
\]
where the supremum is over all finite increasing sequences of integers.
For compactly supported functions,
$\A_{2^\ell}^{(k)}(f_1,\ldots,f_k)(x)=0$ for all sufficiently large
$\ell$, depending on $x$. Hence
$
 V^q\bigl(\{\A_{2^\ell}^{(k)}(f_1,\ldots,f_k)\}_{\ell\in\mathbb Z}\bigr)
 \geq \B_{\mathrm{lac}}^{(k)}(f_1,\ldots,f_k).
$
For the finite-scale lower bound, add to $E_N$ a dyadic radius large
enough that its average vanishes on the fixed bounded lower-bound set.

Finally, the counterexamples can be placed in any nonempty open set
and can use arbitrarily small dyadic radii. This follows by a dyadic
dilation and translation. Also, for any fixed integer $M\geq1$, we
may replace the function $\phi$ by
$
 \phi_M=(-\Delta)^M(\chi*\widetilde\chi).
$
Then
\[
 \widehat{\phi_M}(\xi)
 =(4\pi^2|\xi|^2)^M|\widehat\chi(\xi)|^2\geq0,
 \qquad
 \int x^\alpha\phi_M(x)\,dx=0
 \quad\text{for }|\alpha|<2M.
\]
In the proof of Lemma~\ref{lem:J-positive}, the Fourier integral at
$\epsilon=h=0$ is positive because its integrand is nonnegative and
is bounded below by a positive constant times $|t|^{2Mk}$ near the
origin. The perturbation estimates apply to $\phi_M$ as well. Choose
$h_0$, $\epsilon_0$, and $\rho$ accordingly, allowing them to depend
on $M$. The rest of the construction is unchanged. Thus the atoms may
be chosen to have any prescribed finite number of vanishing moments.
\end{remark}

\begin{remark}
The construction also applies to some non-spherical hypersurfaces
and to spherical averages on the Heisenberg group.

First, there are smooth, compact, strictly convex hypersurfaces
$\Sigma\subset\mathbb R^{2d}$, which are not spheres, for which the
associated lacunary bilinear maximal operator fails to map
$H^1(\mathbb R^d)\times H^1(\mathbb R^d)$ into
$L^{1/2,\infty}(\mathbb R^d)$. For example, $\Sigma$ may agree
with $\mathbb S^{2d-1}$ near the spherical cap used in the
construction and be a small smooth perturbation of the sphere
elsewhere. The local lower bounds remain valid on that cap.
For related bounds for lacunary maximal operators over hypersurfaces,
see \cite{ChoLeeShuin}.

Second, let $\mathbb H^n=\mathbb C^n\times\mathbb R$, and let
$\mathfrak S_R$ be the bilinear Nevo--Thangavelu spherical average
studied by Ghosh and Singh \cite{GhoshSingh}. The sphere in this
average lies in the horizontal variables. The corresponding
lacunary maximal operator does not map
\[
 H^1(\mathbb H^n)\times H^1(\mathbb H^n)
 \longrightarrow L^{1/2,\infty}(\mathbb H^n),
 \qquad n\geq1.
\]
To obtain this example, we lift each Euclidean atom to
$\mathbb H^n$ using a cutoff in the central variable adapted to
the center of the atom. The lifted function is a sum of
$H^1(\mathbb H^n)$ atoms with uniformly bounded coefficient sum.
On a suitable set, its spherical averages agree with the Euclidean
averages used in the construction.

\end{remark}

\bigskip
\bigskip

\noindent \textbf{Acknowledgements.} S. Lee is supported by the National Research Foundation of Korea (NRF) through Grant No. RS-2024-00342160.  J. Li is supported by ARC DP260100485. Liangchuan Wu is supported by NNSF of China \#12201002.

\bigskip
\noindent (S. Lee) Department of Mathematical Sciences and RIM,
Seoul National University, Seoul 08826, Republic of Korea\par
\noindent\textit{E-mail}: \texttt{shklee@snu.ac.kr}

\medskip
\noindent (J. Li) School of Mathematical and Physical Sciences,
Macquarie University, NSW, 2109, Australia\par
\noindent\textit{E-mail}: \texttt{ji.li@mq.edu.au}

\medskip
\noindent (C. Wen) Department of Mathematics, Sun Yat-sen University,
Guangzhou, 510275, China\par
\noindent\textit{E-mail}: \texttt{wenchj@mail2.sysu.edu.cn}

\medskip
\noindent (L. Wu) School of Mathematical Sciences, Anhui University,
Hefei, 230601, P.R.~China\par
\noindent\textit{E-mail}: \texttt{wuliangchuan@ahu.edu.cn}

\end{document}